\documentclass[aos,preprint]{imsart}
\usepackage{amsmath,amstext,amssymb,amsfonts,amscd,color}
\RequirePackage[OT1]{fontenc}
\usepackage{graphicx}
\usepackage{array}
\usepackage{enumitem}
\usepackage[]{algorithm2e}

\usepackage{listings}
\usepackage[authoryear]{natbib}
\usepackage{amsmath}
\usepackage{amsthm}
\usepackage{amssymb}
\usepackage{amsfonts}
\usepackage{bbm}
\usepackage[margin = 1in]{geometry}

\usepackage{xparse}
\usepackage{tikz}
\usetikzlibrary{plotmarks}
\usetikzlibrary{matrix}
\usepackage{pgfplots}
\usetikzlibrary{matrix,backgrounds}
\pgfdeclarelayer{myback}
\pgfsetlayers{myback,background,main}
\usetikzlibrary{decorations.pathreplacing,angles,quotes}
\tikzset{mycolor/.style = {dashed,rounded corners,line width=1bp,color=#1}}%
\tikzset{myfillcolor/.style = {draw,fill=#1}}%

\newtheorem{example}{Example}[section]
\newtheorem{theorem}{Theorem}[section]
\newtheorem{lemma}[theorem]{Lemma}
\newtheorem{assumption}[theorem]{Assumption}

\newtheorem{proposition}[theorem]{Proposition}
\newtheorem{remark}[theorem]{Remark}
\usepackage{listings}
\usepackage{hyperref}
\usepackage[hypcap]{caption}
\usepackage{booktabs}
\usepackage{fancyvrb}
\usepackage{comment}
\usepackage{multirow}
\usepackage{multibib}

\usepackage{color}

\startlocaldefs
\newcommand{\change}[1]{{\color{black}{#1}}}
\newcommand{\csv}[1]{{\color{black}{#1}}}

\makeatletter
\def\BState{\State\hskip-\ALG@thistlm}
\newcommand*{\addFileDependency}[1]{
  \typeout{(#1)}
  \@addtofilelist{#1}
  \IfFileExists{#1}{}{\typeout{No file #1.}}
}

\makeatother

\numberwithin{equation}{section}

\newcommand{\E}{\mathbb{E}}

\newcommand{\cF}{\mathcal{F}}

\newcommand{\cD}{\mathcal{D}}
\newcommand{\Dkonv}{\stackrel{\cD}{\longrightarrow}}
\newcommand{\floor}[1]{\lfloor #1 \rfloor}
\newcommand{\1}{{\bf{1}}}

\newcommand{\R}{\mathbb{R}}
\newcommand{\N}{\mathbb{N}}
\newcommand{\ceil}[1]{\lceil#1 \rceil}
\newcommand{\weak}{\rightsquigarrow }

\def\cov{{\mbox{cov}}}
\def\var{{\mbox{var}}}

\endlocaldefs

\begin{document}

\begin{frontmatter}
	\title{\bf Inference for Change-Points in High-dimensional Data via Self-normalization }
	\runtitle{Inference for Change-Points in High-dimensional Data}
	
	\begin{aug}
		\author{\fnms{Runmin} \snm{Wang}\thanksref{t1,m0}\ead[label=e0]{runminw@smu.edu}},	
		\author{\fnms{Changbo} \snm{Zhu}\thanksref{t1,m4}\ead[label=e1]{cbzhu@ucdavis.edu}},
		
		\author{\fnms{Stanislav} \snm{ Volgushev}\thanksref{t1,m2}\ead[label=e2]{stanislav.volgushev@utoronto.ca}}
		\and
		\author{\fnms{Xiaofeng } \snm{Shao}\thanksref{t1,m1}
			\ead[label=e3]{xshao@illinois.edu}}
		
		\thankstext{t1}{Runmin Wang is Assistant Professor at Southern Methodist University, Department of Statistical Science (email: runminw@smu.edu); Changbo Zhu is Postdoctoral Scholar, Department of  Statistics, University of California at Davis; Stanislav Volgushev is Assistant Professor at Department of Statistical Sciences, University of Toronto (email: stanislav.volgushev@utoronto.ca); Xiaofeng Shao is  Professor, Department of  Statistics, University of Illinois at Urbana-Champaign (e-mail: xshao@illinois.edu). Wang and Zhu are joint first authors and made equal contributions to the paper.}
	
		\runauthor{R. Wang et al.}
		
		\affiliation{Southern Methodist University\thanksmark{m0} and University of California at Davis\thanksmark{m4} and University of Toronto\thanksmark{m2} and				University of Illinois at Urbana-Champaign\thanksmark{m1} }
		
	\end{aug}

\begin{abstract}
	This article considers change point testing and estimation for a sequence of high-dimensional data. In the case of testing for a mean shift for high-dimensional independent data, we propose a new test which is based on $U$-statistic in \cite{chen2010two} and utilizes the self-normalization principle [\cite{shao2010self}, \cite{shao2010testing}]. Our test targets dense alternatives in the high-dimensional setting and involves no tuning parameters. To extend to change point testing for high-dimensional time series, we introduce a trimming parameter and formulate a self-normalized test statistic with trimming to accommodate the weak temporal dependence. On the theory front, we 	derive the limiting distributions of self-normalized test statistics under both the null and alternatives for both independent and dependent high-dimensional data. 	At the core of our asymptotic theory,  we obtain weak convergence of a sequential U-statistic based process for high-dimensional independent data, and weak convergence of  sequential trimmed U-statistic based processes for high-dimensional linear processes, both of which are of independent interests.  Additionally, we illustrate how our tests  can be used in combination with wild binary segmentation to estimate the number and location of multiple change points. 
	Numerical simulations demonstrate the competitiveness of our proposed testing and estimation procedures in comparison with several existing methods in the literature. 
	
\end{abstract}


\end{frontmatter}

	\section{Introduction}

Suppose that  we have a sequence of  $\mathbb{R}^{p}$-valued observations $\{Y_t\}_{t = 1}^{n}$ which share the same distribution, except for possible change points in the mean vector $\mu_t=E(Y_t)$. We are interested in testing
\[
\mathcal{H}_0: \mu_1 = \mu_2 = \cdots= \mu_n \qquad v.s \qquad \mathcal{H}_1: \mu_1 = \cdots = \mu_{k_1}\neq \mu_{k_1+1} =\cdots = \mu_{k_s}\neq\mu_{k_s+1} \cdots = \mu_{n},
\]
for some unknown $s$ and $k_j$, \change{$j=1,...,s$}. Change point testing is a classical problem in statistics and econometrics and it has been extensively studied when the dimension $p$ is low and fixed. For univariate and low/fixed dimensional multivariate data, we refer the readers to \cite{aue2009break}, \cite{shao2010testing}, \cite{matteson2014nonparametric},  \cite{kirch2015detection}, \cite{zhang2018unsupervised} (among many others) for some recent work and \cite{perron2006dealing} and \cite{aue2013structural} for excellent reviews and the huge literature cited therein. A related problem is to estimate the number $s$ and the locations ($k_j$, \change{$j=1,...,s$}) of change points, which is also addressed in this paper.

Owing to the advances in science and technology, high-dimensional data is now produced in many areas, such as neuroscience, genomics and finance, among others. Structural change detection and estimation for high-dimensional data are of prime importance to understand the heterogeneity in the data as well as facilitate statistical modeling and inference. Among recent work that tackles
 change point testing and estimation for the mean of high-dimensional data and large panel data (allowing growing dimension), we mention \cite{horvath2012change}, \cite{chan2013darling},
 \citet{jirak2012change, jirak2015uniform}, \cite{cho2016change}, \cite{yu2017finite}, \cite{wang2018high}, \cite{dette2018relevant}, \cite{enikeeva2013high}. In the high-dimensional environment, we often classify the alternatives into two types: sparse and dense alternatives. \csv{In the change-point context, a sparse change means that only a few components of the vector change their mean, i.e. the $L^0$-norm of the mean change vector is much smaller than $p$; whereas dense change corresponds to the case that a change occurs for a substantial portion of the components. } 
  Several of the above-mentioned tests, including \cite{chan2013darling}, \citet{jirak2012change, jirak2015uniform}, \cite{yu2017finite}, \cite{dette2018relevant} and \cite{wang2018high}, specifically target sparse alternatives. For example, the test proposed by \cite{wang2018high} is based on projection under a sparsity assumption;  the test  by \cite{jirak2015uniform} is based on taking maximum of componentwise CUSUM statistics. On the other hand, the test by \cite{horvath2012change} aggregates the componentwise CUSUM statistic by using the sum, and is thus expected to have power against dense alternatives. However their asymptotic theory is mostly based on independent panel/component assumption and imposes the restrictive growth rate assumption $p/n^2=o(1)$; the test developed by 
\cite{enikeeva2013high}  is adaptive in the sense that it can capture both sparse and dense alternatives. However, the latter paper imposed Gaussian and independent components assumptions and the validity of their method seems questionable when these strong assumptions are violated (see Section~\ref{sec:simulation} for numerical evidence). The test by \cite{cho2016change} is based on the double CUSUM statistic which utilizes the cross-sectional change-point structure by examining the cumulative sums of ordered CUSUMs at each point. A standard binary segmentation procedure was used to estimate the multiple change points and its consistency was shown for high-dimensional time series. Note that several tuning parameters need to be chosen for the double CUSUM based procedure and the computation cost is high due to the use of bootstrap; see Section~\ref{sec:simulation} for some comparisons.

In this paper, we propose a new class of test statistics that target dense alternatives in the high-dimensional setting with either one single change point or multiple change points, which has received relatively less attention in the literature. The focus on the dense alternative can be well motivated by real data and is often the type of alternative we are interested in. For example, copy number variations in cancer cells are commonly manifested as change-points occurring at the same positions across many related data sequences corresponding to cancer samples and biologically-related individuals; see Fan and Mackey (2017). As a second example, the financial crisis is expected to have an impact on a large number of sectors and their stock returns, so a dense change is expected if we study the  stock returns time series for many sectors.  Our approach is nonparametric, requires quite mild structural assumptions on the data generating process, and does not impose any sparsity assumptions. Due to the use of self-normalization the limiting distributions of the proposed tests are pivotal. We note that, while self-normalized change point tests with pivotal limit were also obtained in \cite{shao2010testing} and \cite{zhang2018unsupervised}, the test statistics in the latter papers can not be used when $p \geq n$. Even when $p < n$ but $p$ is moderately large relative to $n$, those tests typically do not work well as shown in some unreported simulations.

To fix ideas, we begin by considering the setting of one single change point alternative for high-dimensional independent data. To construct a procedure that works under mild assumptions on $p$, we build upon the insights from \cite{chen2010two} who demonstrated that U-statistics provide a very effective means of  comparing two high-dimensional mean vectors. Deriving the limiting distribution of our tests requires control over a collection of high-dimensional sequential U-statistics computed from a growing number of different sub-samples. This is achieved by establishing the weak convergence of a two-parameter stochastic process in the form of sequential U-statistic under sensible and mild assumptions. Given this crucial theoretical ingredient, we are able to derive the limiting null distribution of our test for a single change point. Practically, critical values of the proposed test can be obtained by simulation as the limiting null distribution is pivotal, and the procedure is rather straightforward to implement as no tuning parameter is involved. We further derive the power under local alternatives.

Next, we present extensions of this approach to testing against an unknown number of change-points in the spirit of \cite{zhang2018unsupervised} (who only considered fixed $p$) and consider the problem of testing for a change point in the covariance matrix. As in the single change point setting we obtain tests with pivotal limits. All tests are examined in the simulation studies and exhibit quite accurate size and decent power properties relative to some existing ones.

To extend our U-statistic based approach to high-dimensional time series, we introduce a  trimmed version of the original U-statistic. As suggested by preliminary simulations and theoretical calculations this is crucial in the high-dimensional regime in order  to alleviate the impact of temporal dependence on the bias of U-statistic. This trimmed statistic provides a basic ingredient for self-normalized test under simple and multiple change-point alternatives. We derive the limiting distributions under both the null and alternatives for high-dimensional linear processes and under fixed-$b$ asymptotics [\cite{kiefer2005new}], i.e., we assume that the trimming parameter $\tau$ satisfies $\tau/n = \eta\in (0,1)$, and show how the resulting limiting null distribution depends on $\eta$. This provides a better approximation to the finite sample distribution than the conventional small-$b$ counterpart. Finally, we combine the idea of wild binary segmentation [\cite{fryzlewicz2014wild}] with the SN-based test statistic to estimate the number and location of change points, and demonstrate its effectiveness as compared to several competitors in the literature. 


The rest of the paper is structured as follows. Section~\ref{sec:teststat} introduces our SN-based test statistics for both one single change point and multiple change points alternatives.
A rigorous theoretical justification for their limiting properties under the null and alternatives is provided in Section~\ref{sec:theory}, which also contains
a theoretical extension to test for covariance matrix change. 
Section~\ref{sec:HDLP} presents an extension of the U-statistic based approach to the high-dimensional time series setting to test for a single mean shift. In Section~\ref{sec:estimation}, we present an algorithm based on wild binary segmentation and our SN-based test to estimate the number and locations of change points. Section~\ref{sec:simulation} contains all simulation results. Section~\ref{sec:conclusion} concludes. The technical proofs and some additional simulation results are relegated to supplementary material.

A word about notation.  For any real-valued vector \change{$\delta = (\delta_1, \delta_2,..., \delta_p)^T \in \mathbb{R}^p$}, its $L^1$-norm and $ L^2$-norm are denoted as $\| \delta \|_1 := \sum_{i=1}^p |\delta_i|$ and $\| \delta \|_2:= (\sum_{i=1}^p \delta_i^2)^{1/2}$.
 For any matrix \change{$A = (a_{i,j})_{i=1,..., n; j =1,..., m} \in \mathbb{R}^{n \times m}$}, its $L_{1}$ norm is denoted  as $\| A \|_{1} := \max_j \sum_{i =1}^n | a_{i,j} | $, $L_{\infty}$ norm denoted as $ \| A \|_{\infty} := \max_i \sum_{j =1}^m | a_{i,j} |$,  the spectral norm by $\|A\|_2 := \sigma_{max}(A)$, with $ \sigma_{max}$ denoting the largest singular value and Frobenius norm as $\|A\|_F := \{\sum_{i=1}^{n}\sum_{j=1}^{m}a_{i,j}^2\}^{1/2}$. \change{We denote the trace of \csv{a symmetric matrix} $A$ as $tr(A)$.} The joint cumulant of $n$ random variables \change{$Z_1,..., Z_n$} is denoted is as \change{$cum(Z_1, Z_2,..., Z_n)$}.
The notation \change{${\1}_E$ equals to $1$ if \csv{condition} $E$ is satisfied and zero otherwise. We use ``$\overset{\mathcal{D}}{\rightarrow}$'' to denote the convergence in distribution for \csv{random vectors}, and ``$\rightsquigarrow$'' to denote the weak convergence for stochastic processes.}

\section{Test statistics for high-dimensional independent data}
\label{sec:teststat}

\subsection{Single change-point} \label{sec:single}

To introduce our test statistic, we shall first focus on the single change point alternative, i.e.,
\[
\mathcal{H}_1':\mu_1=\mu_2=\cdots=\mu_k\not=\mu_{k+1}=\cdots=\mu_{n},~\mbox{for some}~1\le k\le n-1.
\]
An extension to general case (i.e., $\mathcal{H}_1$) will be made later. Assume that we observe a sample $Y_1,...,Y_n$. We shall describe the underlying rationale in forming our test in two steps. {We begin by recalling the U-statistic approach pioneered by {\cite{chen2010two}} for comparing high-dimensional means from two samples. For $x_1,...,x_4\in R^p$ define $h((x_1,x_2),(x_3,x_4))=(x_1-x_3)^T(x_2-x_4)$. Then
\[
E[h((X,X'),(Y,Y'))]=\|E(X)-E(Y)\|^2,
\]
where $(X',Y')$ is an i.i.d. copy of $(X,Y)$. In other words the parameter $\|E(X)-E(Y)\|^2$ can be estimated by a two-sample U-statistic with kernel $h$. This insight provides the basic building block for the following approach.}

\bigskip

\textbf{Step 1: Form U-statistic based process.} {For any given candidate change point location $k$ compute the two-sample U-Statistic
\begin{eqnarray*}
G_n(k) = \frac{1}{{k (k-1)}}\frac{1}{{(n-k)(n-k-1)}} \sum_{\stackrel{1 \leq j_1, j_3 \leq k}{j_3 \neq j_1}}\sum_{\stackrel{k+1 \leq j_2, j_4 \leq n}{j_2 \neq j_4}} (Y_{j_1}-Y_{j_2})^T(Y_{j_3}-Y_{j_4}).
\end{eqnarray*}
It is not hard to see that under $\mathcal{H}_0$, $\E[G_n(k)] = 0~\forall k$ while $\sup_k \E[G_n(k)] > 0$ under $\mathcal{H}_1'$. This suggests that a consistent test for $\mathcal{H}_1'$ can be constructed by considering the statistic
\[
\sup_{1 \leq k \leq n} w_n(k)|G_n(k)|
\]
with $w_n(k)$ denoting suitable weights. The first challenge in applying this test in practice lies in deriving the limiting distribution of $\sup_{1 \leq k \leq n} w_n(k) |G_n(k)|$ under the null. The results in \cite{chen2010two} suggest that each individual $G_n(k)$ is asymptotically normal, but that is insufficient to find the asymptotic distribution of $\sup_{1 \leq k \leq n} w_n(k) |G_n(k)|$. The process convergence theory that we develop in this paper enables us to overcome this challenge, and given our results it is possible to show that
\[
\sup_{1 \leq k \leq n} \|\Sigma\|_F^{-1} \left(\frac{2}{k(k-1)} + \frac{2}{(n-k)(n-k-1)} + \frac{4}{k(n-k)}\right)^{-1/2} |G_n(k)| \Dkonv W,
\]
where $W$ denotes a pivotal random variable and $\Sigma := Cov(Y_1)$. However, this does not directly lead to an applicable test since the scaling $\|\Sigma\|_F^{-1}$ is unknown. Ratio-consistent estimation of $\|\Sigma\|_F^2$ is a difficult problem when $p$ is large, and this is particularly true in the change point testing context. The estimator used in {\cite{chen2010two}} is consistent under the null, but no longer consistent under the alternative due to a change point in mean. It is possible to formulate Kolmogorov-Smirnov type test with consistent estimation of $\|\Sigma\|_F$ (see Section~\ref{sub:mean} for the details and simulation comparisons), but we will next propose to use an approach that completely avoids consistent estimation.}

\medskip

\textbf{Step 2: Self-normalization.} {The essence of SN is to avoid using a consistent estimator of the unknown parameter in the scale, which is $\|\Sigma\|_F^2$ in the present setting. As we mentioned before, consistent estimation of $\|\Sigma\|_F$ is difficult in the change point setting (especially with multiple unknown change points). The approach in {\cite{shao2010testing}} is not applicable in the present setting, however the basic strategy to use estimators from sub-samples still works after a suitable adaptation.} Define
\begin{equation} \label{eq:defDn}
D(k;\ell,m) := \sum_{\stackrel{\ell \leq j_1, j_3 \leq k}{j_3 \neq j_1}}\sum_{\stackrel{k+1 \leq j_2, j_4 \leq m}{j_2 \neq j_4}} (Y_{j_1}-Y_{j_2})^T(Y_{j_3}-Y_{j_4})
\end{equation}
for $1 \leq \ell \leq k < m \leq n$ and $D(k;\ell,m) = 0$ otherwise. Note that $D(k;1,n)$ is simply a scaled version of $G_n(k)$ defined previously while $D(k;\ell,m)$ can hence be interpreted as a scaled version of the U-Statistic $G_n$ computed on the sub-sample $Y_{\ell},Y_{\ell+1},...,Y_m$. Letting
\begin{equation} \label{eq:defWn}
W_n(k;\ell,m) := \frac{1}{n}\sum_{t = \ell+1 }^{k-2} D(t;\ell,k)^2 + \frac{1}{n}\sum_{t = k+2}^{m-2} D(t;k+1,m)^2,
\end{equation}
the self-normalized test statistic for the presence of a single change point takes the form
\begin{equation}\label{eq:defTn}
T_n := \sup_{k=4,...,n-4} \frac{\{D(k;1,n)\}^2}{W_n(k;1,n)}.
\end{equation}

Heuristically, the fact that $D$ computed on various sub-samples appears both in the numerator and denominator, means that the unknown factor $\|\Sigma\|_F^2$ in their variance cancels out and the limit becomes pivotal; see Theorem~\ref{cor:T0} for a formal statement. {The key to deriving the asymptotic distribution of $T_n$ defined above is to establish the joint behavior of the collection of statistics $D(k;\ell,m)$ indexed by $k, \ell, m$. Due to the U-Statistic nature of our problem this result does not follow from statements about $G_n(k)$ and involves additional technical difficulties.}

{
Note that our test statistic can be computed at the cost of $O(n^2p)$. 
To this end, observe that 
\begin{align*}
D(k;\ell,m) =~& 2(m-k)(m-k-1)  S_n(\ell,k) + 2 (k-\ell)(k-\ell+1) S_n(k+1,m)
\\
& - 2(k-\ell+1)(m-k) ( S_n(\ell,m) -  S_n(\ell,k) -  S_n(k+1,m)),
\end{align*}
where $S_n(k,m) = \sum_{i = k}^{m}\sum_{j = k}^{i} Y_{i+1}^TY_j$. Many quantities in $\{ S_n(k,m) \}_{k < m}$ are repeatedly used in the calculation of our test statistic $T_n$. The trick is to calculate $S_n(k,m)$ for all $1 \leq k < m \leq n$ first, which can be done with the cost $O(n^2p)$. Once $S_n(k,m)$ is available for all $k<m$, $D(k;\ell,m)$ can be computed at the cost of $O(1)$ for fixed $k,l,m$, and $T_n$ at the cost of $O(n^2)$. Hence the total computation cost is of order $O(n^2p)$.}




\subsection{Extension to multiple change-points}

In practice, the number of change points under the alternative is often unknown, which is the `unsupervised' case considered in {\cite{zhang2018unsupervised}}. It is expected that the SN-based test developed in the previous section may lose power when the number of change points is more than one; see Section~\ref{sub:mean} for simulation evidence. Thus it is desirable to develop a test that is adaptive,  i.e., has reasonable power without the need to specify the number of change points under the alternative. Here, we propose to combine the scanning idea in {\cite{zhang2018unsupervised}} and the SN-based test proposed above to form our unsupervised test statistic. To this end, we consider the following additional notation. Following {\cite{zhang2018unsupervised}} define the sets
\begin{align*}
\Omega(\epsilon) &= \{(t_1,t_2) \in [\epsilon,1-\epsilon]^2: t_1 < t_2,  t_2-t_1 \geq \epsilon\},
\\
\Omega_n(\epsilon) &= \{ (k_1,k_2) \in \N^2 : (k_1/n,k_2/n) \in \Omega(\epsilon)\},
\\
\mathcal{G}_\epsilon &= \{k\epsilon/2, k\in\mathbb{Z}\} \cap [0,1],
\\
\mathcal{G}_{\epsilon,n,f} &= \{(\floor{t_1 n}\vee 1,\floor{t_2 n}\vee 1) \in \N^2: (t_1,t_2)\in([0,1]\times\mathcal{G}_{\epsilon})\cap \Omega(\epsilon)\},
\end{align*}
and
\begin{align*}
\mathcal{G}_{\epsilon,n,b} &= \{(\floor{t_1 n}\vee 1,\floor{t_2 n}\vee 1) \in \N^2: (t_1,t_2)\in(\mathcal{G}_{\epsilon} \times [0,1])\cap \Omega(\epsilon)\}.
\end{align*}
The first test statistic now takes the form
\begin{equation} \label{eq:Tn_star}
T_n^{*} := \max_{(l_1,l_2) \in \Omega_n(\epsilon)}\frac{D(l_1;1,l_2)^2}{W_n(l_1;1,l_2)} + \max_{(m_1,m_2) \in \Omega_n(\epsilon)} \frac{D(m_2;m_1,n)^2}{W_n(m_2;m_1,n)}.
\end{equation}
One potential issue with this definition is that it involves the computation of $D_n(l_1;1,l_2)^2$ for $O(n^2)$ combinations of $l_1,l_2$ which can be expensive, especially when $n$ and $p$ are both large. To relax the computational burden, {\cite{zhang2018unsupervised}} also consider a discretised version. In our setting it takes the form
\begin{equation} \label{eq:Tn_diamond}
T_n^{\diamond} := \max_{(l_1,l_2) \in \mathcal{G}_{\epsilon,n,f}} \frac{D(l_1;1,l_2)^2}{W_n(l_1;1,l_2)} + \max_{(m_1,m_2) \in \mathcal{G}_{\epsilon,n,b}} \frac{D(m_2;m_1,n)^2}{W_n(m_2;m_1,n)}.
\end{equation}
It is worth noting that $\epsilon$ is a trimming parameter that needs to be specified by the user.
We set $\epsilon=0.1$ following the practice of \cite{zhang2018unsupervised}, who also provided some discussion on the
 role of $\epsilon$ in the testing.

\section{Theoretical properties}
\label{sec:theory}

Asymptotic properties of the proposed tests will be derived in a triangular array setting where $p = p_n$, the dimension of $X_0$, diverges to infinity. We will need the following  regularity assumptions.

\begin{assumption} \label{ass}
The observations are $Y_{t,n} = \mu_{t,n} + X_{t,n}, t=1,...,n$. $X_{1,n},...,X_{n,n}$ are i.i.d. copies of the $\R^{p_n}$-valued random vector $X_{0,n}$ with $\E[X_{0,n}] = 0$ and $\E[X_{0,n}X_{0,n}^T] = \Sigma_n$. Moreover
\begin{enumerate}[label=A.\arabic*]
\item \label{order}
{$tr(\Sigma_n^4) = o(\|\Sigma_n\|_F^4)$},
\item \label{cumulant} There exists a constant $C$ independent of $n$ such that
\change{\[
\sum_{l_1,...,l_h = 1}^{p}cum^2(X_{0,l_1,n},...,X_{0,l_h,n}) \leq C\|\Sigma_n\|_F^h,
\]}
for $h = 2,3,4,5,6$.
\end{enumerate}
\end{assumption}

We remark that the dimension $p = p_n$ of the vector $X_0$, the vectors $\mu_i$, and the covariance matrix $\Sigma_n$ change with $n$. To keep the notation simple this dependence will be dropped in all of the following results whenever there is no risk of confusion.

\begin{remark}[Discussion of Assumptions] \label{rem:condA} { \rm Simple computation shows that Assumption~\ref{order} is equivalent to $\|\Sigma_n\|_2 = o(\|\Sigma_n\|_F)$, see section~\ref{sec:proofA1A3} in the supplement for details. Hence Assumption~\ref{order} can only hold if $p =p_n \to \infty$ as $n \to \infty$. All other conditions can be satisfied under uniform bounds on moments and `short-range' dependence type conditions on the entries of the vector $(X_{0,1,n},...,X_{0,p_n,n})$. For illustration purposes, consider the following conditions.
\begin{enumerate}
\item[(i)] There exists $c_0 > 0$ independent of $n$ such that $\inf_{i=1,...,p_n} Var(X_{0,i}) \geq c_0$.
\item[(ii)] For $h=2,...,6$ there exist constants $C_h$ depending on $h$ only and a constant $r > 2$ independent of $n,h,m_1,...,m_h$ such that
\[
|cum(X_{0,m_1,n},...,X_{0,m_h,n})| \leq C_h (1 \vee \max_{1 \leq i,j\leq h} |m_i - m_j|)^{-r}.
\]
Note that this assumption is trivially satisfied if the entries of $(X_{0,1,n},...,X_{0,p_n,n})$ are m-dependent over $i$, i.e., if two groups $\{X_{0,i,n}: i \in J_1\}, \{X_{0,i,n}: i \in J_2\}$ are independent whenever $\inf_{i \in J_1, j \in J_2} |i-j| > m$ and if moments of order $h$ are uniformly bounded. It can also be verified under other conditions such as mixing plus moment assumptions [\cite{ZZ1975}] or physical dependence measures, see for instance  Proposition 2 of {\cite{wu2004limit}} and Theorem 4.1 of \cite{shaowu2007} for the latter.

\end{enumerate}
Now it is easy to prove (see section~\ref{sec:proofA1A3} in the supplement for details) that if $p_n \to \infty$, (i) holds and (ii) holds for some $r>3/2$ then Assumption \ref{ass} holds.
}
\end{remark}

{\begin{remark}[Comparison with \cite{chen2010two}]\label{rmk:cq}
 {\rm \change{Although \cite{chen2010two} studied a two-sample mean testing problem which is different from the change point setting \csv{we consider here, the weak cross-sectional dependence condition was also required in their theory to obtain a Gaussian limit.}}
 To quantify the dependence among different components of the vector $X_1$, \cite{chen2010two} proposed a factor model. More precisely they assume that $X_i = \Gamma Z_i$ where $Z_i$ are m-dimensional random vectors with the additional property $\E[Z_{t,l_1}^{\alpha_1}\cdots Z_{t,l_q}^{\alpha_q}] = \E[Z_{t,l_1}^{\alpha_1}]\cdots\E[Z_{t,l_q}^{\alpha_q}]$ for all $l_1\neq...\neq l_q$ and integers $\alpha_k \leq 4$ with $\sum_{k} \alpha_k \leq 8$. In contrast, we assume \ref{cumulant} without imposing a factor model structure. As we shall prove in section \ref{pr:rmk:cq}, the factor model structure of \cite{chen2010two} together with finite moments of order $6$ implies our condition \ref{cumulant}. Moreover, a close look at the proofs reveals that for proving finite-dimensional convergence we only require \ref{cumulant} with $h \leq 4$, which follows from the assumptions of \cite{chen2010two}. Hence, we prove a result which corresponds to that of \cite{chen2010two} under strictly weaker assumptions on the dependence structure and provide process convergence results under only slightly stronger moment conditions and still weaker structural assumptions.
}
\end{remark}}

\subsection{Properties of the test for a single change-point}\label{sec:thsingle}

We begin by deriving the limiting distribution of the test statistic $T_n$ defined in~\eqref{eq:defTn}.

\begin{theorem} \label{cor:T0}
Let Assumption~\ref{ass} hold. If $\mu_t \equiv \mu$ for a vector $\mu \in \R^p$ (i.e. under $\mathcal{H}_0$) then
\[
T_n \Dkonv T = \sup_{r \in [0,1]}\frac{G(r;0,1)^2}{\int_{0}^{r}G(u;0,r)^2du + \int_{r}^{1}G(u;r,1)^2du},
\]
where
\begin{align}\label{eq:defGrab}
G(r;a,b)
& := (b-a)(b-r)Q(a,r) + (r-a)(b-a)Q(r,b) - (r-a)(b-r)Q(a,b)
\end{align}
and $Q$ is a centered Gaussian process on $[0,1]^2$ with covariance structure given by
\begin{equation}\label{eq:covQ}
Cov(Q(a_1,b_1),Q(a_2,b_2)) =  (b_1 \wedge b_2 - a_1 \vee a_2)^2 \1\{b_1 \wedge b_2 > a_1 \vee a_2\}.
\end{equation}
\end{theorem}

The limiting distribution $T$ is pivotal, and an asymptotic level $\alpha$ test for $\mathcal{H}_0: \mu_t \equiv \mu$ is thus given by the decision: reject $\mathcal{H}_0$ if $T_n > Q_T(1-\alpha)$ where $Q_T(1-\alpha)$ denotes the $1-\alpha$ quantile of the distribution of $T$. Simulated quantiles from this distribution (based on 10000 Monte Carlo replications) are provided in Table~\ref{tab:Tncrit}.

{
{Note that the above limiting null distribution requires that $p\wedge n\rightarrow\infty$, (this must hold for Assumption A.1 to be satisfied)}, and does not hold when $p$ is fixed and $n\rightarrow\infty$. Our SN-based test statistic $T_n$ builds on the two sample test statistic proposed by \cite{chen2010two}, whose limit under the fixed $p$ paradigm is expected to be non-Gaussian, as their test statistic is a degenerate $U$-statistic under the null. Here the assumption $p\rightarrow\infty$ is essential to our Gaussian process limit for the  two-parameter process
$\Big\{ \frac{\sqrt{2}}{n\|\Sigma\|_F}\widetilde S_n(\floor{an}+1,\floor{bn}-1) \Big\}_{(a,b)\in [0,1]^2},$
 which is the key to derive the limiting null distribution of $T_n$; see Section \ref{sec:appendixHDID} in the supplement.}

	
\begin{table}[h!]
\centering
\begin{tabular}{cccccc}
\hline
$\gamma$&80\%&90\%&95\%&99\%&99.5\%
\\
\hline
$Q_T(\gamma)$&603.72&  881.78 &1177.45 &2026.28 &2443.27
\\
\hline
\end{tabular}
\caption{Simulated quantiles of the limit $T$} \label{tab:Tncrit}
\end{table}

Next we consider the behavior of the test under alternatives. The following result shows that the test is consistent against local alternatives of a certain order if there is exactly one change-point.

\begin{theorem} \label{cor:T_alt}
Let Assumption~\ref{ass} hold. Assume that there exists $b^* \in (0,1)$ such that $\mu_t = \mu, t=1,...,\floor{b^* n}$ and $\mu_t = \mu + \delta_n, t=\floor{b^* n}+1,...,n$. Then
\begin{enumerate}
\item If $\sqrt{n}\|\delta_n\|_2/\|\Sigma\|_F^{1/2} \to \infty$ then $T_n \to \infty$ in probability.
\item If $\sqrt{n}\|\delta_n\|_2/\|\Sigma\|_F^{1/2} \to 0$ then $T_n \Dkonv T$.
\item If $\sqrt{n}\|\delta_n\|_2/\|\Sigma\|_F^{1/2} \to c \in (0,\infty)$ then
\[
T_n \Dkonv \sup_{r \in [0,1]}\frac{\{\sqrt{2} G(r;0,1) + c \Delta(r,0,1)\}^2}{\int_{0}^{r}\{\sqrt{2}G(u;0,r) + c \Delta(u,0,r)\}^2du + \int_{r}^{1}\{\sqrt{2}G(u;r,1) + c \Delta(u,r,1)\}^2du},
\]
where
\[
\Delta(r,a,b) := \begin{cases}
(b^*-a)^2(b-r)^2 & a < b^* \leq r < b,
\\
(r-a)^2(b-b^*)^2 & a < r < b^* < b,
\\
0 & b^* \le a \mbox{ or } b^* \ge b.
\end{cases}
\]
\end{enumerate}
\end{theorem}


\subsection{Properties of the tests for multiple change-points}

To describe the properties of the test statistics $T_n^*, T_n^\diamond$ under the null, define for $0 \leq r_1 < r_2 \leq 1$ and
$0 \leq s_1 < s_2 \leq 1$,
\begin{align*}
T_1(r_1,r_2) &:= \frac{G(r_1;0,r_2)^2}{\int_{0}^{r_1}G(u;0,r_1)^2du + \int_{r_1}^{r_2}G(u;r_1,r_2)^2du},
\\
T_2(s_1,s_2) &:= \frac{G(s_2;s_1,1)^2}{\int_{s_1}^{s_2}G(u;s_1,s_2)^2du + \int_{s_2}^{1}G(u;s_2,1)^2du}.
\end{align*}

\begin{theorem} \label{th:multcp0} Let Assumption~\ref{ass} hold and assume $\epsilon < 1/4$. If $\mu_t \equiv \mu$ for a vector $\mu \in \R^p$ (i.e. under $H_0$) then
\begin{align*}
T_n^* &\Dkonv T^* := \sup_{(r_1,r_2)\in \Omega(\epsilon)} T_1(r_1,r_2)  + \sup_{(s_1,s_2)\in \Omega(\epsilon)} T_2(s_1,s_2),
\\
T_n^\diamond &\Dkonv T^\diamond := \sup_{(r_1,r_2)\in (\mathcal{G}_\epsilon\times[0,1]) \cap \Omega(\epsilon)} T_1(r_1,r_2)  + \sup_{(s_1,s_2)\in ([0,1]\times\mathcal{G}_\epsilon) \cap \Omega(\epsilon)} T_2(s_1,s_2).
\end{align*}
\end{theorem}

The distributions of $T^*, T^\diamond$ are again pivotal but depend on $\epsilon$ (which is known since it is chosen by the user).
For $\epsilon=0.1$ used in the paper, the critical values of $T^\diamond$ are tabulated in Table~\ref{tab:Tndiamondcrit} below.

\begin{table}[h!]
\centering
\begin{tabular}{cccccc}
\hline
$\gamma$&80\%&90\%&95\%&99\%&99.5\%
\\
\hline
$Q_{T^{\diamond}}(\gamma)$&7226.18	&8762.45&	10410.19& 	14603.51& 	16608.86
\\
\hline
\end{tabular}
\caption{Simulated quantiles of the limit $T^{\diamond}$} \label{tab:Tndiamondcrit}
\end{table}

 To describe the properties of the tests based on $T_n^*,T_n^\diamond$ under the alternative (where we could have several change-points), assume that for some $\epsilon < b_1^* < b_2^* < ... < b_M^* < 1-\epsilon$ we have
\[
\mu_t = \mu_k^* \quad \floor{nb_k^*} + 1 \leq t \leq \floor{nb_{k+1}^*}, \quad k=0,...,M
\]
where we defined $b_0^* = 0, b_{M+1}^* = 1$ and $\mu_0 \neq \mu_1 \neq ... \neq \mu_M$ denote vectors in $\R^p$.

\begin{theorem} \label{th:multcp1}
Let Assumption~\ref{ass} hold and assume $\epsilon < 1/4$. Additionally, assume that in the setting given above we have $\inf_k |b_k^* - b_{k+1}^*| \geq \epsilon, \sup_k \sqrt{n}\|\mu_k^* - \mu_{k+1}^*\|_2/\|\Sigma\|_F^{1/2} \to \infty$. Then $T_n^* \to \infty$ in probability and $T_n^\diamond \to \infty$ in probability. 
\end{theorem}

\subsection{Application to testing for changes in the covariance structure}

In this subsection, we shall focus on testing for a change in the covariance matrix, which is an important problem in the analysis of multivariate data, and has applications in many areas, such as economics and finance.  \cite{aue2009break} proposed a CUSUM-based test in the low dimensional time series setting and documented the early literature, which is mostly focused on the low dimension high sample size setting. In the high dimensional environment, the only work we are aware of is \cite{avanesov2016change}, which will be introduced and compared in our simulation studies; see Section~\ref{sec:covsim} of  the supplement. Following the latter paper, we assume $\mu_{t,n} = 0, t=1,...,n$. Define $Z_0 = vech(X_0X_0^T)$ \change{as the half-vectorization of $X_0X_0^T$, i.e. the vectorization of the lower triangular part \csv{(including the diagonal)} of $X_0X_0^T$.} If $\E[X_0] = 0$ then $\E(Z_0) = vech(\Sigma_X)$. Tests for changes in $\Sigma_X$ can thus be constructed by applying the test statistics from the previous sections to the transformed observations $Z_t := vech(X_tX_t^T), t=1,...,n$. In what follows we provide a result that allows to verify Assumption~\ref{ass} for $Z_0$ from properties of $X_0$.

\begin{proposition} \label{prop:covm}
The vector $Z_0 := vech(X_0X_0^T)$ satisfies Assumption~\ref{ass} provided that the following conditions hold for $X_{0}$ with $\E[X_0] = 0$ and $\Sigma_n := \E[X_0X_0^T].$
{\begin{enumerate}[label=B.\arabic*]
\item  $\|\Sigma_n\|_1 = o(\|\Sigma_n\|_F)$\label{cov:ubound}.
\item $ \max_{l_1,l_2=1,...,p}\sum_{l_3,l_4 = 1}^p|cum(X_{0,l_1},X_{0,l_2},X_{0,l_3},X_{0,l_4})|=o(\|\Sigma_n\|_F^2)$. \label{cov:cum1}
\item \label{cov:cum2} There exists a constant $C$ such that \change{$\sum_{l_1,...,l_h = 1}^{p}cum^2(X_{0,l_1},...,X_{0,l_h}) \leq C\|\Sigma_n\|_F^h$}
for  \change{$h = 2,...,12$}. Moreover
\change{\[
\sum_{l_1,...,l_4 = 1}^{p}cum^2(X_{0,l_1},...,X_{0,l_4}) = o(\|\Sigma_n\|_F^4).
\]}
\end{enumerate}}
\end{proposition}

\begin{remark}[Discussion of Assumptions] \label{rem:condB} {\rm Similar to Remark \ref{rem:condA}, Assumptions \ref{cov:ubound} - \ref{cov:cum2} can be verified by considering the following conditions: (1) $p_n \to \infty$; (2) there exists $c_0 > 0$ independent of $n$ such that $\inf_{i=1,...,p_n} Var(X_{0,i}) \geq c_0$; (3) there exist $c_1 >0$ such that $Var(X_{0,i}X_{0,j}) \geq c_1 > 0$, $\forall 1 \leq i \leq j \leq p$; (4) for $h=2,...,12$ there exist constants $C_h$ depending on $h$ only and a constant $r > 2$ independent of $n,h,m_1,...,m_h$ such that
\[
|cum(X_{0,m_1,n},...,X_{0,m_h,n})| \leq C_h (1 \vee \max_{1 \leq i,j\leq h} |m_i - m_j|)^{-r}.
\]
This can be easily satisfied if the entries of \change{$(X_{0,1,n},...,X_{0,p_n,n})$} are m-dependent and moments of order $12$ are uniformly bounded or under suitable conditions on short-range dependence; see Remark~\ref{rem:condA} for additional details. A proof of this statement is given in Section~\ref{sec:proofA1A3}.}
\end{remark}

\begin{remark}{\rm 
	As pointed out by a referee, we vectorize the covariance matrix and apply the mean change point test, which may not be efficient, since we ignore certain structures of covariance matrices such as symmetricity and positive definiteness. In the two sample testing context, \cite{li12} proposed a novel test for the equality of two high-dimensional covariance matrices by using U-statistic for the scalar parameter $tr\{(\Sigma_1-\Sigma_2)^2\}$, where $\Sigma_j$ denotes the covariance matrix for the $j$th population, $j=1,2$. The test by \cite{li12} can be naturally viewed as an extension of \cite{chen2010two} from the mean testing to covariance matrix testing.  Given this connection, it is indeed possible to build on \cite{li12} to propose a SN-based test for a change-point in covariance matrix, following the developments presented in Section~\ref{sec:single}. However, the associated theory seems fairly complex and we shall leave it for future investigation. }   
\end{remark}

\section{Test statistics for high-dimensional time series}
\label{sec:HDLP}

In this section, we assume that $\{ Y_{t} \}_{t=1}^{n}$ is a realization of $\mathbb{R}^{p}$-valued time series with weak temporal dependence. To extend the U-statistic based approach from high-dimensional independent data to weakly dependent high-dimensional time series, we formulate a trimmed version of the $U$-statistic that excludes pairs of points that are close on time scale.
Trimming is crucial in the high-dimensional context to remove the bias caused by weak temporal dependence and is common for the use of U-statistic in the time series setting. It is also routinely applied in fixed dimensions; see \cite{lee1990}. To confirm the need for trimming, we implemented the untrimmed test statistic $T_n$ for the VAR$(1)$ model in Example~\ref{exp:1} for both $n=p=100$ and $n=p=200$ with $\rho=-0.5, 0.5, 0.7$, and the empirical sizes are uniformly zero for all cases (results based on 2000 replications). This is due to the fact that the temporal dependence incurs a non-negligible bias for the denominator $D(k;1,n)$ (and more generally $D(k;l,m)$)
as under the null and for stationary time series, $E\{D(k;l,m)\}$ is a linear combination of the auto-covariance based terms $E\{(Y_0-\mu)^T(Y_h-\mu)\}$, \change{$h=1,2,...$}, which vanish under the i.i.d. assumption. As an alternative approach, \cite{li2019change} proposed to estimate the bias explicitly, and we shall compare the two approaches in terms of estimation accuracy in Section~\ref{sub:estimateTS} of the supplement.

 Motivated by the discussion above, we modify the statistic $D$ in equation \eqref{eq:defDn} by removing all terms of the form $Y_i^TY_j$ for which $|i-j| \leq \tau$. This considerably reduces the bias which is introduced by weak temporal dependence of the $Y_i$. The resulting trimmed statistic is of the form
 \begin{align*}
 D(k;l,m|\tau)  = & \sum_{\stackrel{l \leq j_1 , j_3 \leq k}{|j_1-j_3|>\tau}}\sum_{\stackrel{k+\tau+1 \leq j_2, j_4 \leq m}{|j_2-j_4|>\tau}} (Y_{j_1} - Y_{j_2})^T (Y_{j_3} - Y_{j_4}),
 \end{align*}
 where $\tau $ is a given positive integer such that $l + \tau +1 \leq k  \leq m - 2\tau - 2$. It is clear that when $\tau =0$, $D(k;l,m| 0) = D(k;l,m)$, where $D(k;l, m)$ is defined in Equation \eqref{eq:defDn}. Furthermore, we let 
 \begin{align*}
 W_{n}(k;l,m|\tau) :=\frac{1}{n} \sum\limits_{t=l+\tau+1}^{k-2\tau-2} D^2(t;l,k|\tau) + \frac{1}{n} \sum\limits_{t=k+\tau+2}^{m-2\tau-2} D^2(t;k+1,m|\tau),
 \end{align*} 
 where $ l+\tau +1 \leq k - 2\tau -2 $ and $ k+\tau +2 \leq m-2\tau -2 $. The self-normalized statistic is then defined as \change{
 \begin{align*}
 T_n := \sup_{k=3\tau+4,..., n-3\tau-4} \frac{D^2(k;1,n|\tau)}{ W_{n}(k;1,n|\tau) }.
 \end{align*}}
 
In the theoretical developments that follow, we assume $\tau = \lfloor \eta n \rfloor, \eta \in (0,1)$ and fix $\eta$ in our asymptotic framework, in other words we consider fixed-$\eta$ asymptotics [this type of approach is termed fixed-b asymptotics in \cite{kiefer2005new}. This is motivated by preliminary simulations, where we found that the limiting null distribution derived under the small-$\eta$ asymptotics (i.e., $\eta\rightarrow 0$ as $n\rightarrow\infty$) provides a poor approximation to the finite sample distribution under the null especially when $\eta$ is not very small, which is required when the temporal dependence is moderate or strong. Explicitly taking into account the effect of trimming through fixed-$\eta$ asymptotics results in a much more accurate size as seen in our simulations.
Note that fixed-$b$ asymptotics and self-normalization are quite related in many ways and for some problems, self-normalization is a special case of fixed-$b$ asymptotics; see \cite{shao2010self} and \cite{shao2015self} for more discussions about the connection and difference.

Compared to the analysis in Section~\ref{sec:theory}, the present setting involves two major challenges. First, adopting the fixed-$\eta$ framework results in a more complex statistic and the  simple representation of the process $D$ without trimming (see equation~\eqref{repr:DXtildeS}) does not hold anymore. A somewhat more involved representation needs to be derived instead; see the first two pages in Section~\ref{sec:pfThmDep} and in particular equation~\eqref{repr:DXdep} therein). Second, each of the four U-processes in the new decomposition is now based on dependent rather than independent data and involves additional weighting. This considerably complicates their asymptotic analysis.

To overcome the technical difficulties described above, we will limit our attention to linear processes. In particular, we assume $Y_{t} = \mu_{t} + X_{t} $, \change{$t=1,...,n$}, where $X_t = \sum_{j =0}^{\infty} c_{j} \epsilon_{t-j}$ and $\{\epsilon_t\}$ are i.i.d $p$-dimensional innovations with mean $0$ and $c_{j}$ are $p \times p$ coefficient matrices. Let 
$$ 
\Gamma = \left(\sum_{u = 0}^{\infty} c_{u}\right) cov(\epsilon_0) \left(\sum_{u = 0}^{\infty} c_{u}\right)^T $$ 
be the corresponding long run variance matrix. The linear processes framework is quite general and it includes the well-known ARMA models. From a technical point of view, we are able to take advantage of the Beveridge-Nelson (BN) decomposition [\cite{phillips1992asymptotics}], which can be shown to  work in the high-dimensional setting.    
 
The following assumptions are imposed to study the asymptotic distribution of $ T_n$. 
 
 \begin{assumption} \label{ass:main} Suppose the following assumptions hold. 
 	\begin{enumerate}[label=C.\arabic*]
 		\item \label{A41A1} \change{$\sup_{l=1,..., p} \| \epsilon_{0,l} \|_8 < \infty $.}
 		\item \label{A41A2} For any $m \geq 0$. 
 		\begin{align*}
 		\sum_{u = m}^{\infty} \| c_u \|_{1} \leq C \rho^{m} \text{ and }
 		\sum_{u = m}^{\infty} \| c_u \|_{\infty} \leq C \rho^{m},
 		\end{align*}
 		where $C > 0$ and $ 0 < \rho < 1 $ are some constants.
 		\item \label{A41A3}$ tr(\Gamma^4) = o(\| \Gamma \|_F^4) $.
 		\item \label{A41A4} $p^{6}  \rho^{\lfloor \eta n \rfloor} / \|\Gamma \|_F^{6} = O(1)$.
 		\item \label{A41A5}For any $h = 2,3,4,5,6$,\change{
 		$
 		\sum_{k_1,...,k_h = 1}^p |cum(\epsilon_{0, k_1}, \cdots, \epsilon_{0, k_h})| \leq C'  \| \Gamma \|_F^h,
 		$ }		where $C'$ is some constant independent of $n,p$. \label{ass:a5}
 	\end{enumerate}
 \end{assumption}

\begin{remark} \rm
Assumptions \ref{A41A1} and \ref{A41A2} imply the Uniform Geometric Moment Contraction (UGMC($8$)) property in \cite{wang2019hypothesis}. The UGMC condition is a generalization of Geometric Moment Contraction  in \cite{hsing2004weighted} and \cite{wu2004limit} to the high-dimensional setting and its equivalent form has been used in \cite{zhang2018gaussian}. Assumption \ref{A41A3} is commonly assumed for covariance matrix [e.g., \cite{chen2010two}] and it can be satisfied under some weak cross-sectional and temporal  dependence conditions.  Assumption \ref{A41A4} implies that the bias caused by temporal dependence is asymptotically negligible. Assumption \ref{A41A5} holds under mild conditions, see Section 3 in \cite{wang2019hypothesis} for some verified examples. 
\end{remark}

\begin{remark} \rm
Recently, \cite{wang2019hypothesis} proposed a new way of doing self-normalization for inference of high-dimensional time series. They dealt with one sample testing problem, and also used the trimming technique in their U-statistic. Their asymptotic theory was developed for a broad class of nonlinear causal processes using martingale approximation. To develop our asymptotic theory for nonlinear processes  would be desirable but seems very challenging as we are dealing with a two-sample testing problem with unknown break date, and the process convergence theory we develop seems considerably more involved. 
\end{remark}

 Now  we are ready to state the asymptotic null distribution of $T_n$.
 \begin{theorem} \label{thm:main} Suppose Assumption \ref{ass:main} is true. Then,
 	\begin{align*}
 	T_n  \overset{\mathcal{D}}{\longrightarrow} T(\eta):= \sup_{r \in (3 \eta, 1 - 3\eta)} \frac{G^2(r;0,1|\eta)}{ \int_{\eta}^{r-2\eta} G^2(u;0,r|\eta) du + \int_{r+\eta}^{1-2\eta} G^2(u;r,1|\eta) du},
 	\end{align*}
 	where 
 	\begin{align*}
 	G(r;a,b|\eta) := &(b-r-2\eta)^2 V_{1}(a,r|\eta) + (r-a-\eta)^2 V_{1}(r+\eta,b|\eta) \\
 	& - (r-\eta)(b-\eta)U_1(a,r-\eta;r+\eta,b-\eta)  + (r-\eta)(r+2\eta)U_1(a, r-\eta; r+2\eta, b) \\ 
 	& + (a+\eta)(b-\eta)U_1 (a+\eta, r; r+\eta, b-\eta) - (a+\eta)(r+2\eta)U_1(a+\eta, r;r+2\eta,b) \\
 	& -  (b-\eta) U_2(a, r-\eta; r+\eta,b-\eta)  + (r+2\eta) U_2(a, r-\eta; r+2\eta,b) \\
 	& + (b-\eta) U_2(a+\eta, r; r+\eta,b-\eta)  - (r+2\eta) U_2(a+\eta, r; r+2\eta,b) \\
 	& - (r- \eta)U_3(a, r-\eta;r+\eta,b-\eta)  + (r- \eta)U_3(a, r-\eta;r+2\eta,b) \\
 	& + (a + \eta) U_3(a+\eta, r; r+\eta,b-\eta)  - (a + \eta) U_3(a+\eta, r; r+2\eta,b) \\
 	& + U_4(a, r-\eta;r+\eta,b-\eta)  - U_4(a, r-\eta;r+2\eta,b) \\
 	& - U_4(a+\eta, r;r+\eta,b-\eta)  + U_4(a+\eta, r;r+2\eta,b).
 	\end{align*}
 	For $u,v=1,2,3,4,$ $$ U_u (a_1, a_2;b_1, b_2) = V_u(a_1, b_2|\eta) - V_u(a_1, b_1|\eta) - V_u(a_2, b_2|\eta) + V_u(a_2, b_1|\eta),$$ 
 	and $V_1, V_2, V_3, V_4$ are Gaussian processes with covariance structures 
 	$$
 	cov(V_u(a_1,b_1| \eta), V_v(a_2, b_2|\eta)) = C_{u,v}(a_1 \vee a_2, b_1 \wedge b_2) \1_{ \{ b_1 \wedge b_2- a_1 \vee a_2 - \eta >0 \} },
 	$$
 	where $C_{u,v}(a,b)$ is defined as
 	\begin{align*}
 	C_{u,v}(a,b) = \lim\limits_{n \rightarrow \infty} \frac{2}{n^2 }  \sum_{i = \lfloor an \rfloor}^{ \lfloor b n \rfloor - \lfloor \eta n \rfloor -1} \sum_{j = \lfloor an \rfloor}^{ i }   w_{i,j}^u w_{i,j}^v,
 	\end{align*}
 	with $
 	w_{i,j}^u = \1_{\{ u=1 \} } +\frac{j}{n}  \1_{\{ u=2 \} } + \frac{ i+\lfloor \eta n \rfloor +1 }{n} \1_{\{ u=3 \} } + \frac{ i+\lfloor \eta n \rfloor +1 }{n} \frac{j}{n} \1_{\{ u=4 \}}$.
 	
 \end{theorem}

The limiting distribution $T(\eta)$ derived above is considerably more complicated than in the independent case but still pivotal for given $\eta$. This is because the cross-covariance of the centered processes $V_1,...,V_4$ depends only on $\eta$ and not on any unknown quantities. In other words, our test involves only one trimming parameter, whose impact is captured to the first order by the limiting null distribution. Simulated quantiles of $T(\eta)$ are tabulated in Table \ref{t1}.




 
 \begin{table}[t]
 	\begin{center}
 		\scalebox{1}{
 			\begin{tabular}{@{\extracolsep{5pt}} cccccc}
 				\\[-1.8ex]\hline \hline \\[-1.8ex]
 				$\eta$ & $\alpha=0.2$ & $\alpha=0.1$ & $\alpha=0.05$ & $ \alpha = 0.01 $ & $ \alpha=0.005 $    \\
 				\hline \\[-1.8ex] 
 				$0.01$ & $795.017$ & $1198.187$ &$1639.631$ &$2758.508$ & $3561.943$ \\
 				$0.02$ & $1069.060$ & $1635.213$  &$2203.461$  &$3788.201$  &$4601.053$ \\
 				$0.03$ & $1465.636$ & $2268.597$  &$3137.729$  &$5571.669$  &$6599.765$ \\
 				$0.04$ & $2058.443$ & $3142.107$  &$4441.035$  &$8294.564$  &$10027.183$ \\
 				$0.05$ & $2969.278$ & $4541.604$  &$6396.979$  &$12103.096$ &$15131.248$ \\
 				$0.06$ & $4471.923$ & $6915.262$  &$9934.987$  &$18447.155$ &$22879.406$ \\
 				$0.07$ & $6640.513$ & $10263.074$ &$14819.331$ &$26974.749$ &$32808.194$ \\
 				$0.08$ & $11555.69$ & $18099.74$  &$25834.53$  &$45742.71$  &$55384.87$ \\
 				$0.09$ & $20332.71$ & $32633.35$  &$46290.22$  &$84578.12$  &$106325.07$ \\
 				$0.10$ & $37737.27$ & $59394.68$  &$84389.98$  &$152412.61$ & $194372.67$ \\
 				\hline \\ [-1.8ex] 		
 		\end{tabular} }
 	\end{center}
 	\caption{Simulated $100(1-\alpha)\%$ quantiles of $T(\eta)$}
 	\label{t1}
 \end{table}

 \begin{remark} \rm
 The main reason for the rather involved structure of $T(\eta)$ above is the  effect of the trimming parameter $\eta$. Indeed, if $\eta = 0$,
 \begin{align*}
 G(r;a,b|0) & = (b-r)^2 V_{1}(a,r|0) +(r-a)^2 V_{1}(r,b|0) - (b-r)(r-a)U_1(a,r;r, b) \\
 & = (b-r)^2 V_{1}(a,r|0) +(r-a)^2 V_{1}(r,b|0) \\
 & \hspace{3cm} - (b-r)(r-a)\{ V_{1}(a,b|0) - V_{1}(a,r|0) - V_{1}(r,b|0) \},
 \end{align*}
 which is identical to $G(r;a,b)$ in Theorem~\ref{cor:T0}.
 \end{remark}

  Next we present the asymptotic distribution under some local alternatives.
\begin{theorem}
\label{thm:alt} Suppose Assumption \ref{ass:main} holds. 
Assume that there exits $\phi \in (3\eta,1-3\eta) $ such that $ \mu_t = \mu^{*} $ for \change{$ t = 1,2,...,\lfloor \phi n \rfloor$} and $ \mu_t = \mu^{*} + \delta_n $ for \change{$ t = \lfloor \phi n \rfloor + 1,...,n$.} Then, 
\begin{itemize}
\item[1,] If $ n^{1/2} \| \delta_n \|_2 / \| \Gamma \|_F^{1/2} \rightarrow \infty $, then $T_n \rightarrow \infty$ in probability.
\item[2,] If $ n^{1/2} \| \delta_n \|_2 / \| \Gamma \|_F^{1/2} \rightarrow 0 $, then $T_n \rightarrow T$.
\item[3,]  If $ n^{1/2} \| \delta_n \|_2 / \| \Gamma \|_F^{1/2} \rightarrow c \in (0, \infty)$, then 
\begin{align*}
T_n  \overset{\mathcal{D}}{\longrightarrow} \sup_{r \in (3\eta, 1-3\eta)} \frac{ \widetilde{G}^2(r;0,1|\eta, \phi) }{ \int_{\eta}^{r-2\eta}  \widetilde{G}^2(u;0,r|\eta, \phi) du + \int_{r+\eta}^{1-2\eta} \widetilde{G}^2(u;r,1|\eta, \phi)  du},
\end{align*}
where 	$
\widetilde{G}(r;a,b|\eta, \phi) := \sqrt{2} G(r;a,b|\eta) + c \Delta (r;a,b|\eta, \phi),$ and $\Delta (r;a,b|\eta, \phi)$ is defined similarly to $G(r;a,b|\eta)$ but with $\triangledown_{u}(\cdot,\cdot|\eta, \phi)$, $\square_u (\cdot,\cdot; \cdot,\cdot)$ replacing all instances of $V_u(\cdot,\cdot|\eta), U_u(\cdot,\cdot;\cdot,\cdot)$ where we defined
\begin{multline*}
\square_u (a_1, a_2;b_1, b_2) =  \triangledown_u(a_1, b_2|\eta, \phi)  - \triangledown_u(a_1, b_1|\eta, \phi) - \triangledown_u(a_2, b_2|\eta, \phi) + \triangledown_u(a_2, b_1|\eta, \phi),
\end{multline*}
and 
\begin{align*}
\triangledown_u(a, b|\eta, \phi) = \left\lbrace  \begin{array}{ll}
\lim\limits_{n \rightarrow \infty}  \frac{2}{n^2} \sum_{i = \lfloor (\phi \vee a) n \rfloor}^{ \lfloor b n \rfloor - \lfloor \eta n \rfloor -1} \sum_{j = \lfloor (\phi \vee a) n \rfloor}^{ i }   w_{i,j}^u , & \text{ if }  \phi < b - \eta, \\
0, & \text{ otherwise }.
\end{array} \right. 
\end{align*}
\end{itemize}
\end{theorem}

 \begin{remark} \rm
 	If $\eta = 0$, we have
 	\begin{align*}
 	\Delta (r;a,b|0, \phi) &=  (b-r)^2 \triangledown_{1}(a,r|0, \phi) + (r-a)^2 \triangledown_{1}(r,b|0, \phi)  - (b-r)(r-a) \square_1(a,r;r,b) \\
 	& =  (b-r)^2 \triangledown_{1}(a,r|0, \phi) + (r-a)^2 \triangledown_{1}(r,b|0, \phi) \\
 	& \hspace{1cm} - (b-r)(r-a)\{ \triangledown_{1}(a,b|0, \phi) - \triangledown_{1}(a,r|0, \phi) - \triangledown_{1}(r,b|0, \phi) \}.
 	\end{align*}
 	It can be easily seen that $\triangledown_{1}(a,b|0, \phi) = (b-(\phi \vee a))^2 \1_{\{ \phi < b  \} }$. Then, some algebra show that
 	\begin{align*}
 	\Delta (r;a,b|0, \phi) = \left\lbrace  \begin{array}{ll}
 	(\phi-a)^2(b-r)^2, & \text{if } a < \phi \leq r, \\
 	(r-a)^2 (b-\phi)^2, & \text{if } r < \phi < b, \\
 	0, & \phi < a \text{ or } \phi >b.
 	\end{array} \right. 
 	\end{align*}
 	Thus, we have that  $ \Delta (r;a,b|0, \phi)$ is equal to $ \Delta(r,a,b) $ with $b^*=\phi$ defined in Theorem~\ref{cor:T_alt}.
 \end{remark}

\begin{remark} \rm
	It is quite straightforward to mimic the test we develop for the unsupervised case in the setting of high-dimensional independent data, and develop an SN-based test for multiple change points alternative in the high-dimensional time series setting. Details are omitted for the sake of brevity.
\end{remark}

\section{Wild binary segmentation and multiple change-point estimation}
\label{sec:estimation}

In practice, an important problem is to estimate the number and location of change points. A classical testing-based method is binary segmentation: run a test over the full sample, and if the test rejects the null, then split the sample into two segments (with the location of first change point estimated by the $k$ where the maximum is achieved in the test statistic), and then continue to test for change points for each segment. The algorithm stops when there is no rejection for each segment. A problem with binary segmentation is that it does not work well when there are multiple change points with changes exhibiting a non-monotonic pattern; see our simulation results. To overcome this drawback, {\cite{fryzlewicz2014wild}} proposed a new approach called Wild Binary Segmentation (WBS, hereafter). The main idea of WBS is to calculate the CUSUM statistic for many random sub-intervals to allow at least one of them to be localized around a change point (with high probability), so this change point can be identified. It  overcomes
the weakness of binary segmentation, where the  CUSUM statistic computed on the full sample is unsuitable for certain configurations
of multiple change-points. It seems natural to combine the WBS with our SN-based test statistic and see whether we can estimate the number and location of change points accurately.

\bigskip

We begin by introducing some additional notation. For arbitrary integers $4 \leq s + 3 \leq  e - 4 \leq n - 4 $ define
\[
Q(s,e) := \max_{b=s+3,...,e-4} \frac{D(b;s,e)^2}{V(b;s,e)},
\]
where $D(b;\ell,m)$ was defined in~\eqref{eq:defDn} and
\[
V(b;s,e) := \frac{1}{e-s+1}\Big( \sum_{t = s+1 }^{k-2} D(t;s,b)^2 + \sum_{t = k+2}^{e-2} D(t;b+1,e)^2 \Big).
\]
Note that $Q(s,e)$ is simply the statistic $T_n$ from~\eqref{eq:defTn} computed pretending that the available sample consists of $X_s,...,X_e$.

Now WBS-SN is applied as follows. Denote by $F_n^M$ a set of $M$ pairs of integers $(s_m,e_m)$ which satisfy {$1 \leq s_m < e_m \leq n$ and $e_m - s_m \geq L_0$} 
with numbers $s_m, e_m$ drawn uniformly from the set \change{$\{1,...,n\}$} (independently with replacement) {and $L_0$ denoting a minimal interval length}.
Given this sample, apply Algorithm 1 with initialization WBS-SN$(1,n,\xi_n,\change{L_0,F_n^M})$. Here, the threshold parameter $\xi_n$ is determined by simulations as follows: generate $R$ samples of i.i.d multivariate normal random variables with constant mean zero and identity covariance matrix, with the same $n$ and $p$ as \change{$Y_1,...,Y_n$}. For the $i^{th}$ sample, calculate
\[
\hat{\xi}^{i}_n = \max_{m =1,...,M} Q(s_m,e_m), \quad i=1,...,R.
\]
Given the $R$ values $\hat{\xi}^{i}_n,i=1,...,R$ above pick $\xi_n$ as the $95\%$ quantile of the values \change{$\hat{\xi}^{1}_n,...,\hat{\xi}^{R}_n$}.  Since the SN test statistic is
asymptotically pivotal, this threshold is expected to well approximate the 95\% quantile of
the finite sample distribution of the maximum SN test statistic on the $M$ random intervals
under the null. The detailed algorithm is presented below.


\begin{algorithm}[H]
	\SetKwProg{Fn}{Function}{:}{}
	\Fn{WBS-SN(s,e,$\xi_n$,\change{$L_0,F_n^M$})}{
		\uIf {$e-s < L_0$ }{
			STOP;
		}\Else{
			\change{$\mathcal{M}_{s,e}:=$ set of those indices $m$ for which $[s_m, e_m] \in F_n^M$ is such that $[s_m, e_m] \subseteq [s,e]$ and $e_m-s_m\ge L_0$} \\ $m_0:=\arg \max_{m \in \mathcal{M}_{s,e}}Q(s_m,e_m)$ \\ \uIf {$Q(s_m,e_m) > \xi_n$}{ add $b_0 := \arg \max_{b} D(b;s_{m_0},e_{m_0})/V(b;s_{m_0},e_{m_0}) $ to set of estimated CP \\
				WBS-SN($s,b_0,\xi_n$,\change{$L_0,F_n^M$}) \\
				WBS-SN($b_0+1,e,\xi_n$,\change{$L_0,F_n^M$})
			}\Else{
				STOP
			}
		}
	}
	\caption{WBS-SN for independent data}
	\label{alg:WBS}
\end{algorithm} 

The same approach can be applied to multiple change point detection for high-dimensional time series, with an incorporation of a trimming parameter in our SN-based test statistic. To obtain the threshold $\xi_n$, we can apply the same $M$ random intervals and the trimmed SN-based test statistic with the same trimming parameter $\tau$ to i.i.d standard normal distributed data with the same $(n,p)$, as done for the independent data case. { Similarly, we also adopt a bound $L_0$ for the minimal length of random intervals which now depends on $\eta, n$. Some investigations of the sensitivity with respect to the choice of $L_0$ and some practical recommendations are provided in the simulation section in the supplement.   
}

\begin{algorithm}[H]
	\SetKwProg{Fn}{Function}{:}{}
	\Fn{WBS-SN$(s,e,\change{\xi_n,L_0,F_n^M, \tau})$}{
		\uIf{$e-s < L_0$}{
			STOP
		}\Else{
			$\mathcal{M}_{s,e}:=$ set of those indices $m$ for which $[s_m, e_m] \in F_n^M$ is such that $[s_m, e_m] \subseteq [s,e]$ and $e_m-s_m\ge L_0$ \\ \change{$(m_0, b_0):= \arg \max\limits_{m \in \mathcal{M}_{s,e}, b\in \{ s_m+3\tau+3, ..., e_m-3\tau-4 \} } \frac{D^2(b;s_m,e_m|\tau)}{ W_{n}(b;s_m,e_m|\tau) } $} \\ \uIf{$\frac{D^2(b_0;s_{m_0},e_{m_0}|\tau)}{ W_{n}(b_0;s_{m_0},e_{m_0}|\tau) } > \change{\xi_n}$}{ add $b_0$ to the set of estimated change points \\ WBS-SN$(s,b_0, \change{\xi_n,L_0,F_n^M, \tau})$ \\ WBS-SN$(b_0+1, e, \change{\xi_n,L_0,F_n^M, \tau})$ }\Else{ STOP } } } \caption{WBS-SN for time series} \label{algorithm:wbs} \end{algorithm}

\section{Numerical Results}
\label{sec:simulation}

In this section, we examine the finite sample performance of our proposed tests and estimation methods via simulations. In Section~\ref{sub:mean}, we present the size and power for our SN-based test in comparison with Kolmogorov-Smirnov type test for a single change point in high-dimensional independent data  and also examine the behavior of the test developed for the unsupervised case. In Section~\ref{sub:meanhdts}, we show the size and power for the test for a single change point in the mean of high-dimensional time series. Section~\ref{sub:estimate} and Section~\ref{sub:estimateTS} in the supplement contain the WBS-based estimation result in comparison with some existing methods for independent and dependent data, respectively.


	\subsection{Testing for high-dimensional independent data}
\label{sub:mean}

	In this subsection we  investigate the finite sample behavior of our test statistic for a mean shift. We shall first focus on the supervised case, i.e., under the alternative that there is one change point in the mean.   Consider the data generating process
\[
Y_t = {\delta} \1_{t > 0.5n} + \epsilon_t, \textrm{ for all } \change{t = 1,2,...,n,}
\]
	where ${\delta}$ is a p-dimensional vector representing the mean shift, and $\{\epsilon_t\}$ are i.i.d samples from multivariate normal distribution, with common mean ${0}$ and covariance matrix $\Sigma$. Under the null hypothesis where there is no change point, it is equivalent to the case that ${\delta} = {0}$, whereas under the alternative (there is one change point), we let ${\delta} = \kappa \change{(1,1,...,1)}^T$ with $\kappa \in \{0.1, 0.2\}$. For $\Sigma$, we consider four scenarios:
\begin{enumerate}
	\item [a)] Independent. $\Sigma=\change{I_p}$ (i.e., identity matrix).  
\item[b)] AR(1)-type correlation. The $(i,j)$ element in $\Sigma$ is $\sigma_{ij}=0.5^{|i-j|}$.
\item[c)] Banded. Specifically, the main diagonal elements are all 1. The first off-diagonal elements are all 0.5 and the second off-diagonal elements are all 0.25. All other elements are  zero.
\item[d)] Compound Symmetric. The main diagonal elements are all 1 and all remaining elements are  0.5.
\end{enumerate}	
We also tried non-Gaussian errors, where $\epsilon_t=\Sigma^{1/2}\widetilde{\epsilon}_t$, where $\widetilde{\epsilon}_t$ have i.i.d components with scaled $t(3)$ distribution that have mean zero and variance one.  
We let $p \in \{100, 200, 500\}$ and $n \in \{100,200,500\}$.

We shall formulate an extension of the classical Kolmogorov-Smirnov (KS) test statistic in the current context and compare with SN-based test via simulations. Let $\widehat{k}=\change{\mbox{argmax}_{k=2,...,n-2}} D(k;1,n)^2$ which is an estimate of change point location without self-normalization. We can then define an estimator of $\|\Sigma\|_F^2$ using the Jackknife-based approach as presented on page 814 of \cite{chen2010two} in two ways. On one hand, we can obtain a pre-break estimate and a post-break estimate of $\|\Sigma\|_F^2$ and then take the average of them, i.e.,
 \begin{multline*}
\widehat{\|\Sigma\|_{F,1}^2} = \frac{1}{2\widehat{k}(\widehat{k}-1)} tr\Big\{\sum_{j_1\not=j_2}^{\widehat{k}} (X_{j_1}-\bar{X}_{(j_1,j_2);1:\widehat{k}}) X_{j_1}^T (X_{j_2}-\bar{X}_{(j_1,j_2);1:\widehat{k}})X_{j_2}^T\Big\}
\\
+ \frac{1}{2(n-\widehat{k})(n-\widehat{k}-1)} tr\Big\{\sum_{j_1\not=j_2,\widehat{k}+1}^{n} (X_{j_1}-\bar{X}_{(j_1,j_2);(1+\widehat{k}):n}) X_{j_1}^T (X_{j_2}-\bar{X}_{(j_1,j_2);(1+\widehat{k}):n})X_{j_2}^T\Big\},
 \end{multline*}
 where $\bar{X}_{(j_1,j_2);a:b}$ denotes the average of the sample \change{$X_a,...,X_b$} without $X_{j_1}$ and $X_{j_2}$.
 On the other hand, we can form a demeaned sample by substracting $\bar{X}_{1:\widehat{k}}$ from \change{$(X_1,...,X_{\widehat{k}})$}
 and $\bar{X}_{(\widehat{k}+1):n}$ from \change{$(X_{\widehat{k}+1},...,X_n)$}, and then apply the jackknifed based estimator to the full demeaned sample; we denote the resulting estimator by
 $\widehat{\|\Sigma\|_{F,2}^2}$. Then we can define the following two statistics
\change{\[
KS_{n,1}=\frac{\sup_{k=2,...,n-2}D(k;1,n)^2}{n^6\widehat{\|\Sigma\|_{F,1}^2}},~
 KS_{n,2}=\frac{\sup_{k=2,...,n-2}D(k;1,n)^2}{n^6\widehat{\|\Sigma\|_{F,2}^2}}.
\]}
To facilitate the comparison, we also introduce an infeasible version,
\change{\[
KS_{n,Inf}=\frac{\sup_{k=2,...,n-2}D(k;1,n)^2}{n^6\|\Sigma\|_F^2}.
\]}
The limiting null distributions of the above three statistics are expected to be $\sup_{r\in [0,1]}|G(r;0,1)|^2$, the critical values of which can be obtained by simulations. It is worth noting that the limiting null for the infeasible test statistic can be easily derived from our
Theorem~\ref{cor:T0}.

  Below we compare four tests, $T_{n}$, $KS_{n,1}$, $KS_{n,2}$, $KS_{n,Inf}$ and $EH$ based on 5000 Monte Carlo replications with the nominal level $0.05$. Here $EH$ refers to the adaptive change point test developed by \cite{enikeeva2013high}, which requires Gaussian and independent components assumptions.  Table~\ref{tab:singleGaussian} below shows the rejection rate in percentage under $\mathcal{H}_0:\kappa=0$, $\mathcal{H}_{1,1}:\kappa=0.1$ and $\mathcal{H}_{1,2}:\kappa=0.2$ for Gaussian errors and Table~\ref{tab:singleNONGaussian} is for the non-Gaussian case.

\centerline{Please insert Table~\ref{tab:singleGaussian} here!}

\centerline{Please insert Table~\ref{tab:singleNONGaussian} here!}

  The above simulation results demonstrate that  when the error is Gaussian, (1) SN-based test has accurate size for independent,  AR(1) and Banded correlation models, whereas the test appears quite distorted in the compound symmetric case. This finding is not surprising as the compound symmetric case violates the theoretical assumptions imposed (see Assumption \ref{ass}), whereas independent,  AR(1) and Banded cases satisfy those assumptions. In a sense, this shows that our (weak componentwise dependence) assumptions are to a certain extent necessary. The KS tests (both infeasible and feasible ones) show similar size behavior except that they are noticeably undersized for $n=100$ case, and their size distortion in the compound symmetric case is even greater than our test. The test by \cite{enikeeva2013high} exhibits size distortion for all cases (undersized for independent case, and oversized for AR(1) and bounded correlation models) and its size for compound symmetric case is way too high. When the error is nonGaussian, our SN-based test and all KS tests appear to have similar rejection rates as the Gaussian case, indicating the robustness of our SN-based test with respect to heavy tailed errors. 
   By contrast, the size for EH 
  in the non-Gaussian case is very high, implying the sensitivity/non-robustness of their test with respect to non-Gaussianity. 
   
  A comparison of the powers for SN-based and KS tests shows that our test is very comparable to all three KS tests, which perform similarly.  Overall the finite sample size and power performance of four tests (SN and three KS tests) are very much comparable with no single test dominating others. Note that the feasible KS tests assume there is one change point, and it may perform very poorly when there are more than one change-point (results not shown). Methodologically, it seems desirable to develop a test that does not involve explicit estimation of change points, which is itself a difficult problem, especially when there are multiple change points. The power of EH is hard to interpret given its distorted size, and we shall not look into the size-adjusted power as we would not recommend EH test for nonGaussian and cross-sectionally dependent
 high-dimensional independent data.

We further  examine the finite sample performance of the test we develop for the unsupervised case (i.e., there could be  multiple change points under the alternative), in comparison with the SN-based test aimed for one change point only. Three different data generating processes are considered below: \\

\noindent ($\mathcal{H}_{1,1}$) (one change-point alternative): $\mu_t = {\delta}\1_{t/n > 1/2}$;\\

\noindent ($\mathcal{H}_{1,2}$) (two change-point alternative): $\mu_t = {\delta}\1_{ t/n > 1/3} -{\delta} \1_{t/n > 2/3}$;\\

\noindent ($\mathcal{H}_{1,3}$) (three change-point alternative): $\mu_t = {\delta}\1_{t/n > 1/4} - {\delta}\1_{3/4 \geq t/n > 1/2}$;\\

Under the null hypothesis, ${\delta} = 0$, whereas under the alternative we let ${\delta} = \change{(0.2,0.2,...,0.2)}^T$.  Following the practice in \cite{zhang2018unsupervised},  we set $\epsilon = 0.1$. The empirical rejection rates (in percentage) are summarized in Table \ref{tab:unsupervised} below for several combinations of $(n,p)$, where we denote the statistic developed for the supervised case as $T_n$ and for the unsupervised case as $T_n^\diamond$.

\centerline{Please insert Table~\ref{tab:unsupervised} here!}

From Table~\ref{tab:unsupervised}, we can observe that $T_n$ have empirical rejection rates close to $5\%$ under the null for all cases except for compound symmetric case, and $T_n^\diamond$ exhibits quite a bit distortion when $n=100, 200$ and its size appears accurate for $n=500$ for the independent,  AR(1) and banded cases. When the error has compound symmetric covariance, the  size distortion for $T_n^\diamond$ is considerably higher than that for $T_n$, showing the difficulty brought by the strong componentwise dependence. Under the alternative, we can see that the supervised  test statistic has much higher power  in the single change point case, but the power lost drastically when there are two or three change points, suggesting the inability of the supervised test that targets one change point to accommodate more than one. By contrast, the unsupervised test still preserves reasonable amount of power, which is consistent with our theory. The results for the non-Gaussian case are qualitatively similar so are not included here to conserve space.


	\subsection{Testing for high-dimensional time series}
\label{sub:meanhdts}

We consider the following single change point model. 
\begin{example} 
	\label{exp:1}
	Consider the following VAR(1) model,
	\begin{align*}
	Y_{t} - {\delta} \1_{\{ t > 0.5n \}} = \rho (Y_{t-1} - {\delta} \1_{\{ t > 0.5n \}}) + \epsilon_{t},
	\end{align*}
	where  $\{ \epsilon_{t} \}$ are the temporally independent errors and we consider $\rho \in \{0.2, 0.5, 0.7, -0.5\}$. Under the null hypothesis, ${\delta}=0$. Under the alternative hypothesis, we examine the following two types of mean shift, i.e.,
	\begin{itemize}
		\item[(i)] Homogeneous alternative:  $ {\delta}^T = 0.1\change{(1,1,...,1)} .$ 
		\item[(ii)] Inhomogeneous alternative: $$ {\delta}^T = 0.2\change{(\delta_1,...,\delta_p)}, $$
		where $ \change{(|\delta_1|,...,. |\delta_p|)} \overset{i.i.d}{\sim} Uniform(0,1) $ and the signs of $\change{(\delta_1,..., \delta_p)}$ are randomly sampled with equal probability.
	\end{itemize}
	Also, for the innovations $ \{ \epsilon_{t} \} $, we consider the following two scenarios
	\begin{itemize}
		\item[(a)] Gaussian errors with AR(1) type convariance structure: 	$ \epsilon_{t} \overset{i.i.d}{\sim} N(0, \Sigma_{\epsilon})$, where $\Sigma_{\epsilon} = (0.5^{|i-j|})_{i,j=1}^p$.
		\item[(b)] Non-Gaussian errors: $\{ \epsilon_{t} \}_{t=1}^n$ are i.i.d and each entry of $\epsilon_t= (\epsilon_{t, 1}, \change{...},\epsilon_{t, p})^T$ is generated independently from $Uniform[-\sqrt{3},\sqrt{3}]$. 
	\end{itemize}
\end{example}
To illustrate the finite sample performance of  our statistic $T_n$, we compare with the methods described in \cite{horvath2012change} (denoted as HH) and the double CUSUM binary segmentation algorithm (denoted as DCBS) [\cite{cho2016change}]. For HH, it works for independent panel time series, targets dense alternative and involves a bandwidth parameter $h$, which is  used in  the kernel estimator of long run variance. For DCBS, it contains several tuning parameters and requires the use of  bootstrap. We shall implement DCBS using the  R package ``hdbinseg" and the default tuning parameter values. Note that the method proposed by \cite{jirak2015uniform} targets sparse alternative in the mean of high-dimensional time series, so is not included in our comparison. 

\centerline{Please insert Table~\ref{tab:sizehdts} here!}

As can be seen from Table~\ref{tab:sizehdts}, the size of our statistic $T_n$ can depend on the amount of trimming $\eta$, magnitude and sign of temporal dependence $\rho$, sample size $n$ and the dimension $p=2n$. When the temporal dependence is weak, i.e.,  $\rho=0.2$, the size is fairly accurate for both trimming levels ($\eta=0.02$ and $0.05$) and all sample sizes ($n=200, 400, 800$). As the temporal dependence gets stronger, especially for $\rho=0.7$, we see some size distortion at small sample size $n=200$ for both trimming levels, but the size distortion is much reduced with larger sample size $n=400, 800$. The above comment applies to both Gaussian model (a) and non-Gaussian setting (b). By contrast,  the rejection rates of DCBS are almost always equal to zero. This may be due to the default tuning parameters used in ``hdbinseg", which aim to make Type I error zero in large samples to be  consistent with the consistency results stated in 
\cite{cho2016change}. The HH method is apparently oversized in all settings, which is presumably 
 due to the cross-sectional dependence. Therefore the size results demonstrate the decent approximation our limiting null distribution (under fixed-$\eta$ asymptotics) is able to provide and shows its practical usefulness in accommodating weak dependence across panel and over time. 
 
 \centerline{Please insert Table~\ref{tab:powerhdts} here!}
 
 The power results are collected in Table~\ref{tab:powerhdts}. Our test exhibits quite reasonable power, which could depend on the choice of trimming parameter, whereas HH method's raw power is high due to the (sometimes severe) oversize under the null and 
 DCBS exhibits lower power, which is presumably due to the large critical values used to control Type I error (to make it zero in large sample).

\section{Summary and Conclusion}
\label{sec:conclusion}

In this paper, we propose a non-parametric methodology to testing and estimation of change-points in the mean of a sequence of high-dimensional data. Our methodological developments start with the relatively simple testing problem: testing for  one change-point in the mean of high dimensional independent data, by marrying  the self-normalization idea in \cite{shao2010testing} and U-statistic based approach of \cite{chen2010two}  for high dimensional two-sample testing. Our test differs from most existing ones in the literature by targeting the dense alternative,  allowing weak cross-sectional dependence, and imposing no particular rate constraints on the dimension $p$ as a function of sample size $n$. It is worth noting that our test does not involve a tuning parameter and is based on critical values tabulated in the paper, which could be appealing for practitioners. 

On the testing front, several extensions were pursued in the paper, including (1)  change point testing in  the presence of multiple change points in mean; (2) testing for a change-point in covariance matrix assuming zero mean; (3) change point testing for the mean of high-dimensional time series. In particular, the extension to high-dimensional time series is highly nontrivial and theoretically challenging. To attenuate the bias caused by weak temporal dependence, we introduce a trimmed U-statistic and adopt the fixed-$b$ asymptotic framework [\cite{kiefer2005new}] to derive the limiting null distribution of the self-normalized test statistic, which appears to approximate the finite sample distribution well for a broad range of time series dependence, as demonstrated in the simulations. On the estimation front, we propose to combine the idea of wild binary segmentation [\cite{fryzlewicz2014wild}] with our SN-based test to estimate the number and location of change points. Simulations show that our method can be more effective when the mean shift is dense as compared to the INSPECT algorithm [\cite{wang2018high}] for high-dimensional independent data, and is at least comparable to the procedures used in \cite{cho2016change} and \cite{li2019change} for high-dimensional time series. On the theory front, we show the weak convergence of the sequential U-statistic based processes for both independent and dependent high-dimensional data, which can be of independent interest.

There are a number of topics that are worth investigating. Firstly, it would be interesting to extend the asymptotic theory for high-dimensional time series to a more general setting, such as nonlinear causal process [\cite{wu2005nonlinear}]; see \cite{wang2019hypothesis} for a recent extension of SN to high-dimensional time series under the framework of nonlinear causal process. Secondly, while we consider a shift in mean in this paper, it is also of great value to study change point detection for other high-dimensional parameters, such as the vector of marginal quantiles; see \cite{shao2010testing} for a more general framework but in a low-dimensional time series setting.  Thirdly, selecting the trimming parameter $\tau$ for real applications can be nontrivial and it would be interesting to develop a data-driven procedure to be adaptive to the magnitude of temporal dependence. Lastly, there is no theory available for the WBS-SN method used here. It would be interesting to provide some theoretical justifications, as done in \cite{fryzlewicz2014wild} in a much simpler setting, and this seems very challenging. Further research along some of these directions is well underway.

\bigskip

{\bf Acknowledgements:} We would like to thank the three reviewers for their constructive comments, which led to substantial improvements. We are grateful to Dr. Farida Enikeeva for sending us the code used in Enikeeva and Harchaoui (2019). Shao's research is partially supported	by NSF-DMS 1807023 and NSF-DMS-2014018. Vogulshev's research is partially supported by a discovery grant from NSERC of Canada. 

\begin{supplement}
	\stitle{Supplement to ``Inference for Change Points in High-Dimensional Data via Self-normalization"}
	\sdescription{The supplementary material contains all the proofs for theoretical results stated in the paper. Additional simulation results are also included. }
\end{supplement}

\bibliography{hdcp}

\newpage

\begin{footnotesize}
	\begin{table}[h]
		\resizebox*{!}{0.98\textheight}{
			\begin{tabular}{ccc|ccccc||ccccc}
				\hline
				&&&\multicolumn{5}{c||}{ID}&\multicolumn{5}{c}{AR(1)}\\
				\cline{4-13}
				&$n$&$p$&$T_n$&$KS_{n,Inf}$&$KS_{n,1}$&$KS_{n,2}$&$EH$&$T_n$&$KS_{n,Inf}$&$KS_{n,1}$&$KS_{n,2}$&$EH$\\\hline
				\multirow{9}{*}{$\mathcal{H}_0$}&\multirow{3}{*}{$100$}&$100$&5.6&2.2&2.3&2.6&1.7&6.3&3.3&3.6&3.7&10.8\\
				&&$200$&4.9&3.4&3.3&3.3&1.4&4.7&3.1&2.9&2.9&11.7\\
				&&$500$&5.3&2.1&2.2&2.0&1.1&6.1&3.3&3.4&3.2&10.8\\\cline{2-13}
				&\multirow{3}{*}{$200$}&$100$&5.8&4.0&4.0&4.3&1.2&5.9&4.2&4.2&4.2&9.4\\
				&&$200$&5.1&4.3&4.4&4.6&0.4&4.6&3.1&3.2&3.4&11.3\\
				&&$500$&6.0&3.7&3.6&3.6&0.8&5.8&3.7&3.9&3.8&10.8\\\cline{2-13}
				&\multirow{3}{*}{$500$}&$100$&6.3&4.9&5.0&5.1&0.5&5.8&6.7&7.0&7.0&10.4\\
				&&$200$&6.2&5.3&5.6&5.6&0.7&5.6&5.0&4.9&5.0&9.8\\
				&&$500$&6.0&4.5&4.3&4.3&0.4&6.2&4.7&4.5&4.5&9.2\\\hline
				\multirow{9}{*}{$\mathcal{H}_{1,1}$}&\multirow{3}{*}{$100$}&$100$&34.5&30.0&30.0&30.5&11.4&27.0&27.2&27.8&28.5&31.1\\
				&&$200$&51.9&49.8&49.5&49.4&24.4&37.4&37.4&38.1&38.0&44.5\\
				&&$500$&82.5&85.5&85.6&84.7&64.5&64.8&65.7&66.3&65.3&72.1\\\cline{2-13}
				&\multirow{3}{*}{$200$}&$100$&77.5&81.5&82.0&82.1&42.6&61.4&62.1&62.1&62.1&59.0\\
				&&$200$&94.7&96.3&96.3&96.5&81.2&79.3&83.1&83.8&83.6&79.6\\
				&&$500$&100.0&100.0&100.0&100.0&99.9&98.5&99.3&99.3&99.3&99.0\\\cline{2-13}
				&\multirow{3}{*}{$500$}&$100$&99.8&100.0&100.0&100.0&99.5&97.9&98.3&98.3&98.3&98.2\\
				&&$200$&100.0&100.0&100.0&100.0&100.0&99.9&99.9&99.9&99.9&99.9\\
				&&$500$&100.0&100.0&100.0&100.0&100.0&100.0&100.0&100.0&100.0&100.0\\\hline
				\multirow{9}{*}{$\mathcal{H}_{1,2}$}&\multirow{3}{*}{$100$}&$100$&99.3&100.0&100.0&100.0&97.8&100.0&100.0&100.0&100.0&93.9\\
				&&$200$&100.0&100.0&100.0&100.0&100.0&100.0&100.0&100.0&100.0&99.9\\
				&&$500$&100.0&100.0&100.0&100.0&100.0&100.0&100.0&100.0&100.0&100.0\\\cline{2-13}
				&\multirow{3}{*}{$200$}&$100$&100.0&100.0&100.0&100.0&100.0&100.0&100.0&100.0&100.0&100.0\\
				&&$200$&100.0&100.0&100.0&100.0&100.0&100.0&100.0&100.0&100.0&100.0\\
				&&$500$&100.0&100.0&100.0&100.0&100.0&100.0&100.0&100.0&100.0&100.0\\\cline{2-13}
				&\multirow{3}{*}{$500$}&$100$&100.0&100.0&100.0&100.0&100.0&100.0&100.0&100.0&100.0&100.0\\
				&&$200$&100.0&100.0&100.0&100.0&100.0&100.0&100.0&100.0&100.0&100.0\\
				&&$500$&100.0&100.0&100.0&100.0&100.0&100.0&100.0&100.0&100.0&100.0\\\hline
				&&&\multicolumn{5}{c||}{BD}&\multicolumn{5}{c}{CS}\\
				\hline
				\multirow{9}{*}{$\mathcal{H}_0$}&\multirow{3}{*}{$100$}&$100$&5.8&3.5&3.5&3.5&10.2&11.4&11.4&12.6&12.4&89.2\\
				&&$200$&4.6&3.1&3.2&3.2&10.8&9.0&10.2&12.4&12.4&94.8\\
				&&$500$&5.5&3.4&3.4&3.3&10.3&9.5&11.7&14.1&14.0&97.9\\\cline{2-13}
				&\multirow{3}{*}{$200$}&$100$&6.1&3.7&3.7&3.8&7.9&10.3&13.6&15.9&15.7&92.6\\
				&&$200$&4.3&3.3&3.4&3.2&9.3&11.9&13.9&15.4&14.9&96.3\\
				&&$500$&5.5&4.1&4.2&4.2&9.6&10.3&13.1&14.9&14.2&98.5\\\cline{2-13}
				&\multirow{3}{*}{$500$}&$100$&6.3&6.7&6.9&7.0&9.5&12.3&16.5&16.8&16.8&97.2\\
				&&$200$&6.3&5.5&5.6&5.7&7.9&10.4&14.5&15.1&14.6&98.5\\
				&&$500$&5.7&4.7&4.7&4.7&7.2&14.1&18.1&18.3&18.2&99.4\\\hline
				\multirow{9}{*}{$\mathcal{H}_{1,1}$}&\multirow{3}{*}{$100$}&$100$&26.5&24.8&25.2&25.2&29.4&18.0&19.3&20.7&20.7&91.1\\
				&&$200$&35.3&36.6&36.7&36.7&42.3&18.1&18.0&20.2&20.2&95.7\\
				&&$500$&67.1&66.4&66.3&65.8&70.6&18.5&17.5&19.5&20.5&98.0\\\cline{2-13}
				&\multirow{3}{*}{$200$}&$100$&60.7&63.2&62.9&63.9&56.1&26.9&28.4&30.5&30.5&94.6\\
				&&$200$&82.9&87.1&87.2&87.2&84.5&27.3&29.3&30.9&30.6&97.6\\
				&&$500$&99.1&99.6&99.6&99.6&99.5&26.4&27.8&28.6&29.0&99.1\\\cline{2-13}
				&\multirow{3}{*}{$500$}&$100$&98.7&99.4&99.4&99.4&98.0&48.1&51.6&51.0&52.0&98.3\\
				&&$200$&100.0&100.0&100.0&100.0&100.0&47.0&49.7&50.2&50.2&99.4\\
				&&$500$&100.0&100.0&100.0&100.0&100.0&46.1&48.4&49.2&49.9&99.7\\\hline
				\multirow{9}{*}{$\mathcal{H}_{1,2}$}&\multirow{3}{*}{$100$}&$100$&95.1&95.8&95.7&95.8&95.3&38.6&38.3&39.6&40.1&94.6\\
				&&$200$&99.7&99.9&99.9&99.9&99.9&41.3&38.4&40.2&40.2&98.1\\
				&&$500$&100.0&100.0&100.0&100.0&100.0&38.4&38.4&40.4&40.8&99.3\\\cline{2-13}
				&\multirow{3}{*}{$200$}&$100$&100.0&100.0&100.0&100.0&100.0&61.8&62.6&63.2&63.3&98.2\\
				&&$200$&100.0&100.0&100.0&100.0&100.0&60.0&62.9&63.7&63.5&99.5\\
				&&$500$&100.0&100.0&100.0&100.0&100.0&62.5&63.9&63.8&63.9&99.9\\\cline{2-13}
				&\multirow{3}{*}{$500$}&$100$&100.0&100.0&100.0&100.0&100.0&90.2&92.1&91.8&91.8&100.0\\
				&&$200$&100.0&100.0&100.0&100.0&100.0&91.4&93.0&93.7&93.8&100.0\\
				&&$500$&100.0&100.0&100.0&100.0&100.0&89.9&91.2&91.4&91.5&99.9\\\hline
		\end{tabular}}
		\caption{Empirical Rejection Rates (in percentage) for One Change Point in Mean (Gaussian Error)}\label{tab:singleGaussian}
	\end{table}
\end{footnotesize}

\clearpage
\newpage

\begin{table}[h!]
	\resizebox*{!}{0.98\textheight}{
		\begin{tabular}{ccc|ccccc||ccccc}
			\hline
			&&&\multicolumn{5}{c||}{ID}&\multicolumn{5}{c}{AR(1)}\\
			\cline{4-13}
			&$n$&$p$&$T_n$&$KS_{n,Inf}$&$KS_{n,1}$&$KS_{n,2}$&$EH$&$T_n$&$KS_{n,Inf}$&$KS_{n,1}$&$KS_{n,2}$&$EH$\\\hline
			\multirow{9}{*}{$\mathcal{H}_0$}&\multirow{3}{*}{$100$}&$100$&5.0&3.7&3.3&3.1&84.3&7.0&3.9&3.2&3.0&76.9\\
			&&$200$&5.4&3.6&2.3&2.9&97.1&5.7&3.4&3.3&2.9&92.9\\
			&&$500$&5.2&2.7&2.4&1.7&100.0&5.2&2.2&2.3&1.9&100.0\\\cline{2-13}
			&\multirow{3}{*}{$200$}&$100$&5.5&4.8&4.8&4.3&84.6&5.4&4.5&4.5&4.6&78.1\\
			&&$200$&5.1&4.0&4.3&4.2&97.3&6.1&4.6&3.9&4.3&93.5\\
			&&$500$&6.2&4.2&3.9&3.7&100.0&6.4&4.3&3.8&3.8&99.7\\\cline{2-13}
			&\multirow{3}{*}{$500$}&$100$&4.1&4.8&4.7&5.0&88.1&4.9&5.5&5.4&5.9&81.9\\
			&&$200$&5.2&3.6&3.4&3.4&97.9&6.4&5.4&5.2&5.4&95.6\\
			&&$500$&5.4&3.4&3.6&3.3&100.0&5.6&3.7&4.1&4.2&99.9\\\hline
			\multirow{9}{*}{$\mathcal{H}_{1,1}$}&\multirow{3}{*}{$100$}&$100$&35.0&29.8&33.2&32.8&85.6&28.8&25.6&26.9&27.5&81.1\\
			&&$200$&54.2&52.1&52.5&51.6&96.9&40.1&36.5&38.2&38.1&94.5\\
			&&$500$&87.4&87.1&87.2&84.9&100.0&66.4&67.9&67.5&65.9&100.0\\\cline{2-13}
			&\multirow{3}{*}{$200$}&$100$&75.9&79.2&80.4&80.0&100.0&58.8&60.2&61.4&61.9&88.5\\
			&&$200$&94.2&97.4&97.3&97.0&99.5&80.2&83.9&83.7&83.8&98.5\\
			&&$500$&100.0&100.0&100.0&99.9&100.0&98.3&99.1&98.8&98.6&100.0\\\cline{2-13}
			&\multirow{3}{*}{$500$}&$100$&99.9&100.0&100.0&99.9&100.0&97.4&98.9&98.7&98.7&99.5\\
			&&$200$&100.0&100.0&100.0&100.0&100.0&100.0&100.0&100.0&100.0&100.0\\
			&&$500$&100.0&100.0&100.0&99.9&100.0&100.0&100.0&100.0&99.9&100.0\\\hline
			\multirow{9}{*}{$\mathcal{H}_{1,2}$}&\multirow{3}{*}{$100$}&$100$&99.2&99.6&99.4&99.2&99.3&93.4&93.9&93.9&93.9&97.5\\
			&&$200$&99.8&100.0&99.9&99.6&100.0&99.5&99.8&99.5&99.2&99.9\\
			&&$500$&100.0&100.0&100.0&100.0&100.0&100.0&100.0&100.0&100.0&100.0\\\cline{2-13}
			&\multirow{3}{*}{$200$}&$100$&100.0&100.0&100.0&99.9&100.0&99.9&100.0&99.9&99.9&100.0\\
			&&$200$&100.0&100.0&100.0&100.0&100.0&100.0&100.0&100.0&100.0&100.0\\
			&&$500$&100.0&100.0&100.0&100.0&100.0&100.0&100.0&100.0&100.0&100.0\\\cline{2-13}
			&\multirow{3}{*}{$500$}&$100$&100.0&100.0&100.0&100.0&100.0&100.0&100.0&100.0&100.0&100.0\\
			&&$200$&100.0&100.0&100.0&100.0&100.0&100.0&100.0&100.0&100.0&100.0\\
			&&$500$&100.0&100.0&100.0&100.0&100.0&100.0&100.0&100.0&100.0&100.0\\\hline
			&&&\multicolumn{5}{c||}{BD}&\multicolumn{5}{c}{CS}\\
			\hline
			\multirow{9}{*}{$\mathcal{H}_0$}&\multirow{3}{*}{$100$}&$100$&6.7&3.7&2.9&3.0&76.4&11.8&10.6&14.4&14.1&90.6\\
			&&$200$&5.2&3.9&3.4&3.0&93.0&11.2&11.5&14.0&13.6&97.2\\
			&&$500$&4.7&2.4&2.3&1.8&100.0&11.9&12.6&15.3&15.3&99.5\\\cline{2-13}
			&\multirow{3}{*}{$200$}&$100$&5.9&4.4&4.2&4.7&77.0&12.3&15.1&16.1&16.1&94.4\\
			&&$200$&6.0&4.1&4.1&4.2&93.4&12.3&15.5&16.2&16.6&98.2\\
			&&$500$&5.5&4.3&4.0&3.9&99.9&12.2&13.3&15.3&15.1&100.0\\\cline{2-13}
			&\multirow{3}{*}{$500$}&$100$&4.6&5.4&5.1&5.5&81.7&11.5&15.9&16.1&15.7&96.5\\
			&&$200$&5.9&5.3&5.0&5.1&95.5&12.0&15.1&15.9&16.0&99.2\\
			&&$500$&6.0&3.4&3.7&3.8&99.9&12.9&16.5&16.8&17.1&100.0\\\hline
			\multirow{9}{*}{$\mathcal{H}_{1,1}$}&\multirow{3}{*}{$100$}&$100$&29.3&24.9&26.3&26.3&80.5&19.6&19.9&21.4&22.1&92.1\\
			&&$200$&40.6&35.9&38.5&38.1&94.6&17.9&17.5&18.4&18.8&97.5\\
			&&$500$&66.9&67.8&67.4&66.7&100.0&20.1&20.1&22.2&22.4&99.8\\\cline{2-13}
			&\multirow{3}{*}{$200$}&$100$&60.5&61.3&61.8&62.1&87.4&25.0&27.1&27.3&27.3&96.3\\
			&&$200$&81.3&84.9&84.5&84.2&98.7&28.8&30.4&31.4&31.4&98.8\\
			&&$500$&98.4&99.2&99.3&98.9&100.0&26.8&28.0&29.7&29.1&99.9\\\cline{2-13}
			&\multirow{3}{*}{$500$}&$100$&97.8&99.4&99.1&99.1&99.7&46.9&49.5&49.7&50.0&98.6\\
			&&$200$&99.9&100.0&100.0&100.0&100.0&44.9&48.6&49.1&49.3&99.4\\
			&&$500$&100.0&100.0&100.0&99.9&100.0&44.4&47.4&48.6&48.9&99.9\\\hline
			\multirow{9}{*}{$\mathcal{H}_{1,2}$}&\multirow{3}{*}{$100$}&$100$&94.3&95.2&94.8&94.7&97.5&39.0&37.5&40.0&40.5&95.9\\
			&&$200$&99.6&99.9&99.6&99.4&100.0&36.7&36.1&37.4&37.2&98.5\\
			&&$500$&100.0&100.0&100.0&100.0&100.0&41.1&38.5&39.1&38.9&99.9\\\cline{2-13}
			&\multirow{3}{*}{$200$}&$100$&99.9&100.0&99.9&99.9&100.0&58.3&60.2&61.8&61.5&98.5\\
			&&$200$&100.0&100.0&100.0&100.0&100.0&62.4&63.4&64.4&63.7&99.7\\
			&&$500$&100.0&100.0&100.0&100.0&100.0&60.5&61.4&62.3&62.1&100.0\\\cline{2-13}
			&\multirow{3}{*}{$500$}&$100$&100.0&100.0&100.0&100.0&100.0&91.2&92.3&92.2&92.4&100.0\\
			&&$200$&100.0&100.0&100.0&100.0&100.0&90.5&92.3&92.8&92.9&99.9\\
			&&$500$&100.0&100.0&100.0&100.0&100.0&88.8&91.4&91.5&91.6&100.0\\\hline
	\end{tabular}}
	\caption{Empirical Rejection Rates (in percentage) for One Change Point in Mean (Non-Gaussian Error)}\label{tab:singleNONGaussian}
\end{table}
\clearpage
\newpage

\begin{table}[h!]
	\centering
	\resizebox{\textwidth}{!}{
		\begin{tabular}{ccc|ccc|ccc|ccc}\hline
			&&&\multicolumn{3}{c|}{$n = 100$}&\multicolumn{3}{c|}{$n = 200$}&\multicolumn{3}{c}{$n = 500$}\\
			&&&$p = 100$&$p = 200$&$p = 500$&$p = 100$&$p = 200$&$p = 500$&$p = 100$&$p = 200$&$p = 500$\\\hline
			\multirow{8}{*}{ID}&\multirow{2}{*}{$\mathcal{H}_0$}&$T_n$&4.1&4.1&4.3&6.1&5.2&4.7&5.9&5.7&6.3\\
			&&$T_n^{\diamond}$&14.1&12.5&13.6&7.6&7.6&5.9&6.4&7.0&4.8\\\cline{3-12}
			&\multirow{2}{*}{$\mathcal{H}_{1,1}$}&$T_n$&99.6&100.0&100.0&100.0&100.0&100.0&100.0&100.0&100.0\\
			&&$T_n^{\diamond}$&51.6&82.4&99.6&97.2&99.9&100.0&100.0&100.0&100.0\\\cline{3-12}
			&\multirow{2}{*}{$\mathcal{H}_{1,2}$}&$T_n$&0.3&0.0&0.0&0.0&0.0&0.0&0.0&0.0&0.0\\
			&&$T_n^{\diamond}$&83.0&97.0&100.0&100.0&100.0&100.0&100.0&100.0&100.0\\\cline{3-12}
			&\multirow{2}{*}{$\mathcal{H}_{1,3}$}&$T_n$&0.3&0.3&0.1&0.2&0.1&0.0&0.0&0.0&0.0\\
			&&$T_n^{\diamond}$&72.3&94.6&100.0&99.8&100.0&100.0&100.0&100.0&100.0\\\hline
			\multirow{8}{*}{AR}&\multirow{2}{*}{$\mathcal{H}_0$}&$T_n$&4.7&4.7&4.5&5.7&5.6&5.0&6.0&5.6&5.9\\
			&&$T_n^{\diamond}$&17.3&15.6&15.0&7.9&7.9&6.8&7.5&7.9&6.3\\\cline{3-12}
			&\multirow{2}{*}{$\mathcal{H}_{1,1}$}&$T_n$&92.8&99.4&100.0&99.9&100.0&100.0&100.0&100.0&100.0\\
			&&$T_n^{\diamond}$&38.7&62.5&94.2&84.0&98.3&100.0&100.0&100.0&100.0\\\cline{3-12}
			&\multirow{2}{*}{$\mathcal{H}_{1,2}$}&$T_n$&1.5&0.5&0.0&0.2&0.0&0.0&0.0&0.0&0.0\\
			&&$T_n^{\diamond}$&65.6&86.2&99.8&97.1&100.0&100.0&100.0&100.0&100.0\\\cline{3-12}
			&\multirow{2}{*}{$\mathcal{H}_{1,3}$}&$T_n$&2.3&0.6&0.6&0.8&0.5&0.0&0.1&0.0&0.0\\
			&&$T_n^{\diamond}$&58.4&81.1&99.5&96.1&99.8&100.0&100.0&100.0&100.0\\\hline
			\multirow{8}{*}{BD}&\multirow{2}{*}{$\mathcal{H}_0$}&$T_n$&4.8&4.9&4.4&5.8&5.6&5.5&6.1&6.0&5.4\\
			&&$T_n^{\diamond}$&16.2&14.3&13.0&8.3&7.6&7.2&7.1&6.5&6.3\\\cline{3-12}
			&\multirow{2}{*}{$\mathcal{H}_{1,1}$}&$T_n$&93.7&99.7&100.0&100.0&100.0&100.0&100.0&100.0&100.0\\
			&&$T_n^{\diamond}$&39.4&62.6&94.9&86.7&99.0&100.0&100.0&100.0&100.0\\\cline{3-12}
			&\multirow{2}{*}{$\mathcal{H}_{1,2}$}&$T_n$&0.9&0.1&0.0&0.0&0.0&0.0&0.0&0.0&0.0\\
			&&$T_n^{\diamond}$&65.2&84.0&99.7&97.1&99.9&100.0&100.0&100.0&100.0\\\cline{3-12}
			&\multirow{2}{*}{$\mathcal{H}_{1,3}$}&$T_n$&1.1&2.0&0.4&0.8&0.2&0.0&0.0&0.0&0.0\\
			&&$T_n^{\diamond}$&60.3&81.3&98.7&96.4&99.9&100.0&100.0&100.0&100.0\\\hline
			\multirow{8}{*}{CS}&\multirow{2}{*}{$\mathcal{H}_0$}&$T_n$&11.6&11.2&11.4&10.6&10.8&11.0&11.2&11.4&11.3\\
			&&$T_n^{\diamond}$&44.7&44.5&44.7&39.3&38.7&39.4&32.6&33.5&33.6\\\cline{3-12}
			&\multirow{2}{*}{$\mathcal{H}_{1,1}$}&$T_n$&41.0&39.4&37.3&57.8&60.2&60.4&89.2&90.0&91.3\\
			&&$T_n^{\diamond}$&52.5&51.4&54.4&59.3&57.2&57.5&79.9&81.9&81.9\\\cline{3-12}
			&\multirow{2}{*}{$\mathcal{H}_{1,2}$}&$T_n$&11.6&10.1&9.5&9.0&11.2&9.3&4.8&7.6&7.3\\
			&&$T_n^{\diamond}$&57.9&58.5&58.9&64.6&66.6&65.6&88.2&88.3&88.2\\\cline{3-12}
			&\multirow{2}{*}{$\mathcal{H}_{1,3}$}&$T_n$&10.9&11.6&12.4&13.1&12.1&15.0&11.3&14.5&12.8\\
			&&$T_n^{\diamond}$&53.9&58.3&55.9&65.1&63.9&64.2&87.7&87.7&88.5\\\hline
		\end{tabular}
	}
	\caption{Empirical Rejection Rates  (in percentage) for Multiple Change Points (Gaussian Error)}\label{tab:unsupervised}
\end{table}

\clearpage
\newpage

\begin{table}[t] \centering 
	\begin{scriptsize}
	\caption{Empirical Rejection Rates (in percentage) for Single Change Point Testing: Sizes  for Example \ref{exp:1} with $\mu = 0$ and $p=2n$.} \label{tab:sizehdts}
	\label{} 
	\begin{tabular}{@{\extracolsep{5pt}} cccccccc} 
		\\[-1.8ex]\hline 
		\hline \\[-1.8ex] 
		& \multirow{2}{*}{$\rho$} & \multirow{2}{*}{$n$} &  \multicolumn{2}{c}{$T_n$} & \multirow{2}{*}{DCBS} & \multicolumn{2}{c}{HH} \\  \cline{4-5}  \cline{7-8}
		&  &  & $ \eta = 0.02 $ & $\eta = 0.05$ &    & $h=3$ & $h=6$ \\ 
		\hline \\[-1.8ex] 
		\multirow{12}{*}{(a)} & $0.2$ & $200$ & $0.042$ & $0.051$ & $0$ & $0.326$ & $0.264$ \\ 
		& $0.2$ & $400$ & $0.048$ & $0.047$ & $0$ & $0.377$ & $0.219$ \\ 
		& $0.2$ & $800$ & $0.062$ & $0.062$ & $0$ & $0.488$ & $0.236$ \\ \cline{2-8}
		& $0.5$ & $200$ & $0.067$ & $0.063$ & $0$ & $0.994$ & $0.363$ \\ 
		& $0.5$ & $400$ & $0.051$ & $0.052$ & $0$ & $1$ & $0.482$ \\ 
		& $0.5$ & $800$ & $0.062$ & $0.066$ & $0$ & $1$ & $0.643$ \\ \cline{2-8}
		& $0.7$ & $200$ & $0.011$ & $0.094$ & $0$ & $1$ & $0.980$ \\ 
		& $0.7$ & $400$ & $0.066$ & $0.058$ & $0.0005$ & $1$ & $1.000$ \\ 
		& $0.7$ & $800$ & $0.061$ & $0.069$ & $0$ & $1$ & $1$ \\  \cline{2-8}
		& $$-$0.5$ & $200$ & $0.020$ & $0.018$ & $0$ & $0.562$ & $0.964$ \\ 
		& $$-$0.5$ & $400$ & $0.034$ & $0.030$ & $0$ & $0.382$ & $0.906$ \\ 
		& $$-$0.5$ & $800$ & $0.057$ & $0.053$ & $0$ & $0.336$ & $0.947$ \\ \hline
		\multirow{12}{*}{(b)} & $0.2$ & $200$ & $0.052$ & $0.058$ & $0$ & $0.204$ & $0.084$ \\ 
		& $0.2$ & $400$ & $0.053$ & $0.049$ & $0$ & $0.204$ & $0.084$ \\ 
		& $0.2$ & $800$ & $0.044$ & $0.050$ & $0$ & $0.312$ & $0.083$ \\  \cline{2-8}
		& $0.5$ & $200$ & $0.076$ & $0.074$ & $0$ & $0.998$ & $0.210$ \\ 
		& $0.5$ & $400$ & $0.052$ & $0.057$ & $0$ & $1$ & $0.311$ \\ 
		& $0.5$ & $800$ & $0.044$ & $0.051$ & $0$ & $1$ & $0.538$ \\  \cline{2-8}
		& $0.7$ & $200$ & $0.002$ & $0.103$ & $0$ & $1$ & $0.978$ \\ 
		& $0.7$ & $400$ & $0.060$ & $0.064$ & $0.0005$ & $1$ & $1.000$ \\ 
		& $0.7$ & $800$ & $0.046$ & $0.054$ & $0$ & $1$ & $1$ \\  \cline{2-8}
		& $$-$0.5$ & $200$ & $0.012$ & $0.022$ & $0$ & $0.370$ & $0.944$ \\ 
		& $$-$0.5$ & $400$ & $0.039$ & $0.032$ & $0$ & $0.210$ & $0.878$ \\ 
		& $$-$0.5$ & $800$ & $0.040$ & $0.040$ & $0$ & $0.150$ & $0.932$ \\ 
		\hline \\[-1.8ex] 
	\end{tabular} 
\end{scriptsize}
\end{table}

\begin{table}[t] \centering 
	\begin{scriptsize}
	\caption{Empirical Rejection Rates (in percentage) for Single Change Point Testing: Powers for Example \ref{exp:1} (i) with $p=2n$.} \label{tab:powerhdts}
	\label{} 
	\begin{tabular}{@{\extracolsep{0pt}} ccc|ccccc|ccccc} 
		\\[-1.8ex]\hline 
		\hline \\[-1.8ex] 
		&  &  &  \multicolumn{5}{c}{Case (i)}  &  \multicolumn{5}{c}{Case (ii)}   \\ 
		& \multirow{2}{*}{$\rho$} & \multirow{2}{*}{$n$} &  \multicolumn{2}{c}{$T_n$} & \multirow{2}{*}{DCBS} & \multicolumn{2}{c}{HH} &  \multicolumn{2}{c}{$T_n$} & \multirow{2}{*}{DCBS} & \multicolumn{2}{c}{HH} \\ \cline{4-5}  \cline{7-8}  \cline{9-10}  \cline{12-13}
		&  &  & $ \eta = 0.02 $ & $\eta = 0.05$ &    & $h=3$ & $h=6$ & $ \eta = 0.02 $ & $\eta = 0.05$ &    & $h=3$ & $h=6$ \\ 
		\hline \\[-1.8ex] 
		\multirow{12}{*}{(a)} & $0.2$ & $200$ & $0.728$ & $0.700$ & $0.058$ & $0.972$ & $0.956$ & $0.908$ & $0.878$ & $0.102$ & $0.998$ & $0.996$ \\ 
		& $0.2$ & $400$ & $1$ & $1$ & $0.906$ & $1$ & $1$ & $1$ & $1$ & $0.999$ & $1$ & $1$ \\ 
		& $0.2$ & $800$ & $1$ & $1$ & $1$ & $1$ & $1$ & $1$ & $1$ & $1$ & $1$ & $1$ \\   \cline{2-13}
		& $0.5$ & $200$ & $0.317$ & $0.252$ & $0$ & $1$ & $0.742$ & $0.392$ & $0.316$ & $0$ & $1$ & $0.835$ \\ 
		& $0.5$ & $400$ & $0.799$ & $0.774$ & $0$ & $1$ & $0.996$ & $0.967$ & $0.948$ & $0$ & $1$ & $1$ \\ 
		& $0.5$ & $800$ & $1$ & $1$ & $0.046$ & $1$ & $1$ & $1$ & $1$ & $0.337$ & $1$ & $1$ \\   \cline{2-13}
		& $0.7$ & $200$ & $0.020$ & $0.160$ & $0.0005$ & $1$ & $0.989$ & $0.023$ & $0.176$ & $0$ & $1$ & $0.994$ \\ 
		& $0.7$ & $400$ & $0.302$ & $0.237$ & $0.004$ & $1$ & $1$ & $0.390$ & $0.337$ & $0.005$ & $1$ & $1$ \\ 
		& $0.7$ & $800$ & $0.840$ & $0.813$ & $0$ & $1$ & $1$ & $0.970$ & $0.958$ & $0$ & $1$ & $1$ \\  \cline{2-13}
		& $$-$0.5$ & $200$ & $0.990$ & $1.000$ & $0$ & $1$ & $1$ & $1$ & $1$ & $0$ & $1$ & $1$ \\ 
		& $$-$0.5$ & $400$ & $1$ & $1$ & $0.188$ & $1$ & $1$  & $1$ & $1$ & $0.9995$ & $1$ & $1$ \\ 
		& $$-$0.5$ & $800$ & $1$ & $1$ & $1$ & $1$ & $1$ & $1$ & $1$ & $1$ & $1$ & $1$ \\ \hline
		\multirow{12}{*}{(b)} & $0.2$ & $200$ & $0.906$ & $0.888$ & $0.010$ & $0.991$ & $0.983$ & $0.983$ & $0.976$ & $0.064$ & $1$ & $0.999$ \\ 
		& $0.2$ & $400$ & $1$ & $1$ & $0.964$ & $1$ & $1$ & $1$ & $1$ & $1$ & $1$ & $1$ \\ 
		& $0.2$ & $800$ & $1$ & $1$ & $1$ & $1$ & $1$ & $1$ & $1$ & $1$ & $1$ & $1$  \\   \cline{2-13}
		& $0.5$ & $200$ & $0.409$ & $0.314$ & $0$ & $1$ & $0.687$ & $0.532$ & $0.456$ & $0$ & $1$ & $0.808$ \\ 
		& $0.5$ & $400$ & $0.954$ & $0.927$ & $0$ & $1$ & $1$ & $0.996$ & $0.994$ & $0.0065$ & $1$ & $1$ \\ 
		& $0.5$ & $800$ & $1$ & $1$ & $0.041$ & $1$ & $1$  & $1$ & $1$ & $0.5675$ & $1$ & $1$  \\   \cline{2-13}
		& $0.7$ & $200$ & $0.006$ & $0.178$ & $0$ & $1$ & $0.992$ & $0.006$ & $0.220$ & $0$ & $1$ & $0.996$ \\ 
		& $0.7$ & $400$ & $0.349$ & $0.328$ & $0.002$ & $1$ & $1$ & $0.489$ & $0.484$ & $0.005$ & $1$ & $1$ \\ 
		& $0.7$ & $800$ & $0.960$ & $0.946$ & $0$ & $1$ & $1$ & $0.995$ & $0.994$ & $0$ & $1$ & $1$ \\   \cline{2-13}
		& $$-$0.5$ & $200$ & $0.999$ & $1$ & $0$ & $1$ & $1.000$ & $1$ & $1$ & $0$ & $1$ & $1$ \\ 
		& $$-$0.5$ & $400$ & $1$ & $1$ & $0.524$ & $1$ & $1$ & $1$ & $1$ & $0.9955$ & $1$ & $1$ \\ 
		& $$-$0.5$ & $800$ & $1$ & $1$ & $1$ & $1$ & $1$ & $1$ & $1$ & $1$ & $1$ & $1$ \\ 
		\hline \\[-1.8ex] 
	\end{tabular} 
\end{scriptsize}
\end{table}

\clearpage
\newpage

\def\thesection{S\arabic{section}}


The supplementary material contains all the proofs for theoretical results  stated in the paper.
In particular, Section~\ref{sec:appendixHDID} contains all proofs for the results stated for high-dimensional independent data in Section~\ref{sec:theory}. Section~\ref{sec:appendixHDTS} collects the proofs for theory stated for high-dimensional time series in Section~\ref{sec:HDLP}. Section~\ref{sec:covsim} presents some simulation results for the covariance matrix change-point testing. Section~\ref{sec:WBSsim} presents simulation results for change-point estimation for both independent and dependent data.

\section{Proofs  for high-dimensional independent data}
\label{sec:appendixHDID}

Throughout this section, $X_t$, $t=1,\change{...},n$ are i.i.d. We begin by proving an intermediate technical result which provides the crucial ingredient for all subsequent developments. Define
\[
\widetilde S_n(k,m) = \sum_{i = k}^{m}\sum_{j = k}^{i} X_{i+1}^TX_j
\]
for any $1 \leq k < m \leq n$ and let $\widetilde S_n(k,m) = 0$ for $k\geq m$ or $k<1$ or $m>n$.

\begin{theorem}\label{weakS}
	Under Assumption \ref{ass}  we have as $n \to \infty$
	\[
	\Big\{ \frac{\sqrt{2}}{n\|\Sigma\|_F}\widetilde S_n(\floor{an}+1,\floor{bn}-1) \Big\}_{(a,b)\in [0,1]^2} \rightsquigarrow Q \quad in \quad \ell^\infty([0,1]^2),
	\]
	where $Q$ is a centered Gaussian process with covariance structure given by~\eqref{eq:covQ}. Moreover, the sample paths of $\frac{\sqrt{2}}{n\|\Sigma\|_F}\widetilde S_n(\floor{an}+1,\floor{bn}-1)$ are asymptotically uniformly equicontinuous in probability.
\end{theorem}

The proof of this Theorem is long and technical. We postpone it to Section~\ref{sec:proofSnwean}.

Next we present some basic results that will be used throughout the following proofs. Define
\begin{equation}\label{eq:DnX}
	D^X(k;\ell,m) := \sum_{\stackrel{\ell \leq j_1, j_3 \leq k}{j_3 \neq j_1}}\sum_{\stackrel{k+1 \leq j_2, j_4 \leq m}{j_2 \neq j_4}} (X_{j_1}-X_{j_2})^T(X_{j_3}-X_{j_4}),
\end{equation}
for $1 \leq \ell \leq k < m \leq n$ and $D^X(k;\ell,m) = 0$ otherwise. Observe that under the null of constant mean function we have $D^X \equiv D$ and we always have the representation
\begin{align}
	D^X(k;\ell,m) =~& 2(m-k)(m-k-1) \widetilde S_n(\ell,k) + 2 (k-\ell)(k-\ell+1)\widetilde S_n(k+1,m) \notag
	\\
	& - 2(k-\ell+1)(m-k) (\widetilde S_n(\ell,m) - \widetilde S_n(\ell,k) - \widetilde S_n(k+1,m)). \label{repr:DXtildeS}
\end{align}
Theorem~\ref{weakS} and uniform asymptotic equi-continuity of the sample paths of $S$ in probability together with some simple calculations yields
\begin{align} \label{eq:weakHn}
	\{ H_n^X(r,a,b)\}_{(a,b,r) \in [0,1]^3} := \Big\{ \frac{\sqrt{2}}{n^3\|\Sigma\|_F} D^X(\floor{rn} ;\floor{an},\floor{bn}) \Big\}_{(a,b,r) \in [0,1]^3} \weak 2 G
\end{align}
in $\ell^\infty([0,1]^3)$ where for $0 \leq a < b < r \leq 1$
\begin{align*}
	G(r;a,b) &= (b-r)^2Q(a,r) + (r-a)^2Q(r,b) - (r-a)(b-r)(Q(a,b) - Q(a,r) - Q(r,b))
	\\
	&= (b-r)(b-a)Q(a,r) + (r-a)(b-a)Q(r,b) - (r-a)(b-r)Q(a,b),
\end{align*}
and $G(r;a,b) = 0$ otherwise. Note that this is the process $G$ appearing in~\eqref{eq:defGrab}. Since the sample paths of $Q$ are uniformly continuous with respect to the Euclidean metric on $[0,1]^2$, a simple computation shows that the sample paths of $G$ are uniformly continuous with respect to the Euclidean metric on $[0,1]^3$.


\subsection{Proof of Theorem~\ref{cor:T0}}

For $n \geq 1$ consider the maps
\[
\Phi_n(f) := \sup_{k=2,...,n-3} \frac{f(k/n,1/n,1)^2}{\frac{1}{n}\Big(\sum_{t=\ell+1}^{k-2} f(t/n,\ell/n,k/n)^2 + \sum_{t=k+2}^{m-2} f(t/n,k/n,m/n)^2 \Big)}
\]
defined for functions $f: [0,1]^3 \to \R$ such that the denominator is non-zero. With this definition we have $T_n = \Phi_n(H_n^X)$ for the process $H_n^X$ defined in~\eqref{eq:weakHn}. Let $D_\Phi$ denote the set of all continuous functions $f$ in $\ell^\infty([0,1]^3)$ with the property $\inf_{r \in [0,1]} \int_0^r f(u,0,r)^2 du  + \int_r^1 f(u,r,1)^2 du  > 0$ and consider the map $\Phi: D_\Phi \to \R$ given by
\[
\Phi(f) = \sup_{r \in [0,1]} \frac{f(r,0,1)^2}{\int_0^r f(u,0,r)^2 du  + \int_r^1 f(u,r,1)^2 du}.
\]
Straightforward arguments show that for any sequence of functions $f_n$ with $\|f_n-f\|_\infty = o(1)$ for some function $f \in D_\Phi$ we have $\Phi_n(f_n) \to \Phi(f)$. Observe that
\[
P\Big(\inf_r \int_0^r G(u,0,r)^2 du  + \int_r^1 G(u,r,1)^2 du = 0 \Big) = 0.
\]
This follows from continuity of the sample paths of $G$, the fact that $P(G(u,r,1)^2>0)=1$ for any $0<r<u<1$ and $P(G(u,0,r)^2>0)=1$ for any $0<u<r<1$. Hence $2G \in D_\Phi$ with probability one. Combined with the fact that $H_n^X \weak 2G$ and the extended continuous mapping theorem (see Theorem~1.11.1 in~\cite{vdVW}) this implies $T_n = \Phi_n(H_n^X) \weak \Phi(G) = T$. This completes the proof.  \hfill $\Box$

\subsection{Proof of Theorem~\ref{cor:T_alt}}

A key step in the proof of this Theorem is an expansion for $D$ from~\eqref{eq:defDn} in terms of $D^X$ from~\eqref{eq:DnX} in the setting where $\E[Y_i] = \mu$ for $i=1,...,\floor{nb^*}$ and $\E[Y_i] = \mu + \delta_n$ for $i=\floor{nb^*}+1,...,n$. To shorten notation let $k^* := \floor{nb^*}$. We will only provide a detailed derivation in the case $\ell < k < k^* < m$, all other cases can be handled similarly. Observe that
\begin{align*}
	D(k;\ell,m) = \sum_{j_1=\ell}^{k}\sum_{j_3  = \ell, j_3 \neq j_1}^{k} \Big\{ &\sum_{j_2=k+1}^{k^*}\sum_{j_4= k+1,j_4 \neq j_2}^{k^*}(X_{j_1}-X_{j_2})^T(X_{j_3}-X_{j_4})
	\\
	&+ \sum_{j_2=k+1}^{k^*}\sum_{j_4= k^*+1}^{m}(X_{j_1}-X_{j_2})^T(X_{j_3}-X_{j_4} - \delta_n)
	\\
	&+ \sum_{j_2=k^*+1}^{m}\sum_{j_4= k+1}^{k^*}(X_{j_1}-X_{j_2}- \delta_n)^T(X_{j_3}-X_{j_4})
	\\
	&+ \sum_{j_2=k^*+1}^{m}\sum_{j_4= k^*+1,j_4 \neq j_2}^{m}(X_{j_1}-X_{j_2} - \delta_n)^T(X_{j_3}-X_{j_4} - \delta_n) \Big\}.
\end{align*}
Now some straightforward algebraic manipulations show that
\begin{align*}
	D(k;\ell,m) =~& D^X(k;\ell,m) + (k-\ell+1)(k-\ell)(m-k^*)(m-k^*-1)\|\delta_n\|_2^2
	\\
	& - 2(k-\ell)(m-k^*)(m-k-2) \sum_{j=\ell}^{k} X_j^T\delta_n
	\\
	& - 4 (k-\ell)(k-\ell-1)(m-k^*) \sum_{j=k+1}^{k^*} X_j^T\delta_n.
\end{align*}
Let $s_n(k) := \sum_{j=1}^k X_j^T\delta_n$. Then
\[
\sup_{1\leq\ell\leq k\leq n} \Big| \sum_{j=\ell}^{k} X_j^T\delta_n \Big| \leq 2 \sup_{1 \leq k \leq n} |s_n(k)|.
\]
Observing that $s_n$ is a sum of centered i.i.d. random variables, Kolmogorov's inequality implies
{\[
	\sup_{1 \leq k \leq n} |s_n(k)| = O_P\Big( (n Var(X_1^T\delta_n ))^{1/2} \Big) = O_P\Big( \Big(n \delta_n^T\Sigma\delta_n\Big)^{1/2} \Big) 
	= o_P(n^{1/2}\|\delta_n\|\|\Sigma\|_F^{1/2})
	\]
	where we used the fact $\|\Sigma_n\|_2 = o(\|\Sigma_n\|_F)$, see Remark~\ref{rem:condA}.
	This implies that, uniformly in $k,\ell,m$ we have for $\ell<k<k^*<m$
	\[
	D(k;\ell,m) = D^X(k;\ell,m) + (k-\ell+1)(k-\ell)(m-k^*)(m-k^*-1)\|\delta_n\|_2^2 +  o_P( n^{7/2} \|\delta_n\|\|\Sigma\|_F^{1/2})
	\]
	Similar arguments show that for $\ell<k^*<k<m$
	\[
	D(k;\ell,m) = D^X(k;\ell,m) + (k^*-\ell+1)(k^*-\ell)(m-k)(m-k-1)\|\delta_n\|_2^2 + o_P( n^{7/2} \|\delta_n\|\|\Sigma\|_F^{1/2})
	\]
	while for $k^* \leq \ell$ or $k^* \geq m$ we have
	\[
	D(k;\ell,m) = D^X(k;\ell,m).
	\]
	Now assuming that $n\|\delta_n\|_2^2/\|\Sigma\|_F \to c^2 \in [0,\infty)$, and hence $n^{7/2} \|\delta_n\|\|\Sigma\|_F^{1/2} = O(n^3 \|\Sigma\|_F)$,  
	it follows that
	\begin{align} \label{eq:weakHn_alt}
		\Big\{ \frac{\sqrt{2}}{n^3\|\Sigma\|_F} D(\floor{rn} ;\floor{an},\floor{bn}) \Big\}_{(a,b,r) \in [0,1]^3} \weak \Big\{2 G(r,a,b) + \sqrt{2} c\Delta(r,a,b)\Big\}_{(a,b,r) \in [0,1]^3}
	\end{align}
	where
	\[
	\Delta(r,a,b) := \begin{cases}
		(b^*-a)^2(b-r)^2 & a < b^* \leq r < b,
		\\
		(r-a)^2(b-b^*)^2 & a < r < b^* < b,
		\\
		0 & b^* < a \mbox{ or } b^* > b.
	\end{cases}
	\]
	The remaining proof in the case $n\|\delta_n\|_2^2/\|\Sigma\|_F \to c^2 \in [0,\infty)$ follows by exactly the same arguments as given in the proof of Theorem~\ref{cor:T0} after replacing~\eqref{eq:weakHn} by~\eqref{eq:weakHn_alt} and the limit $2G$ by $2G + \sqrt{2}c\Delta$.

	Next consider the case $n\|\delta_n\|_2^2/\|\Sigma\|_F \to \infty$. Observe that
	\[
	T_n = \sup_{k=2,...,n-3} \frac{(D_n(k;1,n))^2}{W_n(k;1,n)} \geq \frac{(D_n(\floor{b^*n};1,n))^2}{W_n(\floor{b^*n};1,n)}.
	\]
	Since by assumption $\eta_i$ are constant for $i=1,...,\floor{b^*n}$ and $i=\floor{b^*n}+1,...,n$, respectively, we have
	\begin{equation}\label{eq:WnHn}
		\frac{1}{n^6\|\Sigma\|_F^2}W_n(\floor{b^*n};1,n) = \frac{1}{n}\Big[\sum_{t=1}^{\floor{b^*n}-2} H_n^X\Big(\frac{t}{n},\frac{1}{n},\frac{\floor{b^*n}}{n}\Big)^2 + \sum_{t=\floor{b^*n}+2}^{n-2} H_n^X\Big(\frac{t}{n},\frac{\floor{b^*n}}{n},1\Big)^2 \Big]
	\end{equation}
	for $H_n^X$ defined in~\eqref{eq:weakHn}. Uniform asymptotic equicontinuity of the sample paths of $H_n^X$ together with similar arguments as given in the proof of Theorem~\ref{cor:T0} implies that
	\begin{equation}\label{eq:Wnweak}
		\frac{1}{n^6\|\Sigma\|_F^2}W_n(\floor{b^*n};1,n) \weak \int_0^{b^*} G(u,0,b^*)^2 du  + \int_{b^*}^1 G(u,b^*,1)^2 du,
	\end{equation}
	where the limit is non-zero and finite almost surely.
	Next we will analyze the numerator. From the expansions given above we obtain
	\begin{align*}
		\frac{D_n(k^*;1,n)}{n^3 \|\Sigma\|_F} &= H_n^X(k^*/n;1/n,1) + (b^*)^2(1-b^*)^2\frac{n\|\delta_n\|_2^2}{\|\Sigma\|_F}(1+o(1)) +
		o_P\Big(\frac{n^{1/2}\|\delta_n\|}{\|\Sigma\|_F^{1/2}}\Big)
		\\
		& = O_P(1) + (b^*)^2(1-b^*)^2\frac{n\|\delta_n\|_2^2}{\|\Sigma\|_F}(1+o(1)) +
		o_P\Big(\frac{n\|\delta_n\|^2}{\|\Sigma\|_F}\Big).
	\end{align*}
	This implies that $\frac{D_n(k^*;1,n)}{n^3 \|\Sigma\|_F} \to \infty$ in probability. Combined with~\eqref{eq:Wnweak} and the fact that the limit in~\eqref{eq:Wnweak} is finite almost surely, the convergence $T_n \to \infty$ in probability follows. This completes the proof of Theorem~\ref{cor:T_alt}. \hfill $\Box$

	\subsection{Proof of Theorem~\ref{th:multcp0}}
	The proof is similar to the proof of Theorem~\ref{cor:T0}, and the proofs of the weak convergence of $T_n^*, T_n^\diamond$ are also similar to each other. For the sake of brevity we provide a brief outline for $T_n^*$ and omit all other details. Define the maps
	\begin{multline*}
		\Phi_n^*(f) := \max_{(r_1,r_2)\in \Omega_n(\epsilon)} \frac{f(r_1/n,1/n,r_2/n)^2}{\frac{1}{n}\Big(\sum_{t=1}^{r_1-2} f(t/n,1/n,r_1/n)^2 + \sum_{t=r_1+2}^{r_2-2} f(t/n,r_1/n,r_2/n)^2 \Big)}
		\\
		+ \max_{(s_1,s_2)\in \Omega_n(\epsilon)} \frac{f(s_2/n,s_1/n,1)^2}{\frac{1}{n}\Big(\sum_{t=s_1+1}^{s_2-2} f(t/n,s_1/n,s_2/n)^2 + \sum_{t=s_2+2}^{n-2} f(t/n,s_2/n,1)^2 \Big)}
	\end{multline*}
	for all $f$ for which the expression is well-defined and $\Phi^*: D_{\Phi^*}\to\R$
	\begin{multline*}
		\Phi^*(f) := \sup_{(r_1,r_2)\in \Omega(\epsilon)} \frac{f(r_1;0,r_2)^2}{\int_{0}^{r_1}f(u;0,r_1)^2du + \int_{r_1}^{r_2}f(u;r_1,r_2)^2du}
		\\
		+ \sup_{(s_1,s_2)\in \Omega(\epsilon)} \frac{f(s_2;s_1,1)^2}{\int_{s_1}^{s_2}f(u;s_1,s_2)^2du + \int_{s_2}^{1}f(u;s_2,1)^2du}.
	\end{multline*}
	where $D_{\Phi^*}$ denotes the set of all continuous functions such that all denominators in the fraction above are non-zero.  Similarly to the proof of Theorem~\ref{cor:T0} we have $P(2G \in D_{\Phi^*}) = 1, P(2G \in D_{\Phi^\diamond}) = 1$, and straightforward calculations show that all other conditions of the extended continuous mapping theorem are also satisfied. \hfill $\Box$
	
	\subsection{Proof of Theorem~\ref{th:multcp1}}
	
	We begin by proving the statement about $T^*_n$. Define $\delta_k := \mu_{k+1}^* - \mu_{k}^*$. Let $k_0 = k_{0,n}$ be a sequence such that 
	\[
	n \|\delta_{k_0} \|^2/\|\Sigma\|_F \to \infty
	\] 
	and 
	\[
	\limsup_{n \to \infty} \max_{k < k_0} n \| \delta_{k} \|^2/\|\Sigma\|_F < \infty.
	\]
	Further let $r_k^* := \lfloor n b_k^* \rfloor$, $k=0,...,M$. By assumption $(r_{k_0}^*,r_{k_0+1}^*) \in \Omega_n(\epsilon)$ (for $n$ sufficiently large, where 'sufficiently large' depends on $\epsilon,b_{1}^*,...,b_M^*$ only). Thus for sufficiently large $n$
	\[
	T^*_n \geq \frac{D_n(r_{k_0}^*;1,r_{k_0+1}^*)^2}{W_n(r_{k_0}^*;1,r_{k_0+1}^*)} \quad a.s.
	\]
	Recall the definition of $D^X$ in~\eqref{eq:DnX} and observe that for all $k ,\ell, m$
	\begin{align*}
		D_n(k;\ell,m) = & D^X(k;\ell,m) + \sum_{\stackrel{\ell \leq j_1, j_3 \leq k}{j_3 \neq j_1}}\sum_{\stackrel{k+1 \leq j_2, j_4 \leq m}{j_2 \neq j_4}}
		(X_{j_1}-X_{j_2})^T(\mu_{j_3}-\mu_{j_4})
		\\
		& + \sum_{\stackrel{\ell \leq j_1, j_3 \leq k}{j_3 \neq j_1}}\sum_{\stackrel{k+1 \leq j_2, j_4 \leq m}{j_2 \neq j_4}}
		(\mu_{j_1}-\mu_{j_2})^T(X_{j_3}-X_{j_4})
		\\
		&+ \sum_{\stackrel{\ell \leq j_1, j_3 \leq k}{j_3 \neq j_1}}\sum_{\stackrel{k+1 \leq j_2, j_4 \leq m}{j_2 \neq j_4}}
		(\mu_{j_1}-\mu_{j_2})^T(\mu_{j_3}-\mu_{j_4})
		\\
		=: & D^X(k;\ell,m) + A_{1,n}(k;\ell,m) + A_{2,n}(k;\ell,m) + A_{3,n}(k;\ell,m).
	\end{align*}
	Similar arguments as in the proof of Theorem~\ref{cor:T_alt} show that 
	\begin{multline*}
		\sup_{1 \leq k,\ell, m \leq r_{j+1}^*} |A_{1,n}(k;\ell,m)| + |A_{2,n}(k;\ell,m)|
		\\
		= O_P(n^{7/2} \max_{k \leq j} (\delta_{k}^T \Sigma \delta_{k})^{1/2}) = o_P(n^{7/2} \max_{k \leq j} \|\delta_k\| \|\Sigma\|_F^{1/2} )
	\end{multline*}
	since $\delta_{k}^T \Sigma \delta_{k} \leq \|\Sigma\|_2 \|\delta_k\|^2 = o(\|\Sigma\|_F \|\delta_k\|^2)$. 
	Moreover, straightforward calculations show that
	\begin{align*}
		\sup_{r_{k_0}^*+1 \leq k, \ell, m \leq r_{k_0+1}^*} A_{3,n}(k;\ell,m) &= 0,
		\\
		\sup_{1 \leq k, \ell, m \leq r_{k_0}^*} A_{3,n}(k;\ell,m) & = O(n^4 \max_{k < k_0} \|\delta_k\|^2),
		\\
		A_{3,n}(r_{k_0}^*;1,r_{k_0+1}^*) &= n^4 (b_{k_0}^*)^2(b_{k_0+1}^* - b_{k_0}^*)^2\|\delta_{k_0}\|^2(1+o(1))
	\end{align*}
	where we used the fact that by definition of $k_0$ one has $\max_{k < k_0} \|\delta_{k}\| / \|\delta_{k_0}\| = o(1)$ for the last representation.
	Combining the findings above with process convergence of $D_n^X(r,a,b)/(n^3 \|\Sigma\|_F)$ indexed in $a,b,r \in [0,1]$ it follows that
	\begin{align*}
		\frac{D_n(r_{k_0}^*;1,r_{k_0+1}^*)}{n^3 \|\Sigma\|_F} =&  \frac{n (b_{k_0}^*)^2(b_{k_0+1}^* - b_{k_0}^*)^2\|\delta_{k_0}\|^2}{\|\Sigma\|_F}(1+o(1))
		\\
		& + o_P\Big( \frac{n^{1/2} \max_{k \leq k_0} \|\delta_k\|} {\|\Sigma\|_F^{1/2}} \Big) + O_P(1)
	\end{align*}
	which converges to $+\infty$ in probability under the assumptions made since by construction $\max_{k < k_0} \|\delta_k\| = o(\|\delta_{k_0}\|)$. Moreover, 
	\[
	\frac{W_n(r_{k_0}^*;1,r_{k_0+1}^*)}{n^6 \|\Sigma\|_F^2} =  o_P\Big(n \frac{\max_{k < k_0} \|\delta_k\|^2}{\|\Sigma\|_F} \Big) + O_P(1) + O_P\Big(n ^2 \frac{\max_{k < k_0} \|\delta_k\|^4}{\|\Sigma\|_F^2} \Big) = O_P(1)
	\]
	where we used that by construction $\max_{k<k_0} \|\delta_k\|^2 = O(n^{-1}\|\Sigma\|_F^{-1})$. Combining the above expansions for $D_n(r_{k_0}^*;1,r_{k_0+1}^*), W_n(r_{k_0}^*;1,r_{k_0+1}^*)$ it follows that
	\[ \frac{D_n(r_{k_0}^*;1,r_{k_0+1}^*)^2}{W_n(r_{k_0}^*;1,r_{k_0+1}^*)} \stackrel{P}{\to} \infty.
	\]
	This proves the claim for $T^*_n$. To prove the claim for $T^\diamond_n$, consider $k_0$ as above and define $r^* := (\ceil{2b_{k_0}^*/\epsilon} + 1)\epsilon/2$. Note that by construction $r^*\epsilon/2 \in \mathcal{G}_\epsilon$ and
	\[
	b_{k_0}^* + \epsilon/2 \leq r^*\epsilon/2 \leq b_{k_0}^* + \epsilon < b_{k_0+1}^*.
	\]
	Hence for $n$ sufficiently large $(\floor{r^*n},\floor{nk^*\epsilon/2}) \in \mathcal{G}_{\epsilon,n,f}$ and thus (for sufficiently large $n$)
	\[
	T^\diamond_n \geq \frac{D_n(\floor{nb_{k_0}^*};1,\floor{nr^*\epsilon/2})^2}{W_n(\floor{nb_{k_0}^*};1,\floor{nr^*\epsilon/2})} \quad a.s.
	\]
	From here on the arguments are very similar to the ones for $T^*_n$ and details are omitted for the sake of brevity.
	\hfill $\Box$

	\subsection{Proofs for Remark~\ref{rem:condA} and Remark~\ref{rem:condB}} \label{sec:proofA1A3}
	
	For the equivalence between \ref{order} and $\|\Sigma_n\|_2 = o(\|\Sigma_n\|_F)$ denote by $\lambda_1 \geq \lambda_2 \geq \cdots \geq \lambda_p \geq 0$ the ordered eigenvalues of $\Sigma_n$. Then 
	\[
	tr(\Sigma_n^4) = \sum_i \lambda_i^4 \leq \lambda_1^2 \sum_i \lambda_i^2 = \|\Sigma_n\|_2^2\|\Sigma_n\|_F^2
	\] 
	so $\|\Sigma_n\|_2 = o(\|\Sigma_n\|_F)$ implies \ref{order}. We also have
	\[ 
	\|\Sigma_n\|_2^4 = \lambda_1^4 \leq \sum_i \lambda_i^4 =tr(\Sigma_n^4),
	\]
	so \ref{order} implies $\|\Sigma_n\|_2 = o(\|\Sigma_n\|_F)$.
	For the remaining part of Remark \ref{rem:condA}, observe that
	\begin{equation}\label{eq:Frobbound}
		\|\Sigma_n\|_F \geq (\sum_{j=1}^p \Sigma_n(j,j))^{1/2} \geq c_0^{1/2}p_n^{1/2}.
	\end{equation}
	by (i). We also have by symmetry of $\Sigma_n$ and by (ii) since $\Sigma_n(i,j) = cum(X_{0,i,n},X_{0,j,n})$
	\begin{align*}
		\|\Sigma_n\|_2^2 &\leq \|\Sigma_n\|_1\|\Sigma_n\|_\infty = \|\Sigma_n\|_1^2= \max_{1\leq j \leq p_n} \left\{\sum_{i=1}^{p_n} |\Sigma_n(i,j)|\right\}^2 \\
		& \leq \max_{1\leq j \leq p_n} \left\{\sum_{i=1}^{p_n} C_2(1\vee|i-j|^{-r})\right\}^2  = O(1)
	\end{align*}
	where the last bound follows since $r>1$. Since $p_n \to \infty$, this combined with~\eqref{eq:Frobbound} shows \ref{order} by using the inequality
	\[
	tr(\Sigma^4) = \|\Sigma^2\|_F^2 \leq \|\Sigma\|_2^2\|\Sigma\|_F^2.
	\]
	For~\ref{cumulant} note that for $2 \leq h \leq 6$ we have by (ii)
	\begin{align*}
		\sum_{l_1,\change{...},l_h = 1}^{p}cum^2(X_{0,l_1,n},\change{...},X_{0,l_h,n}) &= \sum_{l_1 = 1}^{p_n} \sum_{m = 0}^{p_n} \sum_{l_2,...,l_h \in S_{m,h}(l_1)} cum^2(X_{0,l_1,n},\change{...},X_{0,l_h,n})
		\\
		&\leq \sum_{l_1 = 1}^{p_n} \sum_{m = 0}^{p_n} |S_{m,h}(l_1)| C_h^2 (1 \vee m)^{-2r}
		\\
		&\lesssim p_n \sum_{m = 0}^{p_n} (1 \vee m)^{h-2-2r},
	\end{align*}
	where
	\[
	S_{m,h}(l_1) := \{1\leq l_2,...,l_h \leq p_n: \max_{1 \leq i,j \leq h} |l_i - l_j| = m\}.
	\]
	Now the sum is of order $O(p_n^{h-1-2r})$ if $h - 2 - 2r > -1$ and of order $O(1)$ if $h - 2 - 2r < -1$. Now a simple computation shows that~\eqref{cumulant} is satisfied if $h-2r < h/2$ for $h=2,...,6$, which is equivalent to $r > 6/4 =3/2$.
	
	For Remark \ref{rem:condB}, all arguments are similar to the proof of Remark \ref{rem:condA} but the verification for assumption \ref{cov:cum1}. Consider
	
	\begin{align*}
		&\max_{l_1,l_2 = 1,\change{...},p_n}\sum_{l_3,l_4 = 1}^{p_n}|cum(X_{0,l_1,n},X_{0,l_2,n},X_{0,l_3,n},X_{0,l_4,n})| \\
		&= \max_{l_1,l_2 = 1,\change{...},p_n}\sum_{m = |l_1-l_2|}^{p_n}\sum_{l_3,l_4 \in S_{m,4}(l_1,l_2)}|cum(X_{0,l_1,n},\change{...},X_{0,l_4,n})|\\
		&\leq \max_{l_1,l_2 = 1,\change{...},p_n}\sum_{m = |l_1-l_2|+1}^{p_n}|S_{m,4}(l_1,l_2)|C_4(1 \vee m)^{-r}) + (|l_2-l_1+1|)^2C_4(1\vee |l_1-l_2| )^{-r}\\
		&\lesssim \sum_{m = 0}^{p_n} C_4(1 \vee m)^{1-r} + O(1)< \infty
	\end{align*}
	where the last line uses the fact that $r>2$,
	$$S_{m,4}(l_1,l_2) := \{1 \leq l_3,l_4 \leq p_n: \max_{1\leq i,j\leq4}|l_i- l_j| = m  \}$$
	and
	$|S_{m,4}(l_1,l_2)| = O(m\vee 1)$ whenever $m > |l_1-l_2|$. This completes the proof.
	\hfill $\Box$
	
	\subsection{Proof of Remark \ref{rmk:cq}} \label{pr:rmk:cq}{ We begin by introducing the following proposition.
		\begin{proposition} \label{prop:cumcq}
			Assume the model $X_{t} = \Gamma Z_t$, where $\Gamma$ is $p$-by-$m$ real matrix such that $\Sigma = \Gamma\Gamma^T$, and $Z_t's$ are i.i.d random $m$-dimensional vectors with $\E[Z_t] = 0$ and $Var(Z_t) = I_m$  Furthermore for any $t > 0$,
			\begin{equation}\label{eq:momentsZ}
				\E[Z_{t,l_1}^{\alpha_1}\cdots Z_{t,l_q}^{\alpha_q}] = \E[Z_{t,l_1}^{\alpha_1}]\cdots\E[Z_{t,l_q}^{\alpha_q}]
			\end{equation}
			for any positive integer $q$ such that $\sum_{l = 1}^{q}\alpha_{q} \leq Q$, where $Q$ is a fixed positive constant, and $l_1 \neq \cdots \neq l_q$. Then for any $j_1,\change{...},j_k = 1,\change{...},p $
			\begin{equation}
				cum(X_{t,j_1},\change{...},X_{t,j_k}) = \sum_{l = 1}^{m}\Gamma_{j_1,l}\Gamma_{j_2,l}\cdots\Gamma_{j_k,l}cum_k(Z_{t,l})
			\end{equation}
			for any $1\leq k \leq Q$, where $cum_k(Z_{t,l})$ denotes the joint cumulants of $k$ identical random variables $Z_{t,l}$.
		\end{proposition}
		\begin{proof}
			By definition of joint cumulants we know
			\begin{align*}
				cum(X_{t,j_1},\change{...},X_{t,j_k}) &= cum(\sum_{l = 1}^{m}\Gamma_{j_1,l}Z_{t,l}, \change{...},\sum_{l = 1}^{m}\Gamma_{j_k,l}Z_{t,l})\\
				&=\sum_{l_1,\change{...},l_k = 1}^{m}\Gamma_{j_1,l_1}\cdots\Gamma_{j_k,l_k}cum(Z_{t,l_1},\change{...},Z_{t,l_k}).
			\end{align*}
			Hence it suffices to show that $cum(Z_{t,l_1},\change{...},Z_{t,l_k}) = 0$ if not all indices $l_1,\change{...},l_k$ are identical. By standard properties of cumulants this would be true if $Z_{t,l_1},\change{...},Z_{t,l_k}$ were independent; indeed, if there existed $l_i \neq l_j$ this would imply that $Z_{t,l_1},\change{...},Z_{t,l_k}$ would consist of at least two independent groups. Next, define $\tilde Z_{t,l}, l=1,...,m$ such that each $\tilde Z_{t,l}$ has the same distribution as $Z_{t,l}$ but $\tilde Z_{t,l_1},\change{...},\tilde Z_{t,l_k}$ are independent. By~\eqref{eq:momentsZ} we have
			\[
			\E[Z_{t,l_1}^{\alpha_1}\cdots Z_{t,l_q}^{\alpha_q}] = \E[Z_{t,l_1}^{\alpha_1}]\cdots\E[Z_{t,l_q}^{\alpha_q}] = \E[\tilde Z_{t,l_1}^{\alpha_1}\cdots \tilde Z_{t,l_q}^{\alpha_q}],
			\]
			and thus $cum(Z_{t,l_1},\change{...},Z_{t,l_k}) = cum(\tilde Z_{t,l_1},\change{...},\tilde Z_{t,l_k})$ by expressing cumulants through moments. Since $cum(\tilde Z_{t,l_1},\change{...},\tilde Z_{t,l_k}) = 0$ if $l_1,...,l_k$ are not identical, this completes the proof.
			
		\end{proof}

		Note that \cite{chen2010two} assume $Q = 8$. Assuming that $\sup_{j=1,...,m} \E[|Z_{1,j}|^6] = O(1)$, we have for any $2 \leq h \leq 6$,
		\begin{align*}
			&\sum_{j_1,\change{...},j_h}^{p}cum^2(X_{0,j_1},\change{...},X_{0,j_h})\\
			=&\sum_{j_1,\change{...},j_h}^{p}\sum_{l_1,l_2 = 1}^{m}(cum_h(Z_{0,l_1})cum_h(Z_{0,l_2}))\Gamma_{j_1,l_1}\cdots\Gamma_{j_h,l_1}\Gamma_{j_1,l_2}\cdots\Gamma_{j_h,l_2}\\
			=&\sum_{l_1,l_2 = 1}^{m}(cum_h(Z_{0,l_1})cum_h(Z_{0,l_2}))\left(\sum_{j = 1}^{p}\Gamma_{j,l_1}\Gamma_{j,l_2}\right)^h \leq C\sum_{l_1,l_2 = 1}^{m}|(\Gamma^T\Gamma)_{l_1,l_2}|^h,
		\end{align*}
		
		for some positive constant $C$. By simple manipulation we get for any $h \geq 2$
		\begin{align*}
			1&=\sum_{l_1,l_2 = 1}^{m}\left(|(\Gamma^T\Gamma)_{l_1,l_2}|^h/\sum_{j_1,j_2 = 1}^{m}|(\Gamma^T\Gamma)_{j_1,j_2}|^h\right)\\
			&\leq \sum_{l_1,l_2 = 1}^{m}\left(|(\Gamma^T\Gamma)_{l_1,l_2}|^h/\sum_{j_1,j_2 = 1}^{m}|(\Gamma^T\Gamma)_{j_1,j_2}|^h\right)^{2/h}\\
			&=\sum_{l_1,l_2 = 1}^{m}|(\Gamma^T\Gamma)_{l_1,l_2}|^2/\left(\sum_{j_1,j_2 = 1}^{m}|(\Gamma^T\Gamma)_{j_1,j_2}|^h\right)^{2/h}.
		\end{align*}
		This implies that
		\[
		\sum_{l_1,l_2 = 1}^{m}|(\Gamma^T\Gamma)_{l_1,l_2}|^h \leq \left(\sum_{l_1,l_2 = 1}^{m}|(\Gamma^T\Gamma)_{l_1,l_2}|^2\right)^{h/2} = \|\Gamma^T\Gamma\|_F^h = \|\Gamma\Gamma^T\|_F^h = \|\Sigma\|_F^h.
		\]
		Hence we have proved that~\eqref{eq:momentsZ} with $1 \leq \alpha_k \leq 6$ and $\sum_{k=1}^q \alpha_k \leq 6$ implies condition \ref{cumulant}.}\hfill\boxed{}
	
	\subsection{Proof of Proposition~\ref{prop:covm}}
	It suffices to verify that $Z$ satisfies~\ref{order} and~\ref{cumulant}. We begin by deriving a useful preliminary result which will be used in both proofs. Observe that
	\begin{align*}
		&\|\Sigma_Z\|_F^2 = \sum_{i,j = 1}^{p^2}\Sigma_Z(i,j)^2
		\\
		=& \sum_{l_1,l_2, l_3,l_4=1}^{p}\left(\Sigma_n{(l_1,l_3)}\Sigma_n{(l_2,l_4)}+\Sigma_n{(l_1,l_4)}\Sigma_n{(l_2,l_3)} + cum(X_{0,l_1},X_{0,l_2},X_{0,l_3},X_{0,l_4})\right)^2
		\\
		\geq& \sum_{l_1,l_2, l_3,l_4=1}^{p}\Sigma_n{(l_1,l_3)}^2\Sigma_n{(l_2,l_4)}^2 + \sum_{l_1,l_2, l_3,l_4=1}^{p}\Sigma_n{(l_1,l_4)}^2\Sigma_n{(l_2,l_3)}^2
		\\
		&+2\sum_{l_1,l_2, l_3,l_4=1}^{p}\Sigma_n{(l_1,l_3)}\Sigma_n{(l_2,l_4)}\Sigma_n{(l_1,l_4)}\Sigma_n{(l_2,l_3)}
		\\
		&+2\sum_{l_1,l_2, l_3,l_4=1}^{p}\Sigma_n{(l_1,l_3)}\Sigma_n{(l_2,l_4)}cum(X_{0,l_1},X_{0,l_2},X_{0,l_3},X_{0,l_4})
		\\
		&+2\sum_{l_1,l_2, l_3,l_4=1}^{p}\Sigma_n{(l_1,l_4)}\Sigma_n{(l_2,l_3)}cum(X_{0,l_1},X_{0,l_2},X_{0,l_3},X_{0,l_4})
		\\
		\geq& 2\|\Sigma_n\|_F^4 + 4\sum_{l_1,l_2, l_3,l_4=1}^{p}\Sigma_n{(l_1,l_3)}\Sigma_n{(l_2,l_4)}cum(X_{0,l_1},X_{0,l_2},X_{0,l_3},X_{0,l_4})
		\\
		\geq& 2\|\Sigma_n\|_F^4 - 4\sum_{l_1,l_2, l_3,l_4=1}^{p}|\Sigma_n{(l_1,l_3)}\Sigma_n{(l_2,l_4)}||cum(X_{0,l_1},X_{0,l_2},X_{0,l_3},X_{0,l_4})|
		\\
		\geq& 2\|\Sigma_n\|_F^4 - 4\sqrt{\sum_{l_1,l_2, l_3,l_4=1}^{p}\Sigma_n{(l_1,l_3)}^2\Sigma_n{(l_2,l_4)}^2}\sqrt{\sum_{l_1,l_2, l_3,l_4=1}^{p}cum(X_{0,l_1},X_{0,l_2},X_{0,l_3},X_{0,l_4})^2}
		\\
		=& 2\|\Sigma_n\|_F^4 - 4\|\Sigma_n\|_F^2o(\|\Sigma_n\|_F^2)
		\\
		\geq& \|\Sigma_n\|_F^4,
	\end{align*}
	for sufficiently large $n$ by condition \ref{cov:cum2}, where the second inequality follows since
	\[
	\sum_{l_1,l_2, l_3,l_4=1}^{p}\Sigma_n{(l_1,l_3)}\Sigma_n{(l_2,l_4)}\Sigma_n{(l_1,l_4)}\Sigma_n{(l_2,l_3)} = \sum_{l_1,l_2}\Big(\sum_{l=1}^p \Sigma_n{(l_1,l)}\Sigma_n{(l_2,l)}\Big)^2 \geq 0.
	\]
	Hence we have proved
	\begin{equation}\label{eq:sigmaZsigman}
		\|\Sigma_Z\|_F^2 \geq \|\Sigma_n\|_F^4.
	\end{equation}

	\textbf{Verification of~\ref{order}}	
	It is easy to see that $\Sigma_Z = \E[Z_0Z_0^T] - vech(\Sigma_n)vech(\Sigma_n)^T$. Specifically any element in $\Sigma_Z$ is in the form of $(\E[X_{1,l_1}X_{1,l_2}X_{1,l_3}X_{1,l_4}] - \E[X_{1,l_1}X_{1,l_2}]\E[X_{1,l_3}X_{1,l_4}])$ and the diagonal elements are of the form $Var(X_{1,l_1}X_{1,l_2})$. Recall that $\|\Sigma_Z\|_2 \leq \sqrt{\|\Sigma_Z\|_1\|\Sigma_Z\|_\infty}$, where
	\[
	\|\Sigma_Z\|_1 = \max_{j}\sum_{i = 1}^{p^2}|\Sigma_Z(i,j)|
	\]
	and
	\[
	\|\Sigma_Z\|_{\infty} = \max_{i}\sum_{j = 1}^{p^2}|\Sigma_Z(i,j)|.
	\]
	Since $\Sigma_Z$ is symmetric, we have $\|\Sigma_Z\|_1 = \|\Sigma_Z\|_\infty$ and thus $\|\Sigma_Z\|_2 \leq \|\Sigma_Z\|_1$. Observe that
	\begin{align*}
		&\|\Sigma_Z\|_1 = \max_j\sum_{i = 1}^{p^2}|\Sigma_Z(i,j)|
		\\
		&= \max_{l_1,l_2}\sum_{l_3,l_4 = 1}^p|\E(X_{0,l_1}X_{0,l_3})\E(X_{0,l_2}X_{0,l_4}) + \E(X_{0,l_1}X_{0,l_4})\E(X_{0,l_2}X_{0,l_3}) + cum(X_{0,l_1},X_{0,l_2},X_{0,l_3},X_{0,l_4})|
		\\
		&=  \max_{l_1,l_2}\sum_{l_3,l_4 = 1}^p|\Sigma_n{(l_1,l_3)}\Sigma_n{(l_2,l_4)}+\Sigma_n{(l_1,l_4)}\Sigma_n{(l_2,l_3)}+cum(X_{0,l_1},X_{0,l_2},X_{0,l_3},X_{0,l_4})|
		\\
		&\leq 2(\|\Sigma_n\|_1)^2 + \max_{l_1,l_2}\sum_{l_3,l_4 = 1}^p|cum(X_{0,l_1},X_{0,l_2},X_{0,l_3},X_{0,l_4})|.
	\end{align*}
	Thus by condition \ref{cov:ubound} and \ref{cov:cum1}, we have $\|\Sigma_Z\|_2 \leq \|\Sigma_Z\|_1 = o(\|\Sigma_n\|_F^2)$. Together with~\eqref{eq:sigmaZsigman} this yields $\|\Sigma_Z\|_2/\|\Sigma_Z\|_F \leq \|\Sigma_Z\|_1/\|\Sigma_n\|_F^2 \rightarrow 0$.
	
	\bigskip
	
	\textbf{Verification of~\ref{cumulant}} By Theorem 2 in {\cite{rosenblatt2012stationary}}, we know that
	\[
	cum(Z_{0,k_1},Z_{0,k_2},\change{...},Z_{0,k_h}) = \sum_{\nu = \{\nu_1,...,\nu_L\}}cum(X_{0,l_{i,j}},(i,j)\in\nu_1)\cdots cum(X_{0,l_{i,j}},(i,j)\in\nu_L)
	\]
	where the summation is over all indecomposable partitions $\nu_1 \cup \cdots \cup \nu_L = \nu$ of the two way table,
	\[
	\begin{matrix}
		(1,1) & (1,2)
		\\
		(2,1) & (2,2)
		\\
		\cdots & \cdots
		\\
		(h,1) & (h,2)
	\end{matrix}
	\]
	Note that for $h$, there are finite number of indecomposable partitions in the $h \times 2$ table. Denote the total number of such partitions as $M$. We have
	\begin{align*}
		\qquad&\sum_{k_1,\change{...},k_h = 1}^{p^2}cum^2(Z_{0,k_1},Z_{0,k_2},\change{...},Z_{0,k_h})
		\\
		&\leq \sum_{l_{i,j} = 1,i=1,\change{...},h,j = 1, 2}^{p}\Big(\sum_{\nu}cum(X_{0,l_{i,j}},(i,j)\in\nu_1)\cdots cum(X_{0,l_{i,j}},(i,j)\in\nu_L)\Big)^2
		\\
		&\leq \sum_{l_{i,j} = 1,i=1,\change{...},h,j = 1, 2}^{p}M^2\sum_{\nu}\Big(cum(X_{0,l_{i,j}},(i,j)\in\nu_1)\cdots cum(X_{0,l_{i,j}},(i,j)\in\nu_L)\Big)^2
		\\
		&= M^2 \sum_{\nu} \sum_{l_{i,j} = 1,i=1,\change{...},h,j = 1, 2}^{p}cum^2(X_{0,l_{i,j}},(i,j)\in\nu_1)\cdots cum^2(X_{0,l_{i,j}},(i,j)\in\nu_L)
		\\
		&= M^2\sum_{\nu}\sum_{l_{i,j}=1,(i,j)\in\nu_1}^{p}\cdots \sum_{l_{i,j}=1,(i,j)\in\nu_L}^{p}cum^2(X_{0,l_{i,j}},(i,j)\in\nu_1)\cdots cum^2(X_{0,l_{i,j}},(i,j)\in\nu_L)
		\\
		&=M^2\sum_{\nu}\prod_{k = 1}^{|\nu|}\Big(\sum_{l_{i,j}=1,(i,j)\in\nu_k}^{p}cum^2(X_{0,l_{i,j}},(i,j)\in\nu_k)\Big)
		\\
		&\leq M^3 C^M \prod_{i = 1}^{|\nu|}\|\Sigma_n\|_F^{|\nu_i|}
		\\
		& = M^3C^M\|\Sigma_n\|_F^{2h},
	\end{align*}
	by condition \ref{cov:cum2} where the third line in the above derivation follows by the Cauchy-Schwartz inequality. The desired result follows by~\eqref{eq:sigmaZsigman}. This completes the proof of Theorem~\ref{prop:covm}. \hfill $\Box$
	
	

	
	
	\subsection{Proof of Theorem~\ref{weakS}} \label{sec:proofSnwean}

	The proof relies on the following technical result which will be proved in Section~\ref{sec:proofMomLem}
	
	\begin{lemma}\label{lem:Xi1Xi6}
		Under assumption~\ref{cumulant} there exists a constant $C_6<\infty$ such that for all $j_1 \leq i_1,\change{...}, j_6 \leq i_6$,
		\[
		\Big| \sum_{l_1,\change{...},l_6 = 1}^{p}\E[X_{i_1+1,l_1}X_{j_1,l_1}\cdots X_{i_6+1,l_6}X_{j_6,l_6}]  \Big| \leq C_6\|\Sigma\|_F^6.
		\]
	\end{lemma}
	
	\bigskip
	\noindent
	In what follows define
	\[
	S_n(a,b) := \widetilde S_n(\floor{an},\floor{bn}).
	\]
	To prove process convergence in Theorem~\ref{sec:proofSnwean}, we need to establish two results: convergence of the finite-dimensional distributions, i.e.
	\begin{equation} \label{eq:fidi}
		\Big(\frac{\sqrt{2}}{n\|\Sigma\|_F}S_n(a_1,b_1), \frac{\sqrt{2}}{n\|\Sigma\|_F}S_n(a_2,b_2),\change{...}, \frac{\sqrt{2}}{n\|\Sigma\|_F}S_n(a_S,b_S) \Big) \overset{\cD}{\rightarrow} \Big(Q(a_1,b_1),\change{...}, Q(a_S,b_S) \Big)
	\end{equation}
	for any fixed points $(a_1,b_1),...,(a_S,b_S)$, and tightness of the sequence $\frac{\sqrt{2}}{n\|\Sigma\|_F}S_n$. The latter will be established by showing asymptotic equicontinuity in probability, i.e. we will prove that for any $x>0$
	\begin{equation} \label{eq:tight}
		\lim_{\delta\downarrow 0}\limsup_{n \rightarrow \infty}P\Big(\sup_{\|u-v\|_2 \leq \delta} \Big|\frac{\sqrt{2}}{n\|\Sigma\|_F}S_n(u) - \frac{\sqrt{2}}{n\|\Sigma\|_F}S_n(v)\Big| > x\Big) = 0.
	\end{equation}

	\subsubsection{Proof of \eqref{eq:fidi}}

	To simplify notation, we only consider the case $S = 2$, the general case follows by similar arguments. It sufices to show that $\forall \alpha_1, \alpha_2 \in \mathbb{R}$, $a_1 \leq b_1, a_2 \leq b_2, a_1, a_2, b_1, b_2 \in [0,1]$
	\begin{equation} \label{eq:fidiS2}
		\frac{\sqrt{2}}{n\|\Sigma\|_F}(\alpha_1S_n(a_1,b_1) + \alpha_2S_n(a_2,b_2)) \overset{\mathcal{D}}{\longrightarrow} \alpha_1Q(a_1,b_1) + \alpha_2Q(a_2,b_2).
	\end{equation}
	By symmetry it suffices to consider the following three cases: $a_1 \leq a_2 \leq b_2 \leq b_1$, $a_1 \leq a_2 \leq b_1 \leq b_2$ and $a_1 \leq b_1 \leq a_2 \leq b_2$. We will discuss the case $a_1 \leq a_2 \leq b_1 \leq b_2$ first. Consider the decomposition
	\begin{multline*}
		\frac{\sqrt{2}}{n\|\Sigma\|_F}(\alpha_1S_n(a_1,b_1) + \alpha_2S_n(a_2,b_2))
		\\
		= \frac{\sqrt{2}}{n\|\Sigma\|_F}\Big(\alpha_1\sum_{i = \floor{na_1}+1}^{\floor{nb_1}-1} \sum_{j=\floor{na_1}+1}^i X_{i+1}^TX_j+ \alpha_2\sum_{i = \floor{na_2}+1}^{\floor{nb_2}-1} \sum_{j=\floor{na_2}+1}^i  X_{i+1}^TX_j\Big)
		= \sum_{i = \floor{na_1}+1}^{\floor{nb_2}-1}\widetilde{\xi}_{n,i+1},
	\end{multline*}
	where
	\[
	\widetilde{\xi}_{n,i+1} = \left\{\begin{matrix}
		\alpha_1\xi_{1,i+1},& \qquad \text{ if } \floor{na_1}+1 \leq i \leq \floor{na_2};
		\\
		\alpha_1\xi_{1,i+1}+\alpha_2\xi_{2,i+1},  & \qquad \text{ if } \floor{na_2}+1 \leq i \leq \floor{nb_1}-1;
		\\
		\alpha_2\xi_{2,i+1},& \qquad \text{ if } \floor{nb_1} \leq i \leq \floor{nb_2}-1,
	\end{matrix}
	\right.\]
	
	\begin{align*}
		\xi_{1,i+1} &:= \frac{\sqrt{2}}{n\|\Sigma\|_F}X_{i+1}^T\sum_{j = \floor{na_1}+1}^{i}X_j,
	\end{align*}
	and	
	\begin{align*}
		\xi_{2,i+1} &:= \frac{\sqrt{2}}{n\|\Sigma\|_F}X_{i+1}^T\sum_{j = \floor{na_2}+1}^{i}X_j.
	\end{align*}
	Define $\cF_i = \sigma(X_i,X_{i-1},\change{...})$. A simple calculation shows that for any fixed $n$ the triangular array $(\sum_{i=1}^{n-1}\xi_{n,\floor{na_1}+i})_{1 \leq i \leq \floor{nb_2} - \floor{na_1} - 1}$ is a mean zero martingale difference sequence with respect to $\cF_i$. To show weak convergence in~\eqref{eq:fidiS2} will apply the martingale CLT (Theorem 35.12 in \cite{billingsley2008}). To this end we need to verify the following two conditions:
	\begin{enumerate}
		{ \item[(1)] $\forall \epsilon > 0, \sum_{i = 1}^{\floor{nb_2} - \floor{na_1} - 1}\E[\widetilde{\xi}_{n,\floor{na_1}+i}^2\1{|\widetilde{\xi}_{n,\floor{na_1}+i}| > \epsilon}|\cF_{\floor{na_1}+i-1}] \overset{p}{\rightarrow} 0$,
			\item[(2)] $V_n = \sum_{i = 1}^{\floor{nb_2} - \floor{na_1} - 1} \E[\widetilde{\xi}_{n,\floor{na_1}+i}^2|\cF_{\floor{na_1}+i-1}] \overset{p}{\rightarrow} \alpha_1^2(b_1-a_1)^2 + \alpha_2^2(b_2-a_2)^2 + 2\alpha_1\alpha_2(b_1-a_2)^2 $.}
	\end{enumerate}
	
	We will prove (1) and (2) above in several steps. First, to prove (1), we shall establish that
	\begin{equation}\label{eq:help1}
		\sum_{i = \floor{na_1}+1}^{\floor{nb_1}-1}\E[\widetilde{\xi}_{n,i+1}^4] \rightarrow 0.
	\end{equation}
	{For a proof of (2), consider the decomposition
		\begin{align*}
			V_n &= \sum_{i = \floor{na_1}+1}^{\floor{nb_2}-1}\E[\widetilde{\xi}_{i+1}^2|\cF_i]
			\\
			&= \alpha_1^2\sum_{i = \floor{na_1}+1}^{\floor{na_2}}\E[\xi_{1,i+1}^2|\cF_i] + \sum_{i = \floor{na_2}+1}^{\floor{nb_1}-1}\E[(\alpha_1\xi_{1,i+1} + \alpha_2\xi_{2,i+1})^2|\cF_i] + \alpha_2^2\sum_{i = \floor{nb_1}}^{\floor{nb_2}-1}\E[\xi_{2,i+1}^2|\cF_i]
			\\
			&=\alpha_1^2\sum_{i = \floor{na_1}+1}^{\floor{nb_1}-1}\E[\xi_{1,i+1}^2|\cF_i] + \alpha_2^2\sum_{i = \floor{na_2}+1}^{\floor{nb_2}-1}\E[\xi_{2,i+1}^2|\cF_i] + 2\alpha_1\alpha_2\sum_{i = \floor{na_2}+1}^{\floor{nb_1}-1}\E[\xi_{1,i+1}\xi_{2,i+1}|\cF_i]
			\\
			& =: \alpha_1^2V_{1,n} + \alpha_2^2V_{2,n} + 2\alpha_1\alpha_2V_{3,n}.
		\end{align*}
		We will show that
		\begin{align}
			V_{1,n} &\overset{p}{\rightarrow} (b_1-a_1)^2, \label{eq:V1n}
			\\
			V_{2,n} &\overset{p}{\rightarrow} (b_2-a_2)^2, \label{eq:V2n}
		\end{align}
		and
		\begin{align}
			V_{3,n} &\overset{p}{\rightarrow} (b_1-a_2)^2. \label{eq:V3n}
		\end{align}
	}
	
	For other cases of $a_1,a_2,b_1,b_2$, arguments are similar. For example, we assume $a_1 \leq a_2 \leq b_2 \leq b_1$, then
	
	\begin{multline*}
		\frac{\sqrt{2}}{n\|\Sigma\|_F}(\alpha_1S_n(a_1,b_1) + \alpha_2S_n(a_2,b_2))
		\\
		= \frac{\sqrt{2}}{n\|\Sigma\|_F}\Big(\alpha_1\sum_{i = \floor{na_1}+1}^{\floor{nb_1}-1} \sum_{j=\floor{na_1}+1}^i X_{i+1}^TX_j+ \alpha_2\sum_{i = \floor{na_2}+1}^{\floor{nb_2}-1} \sum_{j=\floor{na_2}+1}^i  X_{i+1}^TX_j\Big)
		= \sum_{i = \floor{na_1}+1}^{\floor{nb_2}-1}\widehat{\xi}_{n,i+1},
	\end{multline*}
	where
	\[
	\widehat{\xi}_{n,i+1} = \left\{\begin{matrix}
		\alpha_1\xi_{1,i+1},& \qquad \text{ if } \floor{na_1}+1 \leq i \leq \floor{na_2},
		\\
		\alpha_1\xi_{1,i+1}+\alpha_2\xi_{2,i+1},  & \qquad \text{ if } \floor{na_2}+1 \leq i \leq \floor{nb_1}-1,
		\\
		\alpha_1\xi_{2,i+1},& \qquad \text{ if } \floor{nb_1} \leq i \leq \floor{nb_2}-1.
	\end{matrix}
	\right.
	\]
	
	Then similar arguments can be applied. This is  the same for $a_1 \leq b_1 \leq a_2 \leq b_2$.
	
	\noindent
	\textbf{Proof of~\eqref{eq:help1}} Observe that
	\begin{align*}
		\sum_{i = \floor{na_1}+1}^{\floor{nb_2}-1}\E[\widetilde{\xi}_{i+1}^4] &= \alpha_1^4\sum_{i = \floor{na_1}+1}^{\floor{na_2}}\E[\xi_{1,i+1}^4] + \sum_{i = \floor{na_2}+1}^{\floor{nb_1}-1}\E[(\alpha_1\xi_{1,i+1}+\alpha_2\xi_{2,i+1})^4 ] + \alpha_2^4\sum_{i = \floor{nb_1}}^{\floor{nb_2-1}}\E[\xi_{2,i+1}^4]
		\\
		& \leq 8\alpha_1^4\sum_{i = \floor{na_1}+1}^{\floor{nb_1}-1}\E[\xi_{1,i+1}^4] + 8\alpha_2^4\sum_{i = \floor{na_2}+1}^{\floor{nb_2}-1}\E[\xi_{2,i+1}^4].
	\end{align*}
	Since the $X_i$ are i.i.d it follows that
	\begin{align*}
		&\E[\xi_{1,i+1}^4]
		\\
		&= \frac{4}{n^4\|\Sigma\|_F^4}\E\Big[\sum_{j_1,...,j_4 = \floor{na_1}+1}^{i} X_{i+1}^TX_{j_1}X_{j_2}^TX_{i+1}X_{i+1}^TX_{j_3}X_{j_4}^TX_{i+1}\Big]
		\\
		&= \frac{4}{n^4\|\Sigma\|_F^4}\sum_{j_1,...,j_4 = \floor{na_1}+1}^{i}\sum_{l_1,...,l_4 = 1}^{p}
		\E\Big[X_{i+1,l_1}X_{i+1,l_2}X_{i+1,l_3}X_{i+1,l_4}\Big]\E\Big[X_{j_1,l_1}X_{j_2,l_2}X_{j_3,l_3}X_{j_4,l_4}\Big]
		\\
		&= \frac{4}{n^4\|\Sigma\|_F^4}\sum_{j = \floor{na_1}+1}^{i}\sum_{l_1,...,l_4 = 1}^{p} \E\Big[X_{i+1,l_1}X_{i+1,l_2}X_{i+1,l_3}X_{i+1,l_4}\Big]\E\Big[X_{j,l_1}X_{j,l_2}X_{j,l_3}X_{j,l_4}\Big]
		\\
		&\quad+\frac{12}{n^4\|\Sigma\|_F^4}\sum_{j_1,j_2 = \floor{na_1}+1}^{i}\sum_{l_1,...,l_4 = 1}^{p} \E\Big[X_{i+1,l_1}X_{i+1,l_2}X_{i+1,l_3}X_{i+1,l_4}\Big]\E\Big[X_{j_1,l_1}X_{j_1,l_2}\Big]\E\Big[X_{j_2,l_3}X_{j_2,l_4}\Big].
	\end{align*}
	Next observe that
	\begin{align*}
		&\E\left[X_{i+1,l_1}X_{i+1,l_2}X_{i+1,l_3}X_{i+1,l_4}\right]
		\\
		&= \E\left[X_{i+1,l_1}X_{i+1,l_2}\right]\E\left[X_{i+1,l_3}X_{i+1,l_4}\right]\ + \E\left[X_{i+1,l_1}X_{i+1,l_3}\right]\E\left[X_{i+1,l_2}X_{i+1,l_4}\right]
		\\
		&\quad +\E\left[X_{i+1,l_1}X_{i+1,l_4}\right]\E\left[X_{i+1,l_2}X_{i+1,l_3}\right]\ + cum(X_{i+1,l_1}, X_{i+1,l_2},X_{i+1,l_3},X_{i+1,l_4})
		\\
		&= \Sigma_{l_1,l_2}\Sigma_{l_3,l_4} + \Sigma_{l_1,l_3}\Sigma_{l_2,l_4} + \Sigma_{l_1,l_4}\Sigma_{l_2,l_3} + cum(X_{i+1,l_1}, X_{i+1,l_2},X_{i+1,l_3},X_{i+1,l_4}),
	\end{align*}
	where $\Sigma_{l_1,l_2}$ is the $(l_1,l_2)$ component of $\Sigma$ and $cum(X_{i+1,l_1},X_{i+1,l_2},X_{i+1,l_3},X_{i+1,l_4})$ is the fourth order joint cumulant of $X_{i+1,l_1},X_{i+1,l_2},X_{i+1,l_3},X_{i+1,l_4}$. Thus by Cauchy-Schwartz inequality and Assumption~\ref{cumulant} we have for $C$ from~\ref{cumulant}
	\begin{align*}
		&\qquad \Big|\sum_{l_1,l_2,l_3,l_4 = 1}^{p} \E\Big[X_{i+1,l_1}X_{i+1,l_2}X_{i+1,l_3}X_{i+1,l_4}\Big]\E\Big[X_{j,l_1}X_{j,l_2}X_{j,l_3}X_{j,l_4}\Big] \Big|
		\\
		&= \Big|\sum_{l_1,l_2,l_3,l_4 = 1}^{p} \Big[\Sigma_{l_1,l_2}\Sigma_{l_3,l_4} + \Sigma_{l_1,l_3}\Sigma_{l_2,l_4} + \Sigma_{l_1,l_4}\Sigma_{l_2,l_3} + cum(X_{i+1,l_1}, X_{i+1,l_2},X_{i+1,l_3},X_{i+1,l_4})\Big]
		\\
		&{ \qquad \times\Big[\Sigma_{l_1,l_2}\Sigma_{l_3,l_4} + \Sigma_{l_1,l_3}\Sigma_{l_2,l_4} + \Sigma_{l_1,l_4}\Sigma_{l_2,l_3} + cum(X_{j,l_1}, X_{j,l_2},X_{j,l_3},X_{j,l_4})\Big] \Big|}
		\\
		&\leq 9\Big(\sum_{l_1,l_2,l_3,l_4 = 1}^{p}\Sigma_{l_1,l_2}^2\Sigma_{l_3,l_4}^2\Big)^{1/2}\Big(\sum_{l_1,l_2,l_3,l_4 = 1}^{p}\Sigma_{l_1,l_3}^2\Sigma_{l_2,l_4}^2\Big)^{1/2}
		\\
		&\quad +6\Big(\sum_{l_1,l_2,l_3,l_4 = 1}^{p}cum^2(X_{i+1,l_1},X_{i+1,l_2},X_{i+1,l_3},X_{i+1,l_4})\Big)^{1/2}\Big(\sum_{l_1,l_2,l_3,l_4 = 1}^{p}\Sigma_{l_1,l_2}^2\Sigma_{l_3,l_4}^2\Big)^{1/2}
		\\
		&\quad + \sum_{l_1,l_2,l_3,l_4 = 1}^{p}cum^2(X_{i+1,l_1},X_{i+1,l_2},X_{i+1,l_3},X_{i+1,l_4})\\
		& \leq 16(C \vee 1)^2\|\Sigma\|_F^4(1 + o(1)),
	\end{align*}
	where the last line follows from Assumption~\ref{cumulant} with $C$ from that assumption. Similarly we have
	\begin{align*}
		\sum_{l_1,l_2,l_3,l_4 = 1}^{p} \E\left[X_{i+1,l_1}X_{i+1,l_2}X_{i+1,l_3}X_{i+1,l_4}\right]\E\left[X_{j,l_1}X_{j,l_2}\right]\E\left[X_{j,l_3}X_{j,l_4}\right]
		\leq 4(C \vee 1)^2\|\Sigma\|_F^4(1 + o(1)).
	\end{align*}
	Combining the above results we have
	\begin{multline*}
		\sum_{i = \floor{na_1}+1}^{\floor{nb_1}-1}\E[\xi_{n,i+1}^4] \leq \frac{4}{n^4}\sum_{i = \floor{na_1}+1}^{\floor{nb_1}-1}\sum_{j = \floor{na_1}+1}^{i}16(C \vee 1)^2(1+o(1))
		\\
		+ \frac{12}{n^4}\sum_{i = \floor{na_1}+1}^{\floor{nb_1}-1}\sum_{j_1,j_2 = \floor{na_1}+1}^{i}4 (C \vee 1)^2(1+o(1))
		= o(1).
	\end{multline*}
	The bound
	\[
	\sum_{i = \floor{na_2}+1}^{\floor{nb_2}-1}\E[\xi_{2,i+1}^4] = o(1)
	\]
	follows by similar arguments and this completes the proof of~\eqref{eq:help1}.
	
	\bigskip
	
	\noindent
	\textbf{Proof of~\eqref{eq:V1n} and~\eqref{eq:V2n}} Since both statements follow by the same arguments we will only give details for the proof of~\eqref{eq:V1n}. Observe that
	\begin{align*}
		V_{1,n} &= \sum_{i = \floor{na_1}+1}^{\floor{nb_1}-1}\E[\xi_{1,i+1}^2|\cF_i]
		\\
		&= \frac{2}{n^2\|\Sigma\|_F^2}\sum_{i = \floor{na_1}+1}^{\floor{nb_1}-1}\E\left[\sum_{j_1,j_2 = \floor{na_1}+1}^{i} X_{i+1}^TX_{j_1}X_{j_2}^TX_{i+1} \Big|\cF_{i}\right]
		\\
		&= \frac{2}{n^2\|\Sigma\|_F^2}\sum_{i = \floor{na_1}+1}^{\floor{nb_1}-1}\sum_{j_1,j_2 = \floor{na_1}+1}^{i} tr\left(\E\left[X_{j_1}X_{j_2}^TX_{i+1}X_{i+1}^T|\cF_{i}\right]\right)
		\\
		&= \frac{2}{n^2\|\Sigma\|_F^2}\sum_{i = \floor{na_1}+1}^{\floor{nb_1}-1}\sum_{j_1,j_2 = \floor{na_1}+1}^{i} tr\left(X_{j_1}X_{j_2}^T\Sigma\right)
		\\
		&= \frac{2}{n^2\|\Sigma\|_F^2}\sum_{i = \floor{na_1}+1}^{\floor{nb_1}-1}\sum_{j = \floor{na_1}+1}^{i}X_{j}^T\Sigma X_{j} + \frac{2}{n^2\|\Sigma\|_F^2} {\sum_{i = \floor{na_1}+1}^{\floor{nb_1}-1}} \sum_{j_1,j_2 = \floor{na_1}+1}^{i} X_{j_2}^T\Sigma X_{j_1}
		\\
		&=: V_{n,1}^{(1)} + V_{n,1}^{(2)}.
	\end{align*}
	For $V_{n,1}^{(1)}$ we have
	\begin{align*}
		& \E\left[(V_{n,1}^{(1)} - (b_1-a_1)^2)^2\right]
		\\
		&= \E \left[\Big(\frac{2}{n^2\|\Sigma\|_F^2}\sum_{j = \floor{na_1}+1}^{\floor{nb_1}-1} (\floor{nb_1}-1-j) X_j^T\Sigma X_j - (b_1-a_1)^2\Big)^2\right]
		\\
		&= \frac{4}{n^4\|\Sigma\|_F^4}\E\left[\sum_{j,j' = \floor{na_1}+1}^{\floor{nb_1}-1}(\floor{nb_1}-1-j)(\floor{nb_1}-1-j')X_j^T\Sigma X_j X_{j'}^T \Sigma X_{j'}\right]
		\\
		&\quad - \frac{4(b_1-a_1)^2}{n^2\|\Sigma\|_F^2}\E\left[\sum_{j = \floor{na_1}+1}^{\floor{nb_1}-1}(\floor{nb_1}-1-j)X_j^T\Sigma X_j\right] + (b_1-a_1)^4
		\\
		&= \frac{4}{n^4\|\Sigma \|_F^4}\sum_{j = \floor{na_1}+1}^{\floor{nb_1}-1}(\floor{nb_1}-1-j)^2\E\Big[X_j^T\Sigma X_j X_j^T \Sigma X_j\Big] + o(1)\\
		&\leq \frac{4}{n^2\|\Sigma \|_F^4}\sum_{j = \floor{na_1}+1}^{\floor{nb_1}-1}\E\Big[X_j^T\Sigma X_j X_j^T \Sigma X_j\Big] + o(1).
	\end{align*}
	{Note that
		\begin{align*}
			\E\left[X_j^T\Sigma X_jX_j^T\Sigma X_j\right] &= \left|\sum_{l_1,\change{...},l_4 = 1}^{p}\E\left[X_{j,l_1}X_{j,l_2}X_{j,l_3}X_{j,l_4}\Sigma_{j_1,j_2}\Sigma_{j_3,j_4}\right]\right|\\
			&=  \left|\sum_{l_1,\change{...},l_4 = 1}^{p}\Sigma_{l_1,l_2}^2\Sigma_{l_3,l_4}^2\right|+  \left|\sum_{l_1,\change{...},l_4 = 1}^{p}cum(X_{0,l_1},X_{0,l_2},X_{0,l_3},X_{0,l_4})\Sigma_{l_1,l_2}\Sigma_{l_3,l_4}\right|\\
			&+\left|\sum_{l_1,\change{...},l_4 = 1}^{p}\Sigma_{l_1,l_4}\Sigma_{l_2,l_3}\Sigma_{l_1,l_2}\Sigma_{l_3,l_4}\right| +
			\left|\sum_{l_1,\change{...},l_4 = 1}^{p}\Sigma_{l_1,l_3}\Sigma_{l_2,l_4}\Sigma_{l_1,l_2}\Sigma_{l_3,l_4}\right|\\
			&\lesssim \|\Sigma\|_F^4,
		\end{align*}
		where the last inequality is a direct consequence of Cauchy-Schwartz inequality and Assumption \ref{cumulant}. Combining with previous results we have
		$$\E\left[\left(V_{n,1}^{(1)} - (b_1 - a_1)^2\right)^2\right] \leq O(\frac{1}{n}) + o(1) \rightarrow 0.$$
		This implies $V_{n,1}^{(1)} \overset{p}{\rightarrow} (b_1-a_1)^2$.}
	Moreover, {for $j_1\neq j_2, j_3 \neq j_4$ we have $\E\left[X_{j_2}^T\Sigma X_{j_1} X_{j_4}^T\Sigma X_{j_3}\right] = 0$ if $j_1 \notin\{j_3,j_4\}$ or $j_2 \notin\{j_3,j_4\}$ and $\E\left[X_{j_2}^T\Sigma X_{j_1} X_{j_4}^T\Sigma X_{j_3}\right] = tr(\Sigma^4)$ otherwise. Hence}
	\begin{align*}
		\E[(V_{n,1}^{(2)})^2] &= \frac{4}{n^4\|\Sigma\|_F^4}\sum_{i,i' = \floor{na_1}+1}^{\floor{nb_1}-1}\sum_{\substack{j_1,j_2 = \floor{na_1}+1\\j_1\neq j_2}}^{i} \sum_{\substack{j_3,j_4 = \floor{na_1}+1\\ j_3\neq j_4}}^{i'} \E\left[X_{j_2}^T\Sigma X_{j_1} X_{j_4}^T\Sigma X_{j_3}\right]
		\\
		&{\leq \frac{8}{n^4\|\Sigma\|_F^4}\sum_{i = \floor{na_1}+1}^{\floor{nb_1}-1}\sum_{i' = \floor{na_1}+1}^{i}\sum_{j_1 = \floor{na_1}+1}^{i'}\sum_{j_2 \neq j_1}^{i'} tr(\Sigma^4) }
		\\
		&\leq \frac{tr(\Sigma^4)}{\|\Sigma\|_F^4}O(1) \rightarrow 0.
	\end{align*}
	Combining results~\eqref{eq:V1n} follows.
	
	\bigskip
	
	\noindent
	\textbf{Proof of~\eqref{eq:V3n}} {Observe the decomposition
		\begin{align*}
			V_{3,n} &= \sum_{i = \floor{na_2}+1}^{\floor{nb_1}-1}\E[\xi_{1,i+1}\xi_{2,i+1}|\cF_i]
			\\
			&= \frac{2}{n^2\|\Sigma\|_F^2}\sum_{i = \floor{na_2}+1}^{\floor{nb_1}-1}\sum_{j_1 = \floor{na_1}+1}^{i}\sum_{j_2 = \floor{na_2}+1}^{i}X_{j_1}^T\Sigma X_{j_2}
			\\
			&= \frac{2}{n^2\|\Sigma\|_F^2}\sum_{i = \floor{na_2}+1}^{\floor{nb_1}-1}\sum_{j_1 = \floor{na_1}+1}^{\floor{na_2}}\sum_{j_2 = \floor{na_2}+1}^{i}X_{j_1}^T\Sigma X_{j_2}+\frac{2}{n^2\|\Sigma\|_F^2}\sum_{i = \floor{na_2}+1}^{\floor{nb_1}-1}\sum_{j_1, j_2 = \floor{na_2}+1}^{i} X_{j_1}^T\Sigma X_{j_2}.
		\end{align*}
		{Note that for $j_1\neq j_2, j_3 \neq j_4$ we have $\E\left[X_{j_2}^T\Sigma X_{j_1} X_{j_4}^T\Sigma X_{j_3}\right] = 0$ if $j_1 \notin\{j_3,j_4\}$ or $j_2 \notin\{j_3,j_4\}$ and $\E\left[X_{j_2}^T\Sigma X_{j_1} X_{j_4}^T\Sigma X_{j_3}\right] = tr(\Sigma^4)$ otherwise. Hence we obtain for the first term}
		{ \begin{align*}
				&\qquad \E\Big[\Big(\frac{2}{n^2\|\Sigma\|_F^2}\sum_{i = \floor{na_2}+1}^{\floor{nb_1}-1}\sum_{j_1 = \floor{na_1}+1}^{\floor{na_2}}\sum_{j_2 = \floor{na_2}+1}^{i}X_{j_1}^T\Sigma X_{j_2}\Big)^2\Big]
				\\
				&= \frac{4}{n^4\|\Sigma\|_F^4}\sum_{i,i' = \floor{na_2}+1}^{\floor{nb_1}-1}\sum_{j_1 = \floor{na_1}+1}^{\floor{na_2}}\sum_{j_2 = \floor{na_2}+1}^{i}\sum_{j_1' = \floor{na_1}+1}^{\floor{na_2}}\sum_{j_2' = \floor{na_2}+1}^{i'}\E[X_{j_2}^T\Sigma X_{j_1}X_{j_1'}^T \Sigma X_{j_2'}]
				\\
				&\leq\frac{8}{n^4\|\Sigma\|_F^4} (\floor{nb_1}-\floor{na_2}-1)^2(\floor{na_2} - \floor{na_1})(\floor{nb_1} - \floor{na_2} - 1)tr(\Sigma^4)
				\\
				&\leq \frac{tr(\Sigma^4)}{\|\Sigma\|_F^4} O(1)\rightarrow 0.
		\end{align*}}
		Thus the first term in the decomposition of $V_{3,n}$ is $o_p(1)$. The second term is of the same structure as $V_{1,n}$, and it follows that
		\[
		\frac{2}{n^2\|\Sigma\|_F^2}\sum_{i = \floor{na_2}+1}^{\floor{nb_1}-1}\sum_{j_1 = \floor{na_2}+1}^{i}\sum_{j_2 = \floor{na_2}+1}^{i}X_{j_1}^T\Sigma X_{j_2} \overset{p}{\rightarrow} (b_1-a_2)^2.
		\]
		This yields~\eqref{eq:V3n}. Thus~\eqref{eq:help1}-\eqref{eq:V3n} are established and this completes the proof of~\eqref{eq:fidi}. \hfill $\Box$}
	
	\subsubsection{Proof of~\eqref{eq:tight}}
	
	The proof will rely on the following bound for the increments of $S_n$: there exists a constant $\tilde C < \infty$ such that for all $n \geq 2$ and all $a,b,c,d \in [0,1]$ we have
	\begin{equation}\label{eq:momb}
		\E\Big[\frac{1}{n^6\|\Sigma\|_F^6}(S_n(a,b) - S_n(c,d))^6\Big] \leq \tilde C(\|(a,b) - (c,d)\|_2^{3} + n^{-3}).
	\end{equation}
	This bound will be established at the end of the proof. For the remainder of the proof, define
	\[
	B_n(a,b) := \frac{\sqrt{2}}{n\|\Sigma\|_F}S_n(a,b).
	\]
	Note that $B_n(u)$ has a piece-wise constant structure, more precisely we have for any $u \in [0,1]^2$, $B_n(u) = B_n(\floor{nu}/n)$ (here, $\floor{nu}$ is understood component-wise). Define the index set $T_n := \{(i/n,j/n): i,j = 0,...,n\}$. Then
	\[
	\sup_{\|u-v\| \leq \delta} \Big|\frac{\sqrt{2}}{n\|\Sigma\|_F}S_n(u) - \frac{\sqrt{2}}{n\|\Sigma\|_F}S_n(v)\Big| \leq \sup_{u,v \in T_n: \|u-v\| \leq \delta + 2n^{-1}} \Big|B_n(u) - B_n(v)\Big|.
	\]
	Consider the metric (on the set $T_n$) $d(u,v) = \|u-v\|^{1/2}$. From~\eqref{eq:momb} and the definition of $B_n$ we obtain the existence of a constant $C < \infty$ such that for all $n \geq 2$ and all $u,v \in T_n$
	\[
	\E[|B_n(u) - B_n(v)|^6] \leq C\Big(\|u-v\|^{3} + \frac{1}{2^{12}}n^{-3}\Big),
	\]
	which implies
	\[
	\|B_n(u) - B_n(v)\|_{L_6} \leq 2C\|u-v\|^{1/2} \quad \forall~ u,v\in T_n: \|u-v\|^{3} \geq \frac{1}{2^{12}}n^{-3}.
	\]
	Note that the packing number of $T_n$ with respect to the metric $d$ satisfies
	\[
	D_{T_n}(\epsilon,d) \leq D_{[0,1]^2}(\epsilon,d) \leq C_D\epsilon^{-4}
	\]
	for some constant $C_D < \infty$. Now apply Lemma~A.1 from~\cite{kvdh2016} with $\Psi(x) = x^6$, $T = T_n$, $d(u,v) = \|u-v\|^{1/2}$, $\bar\eta = n^{-1/2}/2$ to find that for any $\eta \geq \bar \eta$ there exists a random variable $R_n(\eta,\delta)$ such that
	\[
	\sup_{d(u,v)\leq (\delta + 2n^{-1})^{1/2}}|B_n(u) - B_n(v)| \leq R_n(\eta,\delta) + 2\sup_{u,v \in T_n:~ d(u,v) \leq \bar{\eta}}|B_n(u) - B_n(v)|,
	\]
	and
	\[
	\|R_n(\eta,\delta)\|_{6} \leq K\Big[\int_{\bar\eta/2}^{\eta}(D_{T_n}(\epsilon,d))^{1/6}d\epsilon + ((\delta + 2n^{-1})^{1/2} + 2\bar\eta)(D_{T_n}^2(\eta,d))^{1/6}\Big].
	\]
	for some constant $K$ independent of $\delta, \eta, n$. Next, observe that
	\[
	d(u,v) \leq \bar{\eta} \Leftrightarrow \|u-v\| \leq n^{-1}/4,
	\]
	and since $\inf_{u,v\in T_n, u\neq v}\|u-v\| \geq n^{-1}$ it follows that $u,v \in T_n:~ d(u,v) \leq \bar{\eta} $ implies $u=v$ (recall that $T_n$ is discrete) and thus the supremum vanishes and we obtain
	\begin{equation} \label{eq:WnRn}
		\sup_{d(u,v)\leq(\delta + 2n^{-1})^{1/2}}|B_n(u) - B_n(v)| \leq R_n(\eta,\delta).
	\end{equation}
	Now a simple computation shows that
	\begin{align*}
		&\int_{\bar\eta/2}^{\eta}(D_{T_n}(\epsilon,d))^{1/6}d\epsilon + ((\delta + 2n^{-1})^{1/2} + 2\bar\eta) (D_{T_n}^2(\eta,d))^{1/6}
		\\
		&\lesssim \int_{0}^{\eta}\epsilon^{-2/3}d\epsilon + (\delta^{1/2} + n^{-1/2})\eta^{-4/3}
		\\
		&= 3\eta^{1/3} + (\delta^{1/2} + n^{-1/2})\eta^{-4/3}.
	\end{align*}
	Apply the Markov inequality to find that for any $x > 0$
	\[
	\lim_{\delta\downarrow 0}\limsup_{n \rightarrow \infty} P(|R_n(\eta,\delta)| > x) \leq \frac{3\eta^{1/3}}{x^6}.
	\]
	Since $\eta$ was arbitrary, it follows that
	\[
	\lim_{\delta\downarrow 0}\limsup_{n \rightarrow \infty} P(|R_n(\eta,\delta)| > x) = 0.
	\]
	Combined with~\eqref{eq:WnRn} this implies~\eqref{eq:tight}. Hence it remains to establish~\eqref{eq:momb}.
	
	\bigskip
	
	\noindent
	\textbf{Proof of~\eqref{eq:momb}} We shall assume $a < c < d < b$, proofs in all other cases are similar. By definition of $S_n$,
	\begin{align*}
		S_n(a,b) - S_n(c,d) &= \sum_{i = \floor{na}+1}^{\floor{nb}-1}\sum_{j = \floor{na}+1}^{i} X_{i+1}^TX_j - \sum_{i = \floor{nc}+1}^{\floor{nd}-1}\sum_{j = \floor{nc}+1}^{i} X_{i+1}^TX_j\\
		&= \sum_{i = \floor{na}+1}^{\floor{nc}-1}\sum_{j = \floor{na}+1}^{i}X_{i+1}^TX_{j} + \sum_{i = \floor{nc}}^{\floor{nd}-1}\sum_{j = \floor{na}+1}^{\floor{nc}}X_{i+1}^TX_j\\
		&\qquad +\sum_{i = \floor{nd}}^{\floor{nb}-1}\sum_{j = \floor{na}+1}^{\floor{nc}}X_{i+1}^TX_j + \sum_{i = \floor{nd}}^{\floor{nb}-1}\sum_{j = \floor{nc}+1}^{\floor{nd}}X_{i+1}^TX_j\\
		&\qquad+\sum_{i = \floor{nd}+1}^{\floor{nb}-1}\sum_{j = \floor{nd}+1}^{i}X_{i+1}^TX_j\\
		&= A + B + C + D + E.
	\end{align*}
	
	Note that $A = S_n(a,c)$ and $E = S_n(b,d)$ and that $B$, $C$ and $D$ share the same structure. Applying H\"older's inequality yields
	\[
	(A+B+C+D+E)^6 \lesssim (A^6+B^6+C^6 + D^6 + E^6),
	\]
	and thus it suffices to show that
	\[
	\E\Big[\frac{1}{n^6\|\Sigma\|_F^6}(A^6+B^6+C^6 + D^6 + E^6)\Big] \lesssim (\|(a,b) - (c,d)\|^{3} + n^{-3}).
	\]
	Apply Lemma~\ref{lem:Xi1Xi6} to obtain
	\begin{align*}
		\E[S_n(a,c)^6] &=  \sum_{i_1,\change{...},i_6  = \floor{na}+1}^{\floor{nc}-1}\sum_{j_1 = \floor{na}+1}^{i_1}\cdots\sum_{j_6 = \floor{na}+1}^{i_6}
		\sum_{l_1,\change{...},l_6 = 1}^{p}\E[X_{i_1+1,l_1}X_{j_1,l_1}\cdots X_{i_6+1,l_6}X_{j_6,l_6}]\\
		&\leq C_6(\floor{nc}-\floor{na}-1)^6\|\Sigma\|_F^6.
	\end{align*}
	By definition, $S_n(a,c)  = 0$ if $\floor{nc} - \floor{na} < 2$. If $\floor{nc} - \floor{na} \geq 2$, which implies $c - a > 1/n$,
	\[
	\floor{nc} - \floor{na} - 1 \leq nc - na + (\floor{na} - na) - 1 \leq n(c-a).
	\]
	Thus
	\[
	\frac{1}{n^6\|\Sigma\|_F^6}\E[S_n(a,c)^6] \leq C_6(c-a)^6.
	\]
	Exactly the same argument can be used to bound $\E[S_n(a,c)^6]$. Next observe that we have for $\floor{nd} - \floor{nc} > 2$,
	\begin{align*}
		\frac{1}{n^6\|\Sigma\|_F^6}\E[B^6] &\leq C_6(\floor{nd} - \floor{nc})^3(\floor{nc} - \floor{na})^3/n^6
		\\
		&\leq C_6(\floor{nc} - \floor{na})^3/n^3
		\\
		&\leq C_6(nc - na + (\floor{na} - na))^3/n^3
		\\
		&\leq C_6(c-a+1/n)^3
		\\
		&\lesssim ((c-a)^3 + n^{-3}).
	\end{align*}
	Thus by summarizing the above steps, we have
	\begin{align*}
		&\qquad \E[\frac{1}{n^6\|\Sigma\|_F^6}(S_n(a,b) - S_n(c,d))^6]
		\\
		&\lesssim ((c-a)^6 + (c-a)^3 + (b-d)^3 + (b-d)^6 + (b-d)^3 + n^{-3})
		\\
		&\lesssim ((c-a)^3 + (b-d)^3 + n^{-3})
		\\
		&\leq \|(c-a),(b-d)\|_2^3 + n^{-3},
	\end{align*}
	where the last inequality in the previous line follows from
	\begin{align*}
		((c-a)^3 + (b-d)^3)^2  &=  (c-a)^6+(b-d)^6 + 2(c-a)^3(b-d)^3
		\\
		&= (c-a)^6+(b-d)^6 + (c-a)^2(b-d)^2(2(c-a)(b-d))
		\\
		&\leq (c-a)^6 + (b-d)^6 + (c-a)^2(b-d)^2((c-a)^2 + (b-d)^2)
		\\
		&\leq ((c-a)^2 + (b-d)^2)^3,
	\end{align*}
	which implies
	\[
	(c-a)^3 + (b-d)^3 \leq ((c-a)^2 + (b-d)^2)^{3/2} = \|(c-a),(b-d)\|_2^3.
	\]
	
	\hfill $\Box$
	
	\subsubsection{Proof of Lemma~\ref{lem:Xi1Xi6}}\label{sec:proofMomLem}

	By the generalized H\"older's inequality, we have
	
	\begin{multline*}
		\left|\sum_{l_1,l_2,...,l_6 = 1}^p\E[X_{i_1+1,l_1}X_{j_1,l_1}\cdots X_{i_6+1,l_6}X_{j_6,l_6}]\right|
		=\left|\E[(X_{i_1+1}^TX_{j_1})\cdots (X_{i_6+1}^TX_{j_6})]\right| \\\leq \sqrt[6]{\E[(X_{i_1+1}^TX_{j_1})^6]}\cdots\sqrt[6]{\E[(X_{i_6+1}^TX_{j_6})^6]}
		={\E[(X_{2}^TX_{1})^6]}.
	\end{multline*}
	
	Let $\pi$ be any disjoint partition over the set $\{l_1,l_2,l_3,l_4, l_5, l_6 \}$ such that for any $B \in \pi$, $|B| \neq 1$. Thus,
	\begin{align*}
		& E \left[ (X_{2}^T X_{1})^6 \right] \\
		= &  \sum_{l_1,..., l_6 = 1}^p E[ X_{2,l_1} X_{2,l_2} X_{2,l_3} X_{2,l_4} X_{2,l_5} X_{2, l_6} ] E[X_{1,l_1} X_{1,l_2} X_{1,l_3} X_{1,l_4} X_{1,l_5}X_{1, l_6} ]  \\
		= &  \sum_{l_1,..., l_6 = 1}^p (E[ X_{1,l_1} X_{1,l_2} X_{1,l_3} X_{1,l_4}X_{1,l_5}X_{1, l_6} ])^2  =  \sum_{l_1,..., l_6 = 1}^p \left(\sum_{\pi} \prod_{B \in \pi} cum(X_{1,l_k} : l_k \in B ) \right)^2   \\
		\leq &  C_6\sum_{l_1,..., l_6 = 1}^p  \sum_{\pi} \prod_{B \in \pi} cum^2(X_{1,l_k} : l_k \in B )  \text{ by Cauchy's inequality} \\
		\leq & C_6\sum_{\pi} \prod_{B \in \pi} \left\lbrace \sum_{l_k \in B} \sum_{l_k=1}^p cum^2(X_{1,l_k} : l_k \in B ) \right\rbrace \\
		\leq &   C_6\sum_{\pi} \prod_{B \in \pi} \| \Sigma \|^{|B|}_F \hspace{2cm} \text{ by Assumption~\ref{cumulant}} \\
		\leq & C_6\| \Sigma \|^6_F,
	\end{align*}
	where $C_6 > 0$ is a generic constant that varies from line by line. This completes the proof. \hfill \boxed{}

	\section{Proofs of results for high-dimensional time series}
	\label{sec:appendixHDTS}
	Throughout this section, we assume that the process $X_t$ admits a linear process represenation.
	
	\subsection{Properties of Linear Process}
	Firstly, applying Beveridge Nelson (BN) decomposition in \cite{phillips1992asymptotics}, we have
	\begin{align*}
		X_i = D_i - \varepsilon_i ,
	\end{align*}
	where $D_i = (\sum_{u = 0}^{\infty} c_{u}) \epsilon_i $, 
	$ 
	\widetilde{D}_{i} = \sum_{j =0}^{\infty} (\sum_{u = j+1}^{\infty} c_{u}) \epsilon_{i-j}
	$ and $\varepsilon_i = \widetilde{D}_{i} - \widetilde{D}_{i-1}$. 
	We then state three useful auxiliary lemmas.
	\begin{lemma} \label{lem:cov} Suppose Assumption \ref{ass:main} (\ref{A41A1}, \ref{A41A2}, \ref{A41A5}) is true. Then, for any $h = 2,3,4,5,6$ and $j = 0,1,\change{...}, h$, we have
		\begin{align*}
			& \sum_{l_1,l_2,\change{...},l_h=1}^p |cum( D_{i_1, l_1},\change{...}, D_{i_j, l_j}, \widetilde{D}_{i_{j+1}, l_{j+1}},\change{...},  \widetilde{D}_{i_{h}, l_{h}})|  \lesssim \| \Gamma \|_F^h, \\
			& \sum_{l_1,l_2,\change{...},l_h=1}^p cum^2( D_{i_1, l_1},\change{...}, D_{i_j, l_j}, \widetilde{D}_{i_{j+1}, l_{j+1}},\change{...},  \widetilde{D}_{i_{h}, l_{h}})  \lesssim \| \Gamma \|_F^h.
		\end{align*}
	\end{lemma}

	\begin{lemma} \label{lem:exp}
		Suppose Assumption \ref{ass:main} (\ref{A41A1}, \ref{A41A2}) is true. Then, for some constant $C$ and $0< \rho <1$, we have for $k\leq 7$
		\begin{align*}
			|cum(Z_{i_0},\change{...}, Z_{i_k})| \leq C \rho^{i_{max} - i_{min}},
		\end{align*}
		where $i_{max} = \max\{ i_0, \change{...}, i_k \}$, $i_{min} = \min \{ i_0, \change{...}, i_k \}$ and for each $i \in \{ i_0, \change{...}, i_k \}$, $Z_i$ can be any element from the set $ \{ X_{i,j}, D_{i,j}, \widetilde{D}_{i,j}, \varepsilon_{i,j} \}_{i=i_0, \change{...}, i_k, j = 1, \change{...}, p}$. 
	\end{lemma}
	
	\begin{lemma} \label{lem:main}  Under Assumption \ref{ass:main}, for any  $i \neq j $, we have
		\begin{align*}
			E [ (D_{i}^T D_{j})^6 ]  \lesssim \|\Gamma\|_F^6.
		\end{align*} 
	\end{lemma}
	
	\subsection{Proof of Theorem \ref{thm:main}}\label{sec:pfThmDep}
	Recall that
	\begin{align*}
		\widetilde{H}_{X_n}(k;l,m|\tau)   = \frac{\sqrt{2}}{n^3 \|\Gamma \|_F} \sum\limits_{l \leq j_1 , j_3 \leq k }^{|j_1-j_3|>\tau}\sum\limits_{k+ \tau + 1 \leq j_2, j_4 \leq m }^{|j_2-j_4|>\tau} (X_{j_1} - X_{j_2})^T (X_{j_3} - X_{j_4}). 
	\end{align*}
	For $u=1,2,3,4$, define
	\begin{equation}
		\label{eq:Stilde}
		\widetilde{S}_{X_n}^{u}(k,m| \tau)  = \left\lbrace  \begin{array}{ll}
			\sum\limits_{i=k}^{m-\tau-1}\sum\limits_{j=k}^{i} \widetilde{w}_{i,j}^u X_{i+\tau+1}^T X_j, & m - \tau -1 \geq k; \\
			0, & \text{otherwise, }  
		\end{array} \right. 
	\end{equation}
	where 
	\begin{align*}
		\widetilde{w}_{i,j}^u = \1_{\{ u=1 \} } +\frac{j}{n}  \1_{\{ u=2 \} } + \frac{ i+\tau +1 }{n} \1_{\{ u=3 \} } + \frac{ i+\tau +1 }{n} \frac{j}{n} \1_{\{ u=4 \} } .
	\end{align*}
	Expanding the inner product $ (X_{j_1} - X_{j_2})^T (X_{j_3} - X_{j_4}) $ leads to 
	\begin{multline*}
		\widetilde{H}_{X_n}(k;l,m|\tau) = 2 \frac{(m - 2\tau - k -1)(m-2\tau-k)}{n^2}\frac{\sqrt{2}}{n \|\Gamma \|_F}\widetilde{S}_{X_n}(l,k|\tau)  \\
		+ 2\frac{(k-\tau-l)(k-\tau-l+1)}{n^2}\frac{\sqrt{2}}{n \|\Gamma \|_F} \widetilde{S}_{X_n}(k+\tau+1, m|\tau) \\
		- 2 \frac{\sqrt{2}}{n \|\Gamma \|_F} \frac{1}{n^2} \sum\limits_{l \leq j_1, j_3 \leq k}^{|j_1- j_3| > \tau}X_{j_1}^T \sum\limits_{k+\tau+1 \leq j_2, j_4 \leq m}^{|j_2-j_4| > \tau} X_{j_4}.
	\end{multline*}
	where vectors in the last term admits
	\begin{multline*}
		\sum\limits_{l \leq j_1, j_3 \leq k}^{|j_1- j_3| > \tau}X_{j_1} = (k-\tau)\sum\limits_{j_1=l}^{k-\tau-1} X_{j_1} - (l+\tau) \sum\limits_{j_1=l+\tau+1}^k X_{j_1} - \sum\limits_{j_1=l}^{k-\tau-1} j_1 X_{j_1} + \sum\limits_{j_1=l+\tau+1}^k j_1 X_{j_1}, 
	\end{multline*}
	and
	\begin{multline*}
		\sum\limits_{k+\tau+1 \leq j_2, j_4 \leq m}^{|j_2-j_4|>\tau} X_{j_4} = (m-\tau) \sum\limits_{j_4 = k+\tau+1}^{m-\tau-1}X_{j_4} - (k+2\tau+1)\sum\limits_{j_4 = k+2\tau+2}^m X_{j_4} \\ - \sum\limits_{ j_4 = k+\tau+1 }^{m-\tau-1} j_4 X_{j_4} + \sum\limits_{j_4 = k+2\tau+2}^m j_4 X_{j_4}.
	\end{multline*}
	Thus, we can decompose the last term as 
	\begin{footnotesize}
		\begin{align}
			& \frac{1}{n^2} \sum\limits_{l \leq j_1, j_3 \leq k}^{|j_1- j_3| > \tau}X_{j_1}^T \sum\limits_{k+\tau+1 \leq j_2, j_4 \leq m}^{|j_2-j_4| > \tau} X_{j_4} \notag\\ 
			= &    \frac{(k-\tau)(m-\tau)}{n^2}Q_{X_n}^{1}(l, k-\tau-1;k+\tau+1,m-\tau-1)  -  \frac{(k-\tau)(k+2\tau+1)}{n^2}Q_{X_n }^{1}(l, k-\tau-1;k+2\tau+2,m) \notag\\
			& -  \frac{(l+\tau)(m-\tau)}{n^2} Q_{X_n }^{1}(l+\tau+1, k; k+\tau+1,m-\tau-1) +  \frac{(l+\tau)(k+2\tau+1)}{n^2} Q_{X_n }^{1}(l+\tau+1, k; k+2\tau+2,m) \notag\\
			& -  \frac{(m-\tau)}{n} Q_{X_n}^{2}(l, k-\tau-1; k+\tau+1,m-\tau-1)  +  \frac{(k+2\tau+1)}{n} Q_{X_n}^{2}(l, k-\tau-1; k+2\tau+2,m) \notag\\
			& +  \frac{(m-\tau)}{n} Q_{X_n}^{2}(l+\tau+1, k; k+\tau+1,m-\tau-1)  -  \frac{(k+2\tau+1)}{n} Q_{X_n}^{2}(l+\tau+1, k; k+2\tau+2,m) \notag\\
			& -  \frac{(k- \tau)}{n}Q_{X_n }^{3}(l, k-\tau-1;k+\tau+1,m-\tau-1)  +  \frac{(k- \tau)}{n}Q_{X_n }^{3}(l, k-\tau-1;k+2\tau+2,m) \notag\\
			& + \frac{(l + \tau)}{n} Q_{X_n }^{3}(l+\tau+1, k; k+\tau+1,m-\tau-1)  - \frac{(l + \tau)}{n} Q_{X_n }^{3}(l+\tau+1, k; k+2\tau+2,m) \notag\\
			& + Q_{X_n }^{4}(l, k-\tau-1;k+\tau+1,m-\tau-1) -  Q_{X_n }^{4}(l, k-\tau-1;k+2\tau+2,m) \notag\\
			& -  Q_{X_n }^{4}(l+\tau+1, k;k+\tau+1,m-\tau-1)  + Q_{X_n }^{4}(l+\tau+1, k;k+2\tau+2,m), \label{repr:DXdep}
		\end{align}
	\end{footnotesize}
	where for $w_1 < w_2$, $h_1< h_2$, $w_2 \leq h_1 - \tau$, 
	\begin{multline*}
		Q_{X_n }^{u}(w_1, w_2;h_1,h_2) = \widetilde{S}_{X_n}^{u}(w_1, h_2| \tau) - \widetilde{S}_{X_n}^{u}(w_1,h_1-1| \tau) \\ - \widetilde{S}_{X_n}^{u}(w_2+1,h_2| \tau) + \widetilde{S}_{X_n}^{u}(w_2+1, h_1-1| \tau), 
	\end{multline*}
	otherwise $ Q_{X_n }^{u}(w_1, w_2;h_1,h_2)=0 $. See Figure \ref{fig1} for an illustration of $Q_{X_n }^{u}$ and $\widetilde{S}_{X_n}^{u}$. The above decomposition suggests us to write $\widetilde{H}_{X_n}(k;l,m|\tau)$ as a continuous functional of $\widetilde{S}_{X_n}^{u}(k,m| \tau)$.
	\begin{figure}[t]
		\centering
		\begin{tikzpicture}
			\draw[thick,->] (0,12) -- (12,12) node[anchor=south east] {$ j $};
			\draw[thick, ->] (0,12) -- (0,0) node[anchor=south east] {$ i $};
			\draw[dashed] (0,12) -- (12,0);
			\draw[dashed] (0,11) -- (11,0);
			\draw[dashed] (7,3) -- (12,3);
			\draw (1,1)  -- (11,1) -- (11,11)   -- (1,11)  -- (1,1);
			\draw (1,2)  -- (7,2) -- (7,3)   -- (1,3)  -- (1,1);
			\draw[dashed] (7,3) -- (7,12) node[anchor=south] {$w_2$};
			\draw[dashed] (1,11) -- (1,12) node[anchor=south] {$w_1$};
			\draw[dashed] (1,2) -- (0,2) node[anchor=east] {$h_2$};
			\draw[dashed] (1,3) -- (0,3) node[anchor=east] {$h_1$};
			\node (1) at(3.5,2.5) {$Q_{n }^{u}(w_1, w_2;h_1,h_2)$};
			\node (1) at(3.5,4.5) {$\widetilde{S}_{n}^{u}(w_1,h_1-1| \tau)$};
			\fill[green,nearly transparent] (1,2)  -- (7,2) -- (7,3)   -- (1,3)  -- (1,1);
			\fill[yellow,nearly transparent] (1,3)  -- (8,3) -- (1,10)   -- (1,3) ;
			\draw[decoration={brace,mirror,raise=3pt},decorate]
			(0, 12) -- node[left=6pt] {$\tau$}  (0,11);
		\end{tikzpicture}
		\caption{Illustration of $Q_{X_n }^{u}(w_1, w_2;h_1,h_2)$ (green region) and $ \widetilde{S}_{X_n}^{u}(w_1,h_1-1| \tau) $ (yellow region)}
		\label{fig1}
	\end{figure}
	Thus, for fixed $\eta \in (0,1)$ and $u=1,2,3,4$, the key step is to study the following two parameter processes
	\begin{equation}
		\label{eq:S}
		S_{X_n}^{u}(a,b| \eta)  = 
		\left\lbrace 
		\begin{array}{ll}
			\sum\limits_{i = \lfloor an \rfloor }^{ \lfloor bn \rfloor - \lfloor \eta n \rfloor -1 }  \sum\limits_{j = \lfloor an \rfloor}^{i} w_{i,j}^u X_{i+\lfloor \eta n \rfloor +1 }^T  X_{j}, & 0<a < b - \eta < 1-\eta; \\
			0, & \text{otherwise}.
		\end{array} \right. 
	\end{equation}
	where 
	\begin{align*}
		w_{i,j}^u = \1_{\{ u=1 \} } +\frac{j}{n}  \1_{\{ u=2 \} } + \frac{ i+\lfloor \eta n \rfloor +1 }{n} \1_{\{ u=3 \} } + \frac{ i+\lfloor \eta n \rfloor +1 }{n} \frac{j}{n} \1_{\{ u=4 \}}.
	\end{align*}
	Recall that Beveridge Nelson (BN) decomposition in \cite{phillips1992asymptotics} implies 
	$
	X_i = D_i - \varepsilon_i ,
	$ where $D_i = (\sum_{u = 0}^{\infty} c_{u}) \epsilon_i $, 
	$ 
	\widetilde{D}_{i} = \sum_{j =0}^{\infty} (\sum_{u = j+1}^{\infty} c_{u}) \epsilon_{i-j}
	$ and $\varepsilon_i = \widetilde{D}_{i} - \widetilde{D}_{i-1}$. By applying the BN decomposition, we would have for any $u=1,2,3,4$
	\begin{align} \label{eq:BN}
		\frac{\sqrt{2}}{n \|\Gamma \|_F}  S_{X_n}^{u}(a,b| \eta) =  \frac{\sqrt{2}}{n \|\Gamma \|_F} S_{D_n}^{u}(a,b| \eta) + R_{u}, 
	\end{align}
	where $ S_{D_n}^{u}(a,b| \eta) $ is defined similarly as in equation \eqref{eq:S} and
	it holds in $l^{\infty}([0,1]^2)$ that $ R_{u}  \rightsquigarrow 0$. The proof is postponed to Section \ref{sec:BN}. Consequently, it holds in $ l^{\infty}([0,1]^3) $ that
	\begin{align*}
		H_{X_n}(r;a,b|\eta):= \widetilde{H}_{X_n}(\lfloor rn \rfloor ; \lfloor an \rfloor , \lfloor b n \rfloor | \lfloor \eta n \rfloor) = H_{D_n}(r;a,b|\eta) + o_p(1).
	\end{align*}
	The convergence of marginals $\left(H_{D_n}(r_1; a_1, b_1|\eta), \change{...}, H_{D_n}(r_K; a_K, b_K|\eta)  \right)$ is shown in Section \ref{sec:marginal} and the tightness of $H_{D_n}(r;a,b|\eta)$ follows from the tightness of each $\frac{\sqrt{2}}{n \|\Gamma \|_F} S_{D_n}^{u}(a,b| \eta)$. When $\eta =0$, the tightness of $\frac{\sqrt{2}}{n \|\Gamma \|_F} S_{D_n}^{u}(a,b| 0)$ is given by Equation \eqref{eq:tight}. When $\eta >0$, consider $0< a<c<d<b< 1-\eta $ such that $a-c < \eta$ and $b-d <\eta$, we get
	\begin{align*}
		S_{D_n}^{u}(a,b| \eta) - S_{D_n}^{u}(c,d| \eta) =  \sum_{i = \lfloor c n \rfloor +1}^{\lfloor bn \rfloor - \lfloor \eta n \rfloor -1 } \sum_{j = \lfloor an \rfloor}^{\lfloor cn \rfloor -1}  D_{i+\lfloor \eta n \rfloor +1 }^T  D_{j} + \sum_{i = \lfloor d n \rfloor +1}^{\lfloor bn \rfloor } \sum_{j = \lfloor cn \rfloor}^{\lfloor dn \rfloor- \lfloor \eta n \rfloor -1}  D_{i }^T  D_{j}.
	\end{align*}
	Since $\{ D_i \}_{i=1}^n$ are independent, by applying Lemma \ref{lem:main}, we can obtain
	\begin{align*}
		& E\left[  \left( \frac{\sqrt{2}}{n \|\Gamma \|_F} \sum_{i = \lfloor c n \rfloor +1}^{\lfloor bn \rfloor - \lfloor \eta n \rfloor -1 } \sum_{j = \lfloor an \rfloor}^{\lfloor cn \rfloor -1}  D_{i+\lfloor \eta n \rfloor +1 }^T  D_{j} \right)^6 \right] \\
		\lesssim & \frac{1}{n^6} ( \lfloor bn \rfloor - \lfloor \eta n \rfloor - \lfloor cn \rfloor )^3( \lfloor cn \rfloor - \lfloor an \rfloor)^3 \\
		\lesssim & ((c-a)^3 + n^{-3}).
	\end{align*}
	Similarly, we can show that
	\begin{align*}
		E\left[  \left( \frac{\sqrt{2}}{n \|\Gamma \|_F} \sum_{i = \lfloor d n \rfloor +1}^{\lfloor bn \rfloor  } \sum_{j = \lfloor cn \rfloor}^{\lfloor dn \rfloor - \lfloor \eta n \rfloor -1}  D_{i+\lfloor \eta n \rfloor +1 }^T  D_{j} \right)^6 \right] \lesssim ((d-b)^3 + n^{-3}).
	\end{align*}
	Then, the asymptotic tightness of $\frac{\sqrt{2}}{n \|\Gamma \|_F} S_{D_n}^{u}(a,b| \eta)$ follows similarly from the proof of Equation \eqref{eq:tight}. So, we have $H_{D_n}(r;a,b|\eta) \rightsquigarrow G(r;a,b|\eta) \text{ in } l^{\infty}([0,1]^3)$ and Theorem \ref{thm:main} can be proved similarly as Theorem \ref{cor:T0}. 
	
	\subsubsection{Convergence of Marginals} \label{sec:marginal}
	It suffices to show that for any fixed intervals $ (a_{u,k}, b_{u,k}) \in (0,1)^2$ and constants $\alpha_{u,k} \in \mathbb{R}$, where $k=1,2, \change{...}, K$, $u=1,2,3,4$, it holds that
	\begin{align*}
		\frac{\sqrt{2}}{n \|\Gamma \|_F}  \sum_{u =1}^4 \sum_{k=1}^K  \alpha_{u,k} S_{D_n}^{u}(a_{u,k},b_{u,k}| \eta)   \overset{\mathcal{D}}{\longrightarrow}  \sum_{u =1}^4 \sum_{k=1}^K \alpha_{u,k} V_{u}(a_{u,k}, b_{u,k} | \eta). 
	\end{align*}
	Some algebra show that 
	\begin{align*}
		\frac{\sqrt{2}}{n \|\Gamma \|_F} \sum_{u =1}^4 \sum_{k=1}^K  \alpha_{u,k} S_{D_n}^{u}(a_{u,k},b_{u,k}| \eta) =  \sum_{i= \lfloor a_{\min} n\rfloor}^{\lfloor b_{\max} n\rfloor - \lfloor \eta n\rfloor -1} \widetilde{\xi}_{i} ,
	\end{align*}
	where $a_{\min} = \min_{u,k} a_{u,k}$, $ b_{\max} = \max_{u,k} b_{u,k} $,
	\begin{align*}
		\widetilde{\xi}_{i} & = \sum_{u =1}^4 \sum_{k=1}^K \1_{\{ \lfloor a_{u,k} n\rfloor \leq i \leq  \lfloor b_{u,k} n\rfloor - \lfloor \eta n\rfloor -1 \}}\alpha_{u,k} \xi_{a_{u,k},i}^{u},
	\end{align*}
	and
	\begin{align*}
		\xi_{a_{u,k},i}^u & = \frac{\sqrt{2}}{n \| \Gamma \|_F }  \sum_{j = \lfloor a_{u,k} n \rfloor}^{ i }w_{i,j}^u  D_{i + \lfloor \eta n \rfloor +1}^T D_{j}. 
	\end{align*}
	Similarly, $ \sum_{i= \lfloor a_{\min} n\rfloor}^{j} \widetilde{\xi}_{i} $ is a martingale with respect to 
	$\mathcal{F}_{j-1} = \sigma (X_{j + \lfloor \eta n \rfloor},X_{j + \lfloor \eta n \rfloor - 1}, \change{...}  )$. Then, the conditional variance is calculated as
	\begin{multline*}
		\sum\limits_{i = \lfloor a_{\min} n \rfloor }^{ \lfloor b_{\max}n \rfloor - \lfloor \eta n \rfloor -1 } E [  \widetilde{\xi}_{i}^2 | \mathcal{F}_{i-1} ] =  \sum_{u_1 =1}^4 \sum_{k_1=1}^K \sum_{u_2 =1}^4 \sum_{k_2=1}^K
		\Bigg\{ \alpha_{u_1, k_1} \alpha_{u_2, k_2} \\ \sum\limits_{i = \lfloor (a_{u_1,k_1} \vee a_{u_2, k_2})n \rfloor }^{ \lfloor (b_{u_1,k_1} \wedge b_{u_2,k_2}) n \rfloor - \lfloor \eta n \rfloor -1 }  E [ \xi_{a_{u_1, k_1},i}^{u_1} \xi_{a_{u_2, k_2},i}^{u_2}    | \mathcal{F}_{i-1} ] \Bigg\} . 
	\end{multline*}
	It can be shown that  under Assumption \ref{ass:main}, for $a' \leq  a \leq  b - \eta \leq 1 - \eta$
	\begin{align} \label{eq:finite}
		\sum\limits_{i = \lfloor an \rfloor }^{ \lfloor b n \rfloor - \lfloor \eta n \rfloor -1 }  E [ \xi_{a',i}^{u} \xi_{a,i}^{v}    | \mathcal{F}_{i-1} ] \overset{L^2}{\longrightarrow} C_{u,v}(a,b),
	\end{align}
	where, $C_{u,v}(a,b)=0$ if $a > b - \eta$; otherwise, it is given as 
	\begin{align*}
		C_{u,v}(a,b) = \lim\limits_{n \rightarrow \infty} \frac{2}{n^2 }  \sum_{i = \lfloor an \rfloor}^{ \lfloor b n \rfloor - \lfloor \eta n \rfloor -1} \sum_{j = \lfloor an \rfloor}^{ i }   w_{i,j}^u w_{i,j}^v.
	\end{align*}
	The proof is postponed to Section \ref{lem:finite}. Thus, we have
	\begin{align*}
		\sum\limits_{i = \lfloor a_{\min} n \rfloor }^{ \lfloor b_{\max}n \rfloor - \lfloor \eta n \rfloor -1 } E [  \widetilde{\xi}_{i}^2 | \mathcal{F}_{i-1} ]  \overset{p}{\longrightarrow}  \sum_{u_1 =1}^4 \sum_{k_1=1}^K \sum_{u_2 =1}^4 \sum_{k_2=1}^K
		\alpha_{u_1, k_1} \alpha_{u_2, k_2} C_{u,v}(a_{u_1,k_1} \vee a_{u_2, k_2}, b_{u_1,k_1} \wedge b_{u_2,k_2}).
	\end{align*}
	Next, we check the conditional Lindeberg condition. To this end, it suffices to show that for any fixed interval $(a,b)$ and $u \in \{ 1,2,3,4 \}$
	\begin{align*}
		\sum\limits_{i = \lfloor a n \rfloor }^{ \lfloor b n \rfloor - \lfloor \eta n \rfloor -1 } E[(\xi_{a,i}^u)^4] = o(1).
	\end{align*} 
	Due to the independence of $\{D_{i}\}_{i=1}^n$, we have 
	\begin{align*}
		E[D_{i + \lfloor \eta n \rfloor +1}^T D_{j_1} D_{i + \lfloor \eta n \rfloor +1}^T D_{j_2} D_{i + \lfloor \eta n \rfloor +1}^T D_{j_3} D_{i + \lfloor \eta n \rfloor +1}^T D_{j_4}] \neq 0,
	\end{align*}
	only if $j_1, j_2, j_3, j_4$ are pair-wise equal. In addition, from Lemma \ref{lem:main}
	\begin{align*}
		E[D_{i + \lfloor \eta n \rfloor +1}^T D_{j_1} D_{i + \lfloor \eta n \rfloor +1}^T D_{j_2} D_{i + \lfloor \eta n \rfloor +1}^T D_{j_3} D_{i + \lfloor \eta n \rfloor +1}^T D_{j_4}] \lesssim \| \Gamma \|_F^4.
	\end{align*}
	Thus, we have
	$
	\sum_{i = \lfloor a n \rfloor }^{ \lfloor b n \rfloor - \lfloor \eta n \rfloor -1 } E[(\xi_{a,i}^u)^4]  \lesssim O(1/n).
	$
	
	\subsubsection{Proof of Equation \eqref{eq:BN}} \label{sec:BN}
	First, write $ S_{X_n}^{u} (a,b| \eta) $ as
	\begin{align*}
		\frac{\sqrt{2}}{n \|\Gamma \|_F} S_{X_n}^{u}(a,b| \eta) = \frac{\sqrt{2}}{n \|\Gamma \|_F} S_{D_n}^{u}(a,b| \eta) + R_{u},
	\end{align*}
	where
	\begin{multline*}
		R_u  =  \frac{\sqrt{2}}{n \|\Gamma \|_F} \Bigg\{   - \underbrace{ \sum_{i = \lfloor an \rfloor }^{ \lfloor bn \rfloor - \lfloor \eta n \rfloor -1 } \sum_{j = \lfloor an \rfloor}^{i} w_{i,j}^u  D_{i+\lfloor \eta n \rfloor +1}^T  \varepsilon_j}_{R_{u,1}} \\ 
		-\underbrace{ \sum_{j = \lfloor an \rfloor }^{ \lfloor bn \rfloor - \lfloor \eta n \rfloor -1 }  \sum_{i = j }^{\lfloor bn \rfloor - \lfloor \eta n \rfloor - 1} w_{i,j}^u \varepsilon_{i+\lfloor \eta n \rfloor +1}^T  D_j}_{R_{u,2}}  + \underbrace{ \sum_{i = \lfloor an \rfloor }^{ \lfloor bn \rfloor - \lfloor \eta n \rfloor -1 } \sum_{j = \lfloor an \rfloor}^{i} w_{i,j}^u  \varepsilon_{i+\lfloor \eta n \rfloor +1} ^T\varepsilon_j}_{R_{u,3}} \Bigg\}.
	\end{multline*}
	We then show that each of the terms $R_{u,1}, R_{u,2}, R_{u,3}$ converges weakly to 0 in $l^{\infty}([0,1]^2)$. The proof techniques for $R_{u,1}$ and $R_{u,2}$ are very similar, here we only give details to show that $ R_{u,2} \rightsquigarrow 0$. Some algebra show that
	\begin{multline*}
		\sum_{i = j }^{\lfloor bn \rfloor - \lfloor \eta n \rfloor - 1} w_{i,j}^u \varepsilon_{i+\lfloor \eta n \rfloor +1} = \\
		\left\lbrace 
		\begin{array}{ll}
			\widetilde{D}_{ \lfloor bn \rfloor} - \widetilde{D}_{ j+ \lfloor \eta n \rfloor }, & u=1; \\
			\frac{j}{n} \left\{ \widetilde{D}_{ \lfloor bn \rfloor} - \widetilde{D}_{ j+ \lfloor \eta n \rfloor } \right\}  , & u=2; \\
			\frac{\lfloor bn \rfloor+1}{n} \widetilde{D}_{\lfloor bn \rfloor } - \frac{1}{n} \left\{  \sum_{i = j}^{ \lfloor bn \rfloor - \lfloor \eta n \rfloor -1} \widetilde{D}_{i+ \lfloor \eta n \rfloor +1} \right\} -\frac{ j+\lfloor \eta n \rfloor +1 }{n}\widetilde{D}_{ j+\lfloor \eta n \rfloor}, & u=3; \\
			\frac{ j }{n} \left\{ \frac{\lfloor bn \rfloor+1}{n} \widetilde{D}_{\lfloor bn \rfloor } - \frac{1}{n} \left\{  \sum_{i = j}^{ \lfloor bn \rfloor - \lfloor \eta n \rfloor -1} \widetilde{D}_{i+ \lfloor \eta n \rfloor +1} \right\} -\frac{ j+\lfloor \eta n \rfloor +1 }{n}\widetilde{D}_{ j+\lfloor \eta n \rfloor}  \right\}, & u=4.
		\end{array} \right. 
	\end{multline*}
	The above decomposition for $R_{u,2}$ is complex, fortunately it can be simplified with the following lemma.
	\begin{lemma} \label{lem:run} Under Assumption \ref{ass:main}, it holds in $l^{\infty}([0,1]^2)$ that
		\begin{align*}
			\sup_{a,b}	\left|\frac{\sqrt{2}}{n \|\Gamma \|_F}  \sum_{i = \lfloor an \rfloor }^{ \lfloor bn \rfloor - \lfloor \eta n \rfloor -1 }  v_i  D_i^T \widetilde{D}_{ i+\lfloor \eta n \rfloor}\right| = o_p(1),
		\end{align*}
		where $\{v_i\}$ is a sequence of constants such that $\sup_i |v_i| \leq 1$.
	\end{lemma} 
	Thus, we can throw away the following terms
	\begin{align*}
		\sum_{j = \lfloor an \rfloor }^{ \lfloor bn \rfloor - \lfloor \eta n \rfloor -1 } D_j^T \widetilde{D}_{ j+ \lfloor \eta n \rfloor }&, \sum_{j = \lfloor an \rfloor }^{ \lfloor bn \rfloor - \lfloor \eta n \rfloor -1 } \frac{j}{n} D_j^T  \widetilde{D}_{ j+ \lfloor \eta n \rfloor }, \\
		\sum_{j = \lfloor an \rfloor }^{ \lfloor bn \rfloor - \lfloor \eta n \rfloor -1 }\frac{ j+\lfloor \eta n \rfloor +1 }{n} D_j^T\widetilde{D}_{ j+\lfloor \eta n \rfloor}&, \sum_{j = \lfloor an \rfloor }^{ \lfloor bn \rfloor - \lfloor \eta n \rfloor -1 }\frac{j}{n}\frac{ j+\lfloor \eta n \rfloor +1 }{n} D_j^T\widetilde{D}_{ j+\lfloor \eta n \rfloor}.
	\end{align*}
	Next, we focus on the following term
	\begin{align*}
		\frac{\sqrt{2}}{n^2 \|\Gamma \|_F}  \sum_{i = \lfloor an \rfloor }^{ \lfloor bn \rfloor - \lfloor \eta n \rfloor -1 } \sum_{j = i}^{ \lfloor bn \rfloor - \lfloor \eta n \rfloor -1} v_i  D_{i}^T  \widetilde{D}_{j+ \lfloor \eta n \rfloor +1}  
		= \frac{\sqrt{2}}{n^2 \|\Gamma \|_F}  \sum_{i = \lfloor an \rfloor }^{ \lfloor bn \rfloor - \lfloor \eta n \rfloor -1 } \sum_{j = \lfloor an \rfloor}^{ i } v_j \widetilde{D}_{i+ \lfloor \eta n \rfloor +1}^T D_{j}. 
	\end{align*}
	Then, applying the triangle inequality
	\begin{align}
		&  \left|  \frac{\sqrt{2}}{n^2 \|\Gamma \|_F}  \sum_{i = \lfloor an \rfloor }^{ \lfloor bn \rfloor - \lfloor \eta n \rfloor -1 } \sum_{j = \lfloor an \rfloor}^{ i } v_j \widetilde{D}_{i+ \lfloor \eta n \rfloor +1}^T D_{j}\right| \\ 
		\leq &  \left|  \frac{\sqrt{2}}{n^2 \|\Gamma \|_F}  \sum_{i = 1 }^{ \lfloor bn \rfloor - \lfloor \eta n \rfloor -1 } \sum_{j = 1 }^{ i } v_j \widetilde{D}_{i+ \lfloor \eta n \rfloor +1}^T D_{j}\right| \label{eq:re1} \\ 
		+ &  \left|  \frac{\sqrt{2}}{n^2 \|\Gamma \|_F}  \sum_{i = 1 }^{ \lfloor an \rfloor  -1 } \sum_{j = 1 }^{ i } v_j \widetilde{D}_{i+ \lfloor \eta n \rfloor +1}^T D_{j}\right| \label{eq:re2} \\
		+ & \left|  \frac{\sqrt{2}}{n^2 \|\Gamma \|_F}  \sum_{i = \lfloor an \rfloor }^{ \lfloor bn \rfloor - \lfloor \eta n \rfloor -1 } \sum_{j = 1}^{ \lfloor an \rfloor -1}v_j \widetilde{D}_{i+ \lfloor \eta n \rfloor +1}^T D_{j}\right| \label{eq:re3}. 
	\end{align}
	Due to the following lemma, terms \eqref{eq:re1} and \eqref{eq:re2} are both of order $o_p(1)$ in metric space $l^{\infty}([0,1]^2)$.
	\begin{lemma} \label{eq:final2} Under Assumption \ref{ass:main}, it holds in $l^{\infty}([0,1]^2)$ that
		\begin{align*} 
			\frac{1}{n}\left\| \sup_{a} \frac{\sqrt{2}}{n \|\Gamma \|_F}  \Bigg|  \sum_{i = 1 }^{ \lfloor an \rfloor } \sum_{j = 1 }^{ i } v_j \widetilde{D}_{i+ \lfloor \eta n \rfloor +1}^T D_{j}\Bigg| \right\|_2 = o(1),
		\end{align*}
		where $\{v_j\}$ is a sequence of constants such that $\sup_j |v_j| \leq 1$.
	\end{lemma}
	Next, denote the $L^p$-norm of a random variable $X$ as $\| X \|_p := \left( E [|X|^p] \right)^{1/p}$. For any two parameter process $W(a,b)$, if $\| W(a,b) \|_6 \lesssim 1/\sqrt{n} $, then its marginals $\left(W(a_1, b_1), \change{...},W(a_k, b_k)  \right)$ converges to $0$ and from the proof of Equation \eqref{eq:tight}, it is asymptotically tight. Thus, we have $W(a,b)  \rightsquigarrow 0$. With this logic and the following lemma, term \eqref{eq:re3} is also asymptotically negligible.
	\begin{lemma} \label{lem:neg2} Under Assumption \ref{ass:main}, it holds in $l^{\infty}([0,1]^2)$ that
		\begin{align*}
			\left\| \frac{\sqrt{2}}{n^2 \|\Gamma \|_F}  \sum_{i = \lfloor an \rfloor }^{ \lfloor bn \rfloor - \lfloor \eta n \rfloor -1 } \sum_{j = 1}^{ \lfloor an \rfloor -1}v_j \widetilde{D}_{i+ \lfloor \eta n \rfloor +1}^T D_{j} \right\|_6 \lesssim \frac{1}{\sqrt{n}},
		\end{align*}
		where $\{v_i\}$ is a sequence of constants such that $\sup_i |v_i| \leq 1$.
	\end{lemma}
	Finally, using the same logic, it can be seen from the lemma below that all the other terms in $R_{u,2}$ converge weakly to $0$. 
	\begin{lemma} \label{lem:neg} Under Assumption \ref{ass:main}, it holds in $l^{\infty}([0,1]^2)$ that
		\begin{align*}
			\left\| \frac{\sqrt{2}}{n \|\Gamma \|_F}  \sum_{i = \lfloor an \rfloor }^{ \lfloor bn \rfloor - \lfloor \eta n \rfloor -1 } v_i  D_{i}^T  \widetilde{D}_{ \lfloor b n \rfloor } \right\|_6 \lesssim \frac{1}{\sqrt{n}},
		\end{align*}
		where $\{v_i\}$ is a sequence of constants such that $\sup_i |v_i| \leq 1$.
	\end{lemma}
	This concludes the proof that $R_{u,2} \rightsquigarrow 0 $. For $R_{u,3}$, notice that
	\begin{align*}
		\left|  R_{u,3} \right| \leq \left| \sum_{i = \lfloor an \rfloor }^{ \lfloor bn \rfloor - \lfloor \eta n \rfloor -1 } \sum_{j = \lfloor an \rfloor}^{i} w_{i,j}^u  \widetilde{D}_{i+\lfloor \eta n \rfloor +1}^T \varepsilon_j \right| + \left| \sum_{i = \lfloor an \rfloor }^{ \lfloor bn \rfloor - \lfloor \eta n \rfloor -1 } \sum_{j = \lfloor an \rfloor}^{i} w_{i,j}^u  \widetilde{D}_{i+\lfloor \eta n \rfloor }^T \varepsilon_j \right|.
	\end{align*}
	Comparing the above two terms with $R_{u,2}$, we have $\widetilde{D}_{i+\lfloor \eta n \rfloor}$ instead of $D_{i+\lfloor \eta n \rfloor}$. Since both Lemma \ref{lem:cov} and \ref{lem:exp} hold for any combination of $\widetilde{D}_i$ and $D_j$, we can show similarly as in the proof for $R_{u,2}$ that these two terms are of order $o_p(1)$.

	\subsubsection{Proof of  Equation \eqref{eq:finite}} \label{lem:finite}
	\begin{proof}
		Notice that 
		\begin{align*}
			\sum\limits_{i = \lfloor an \rfloor }^{ \lfloor b n \rfloor - \lfloor \eta n \rfloor -1 }  E [ \xi_{a',i}^{u} \xi_{a,i}^{v}    | \mathcal{F}_{i-1} ] = 	\sum\limits_{i = \lfloor an \rfloor }^{ \lfloor b n \rfloor - \lfloor \eta n \rfloor -1 }  E [ \xi_{a,i}^{u} \xi_{a,i}^{v}    | \mathcal{F}_{i-1} ] + \widetilde{R},
		\end{align*}
		where 
		\begin{align*}
			\widetilde{R} & =
			\frac{2}{n^2 \| \Gamma \|_F^2 }\sum\limits_{i = \lfloor an \rfloor }^{ \lfloor b n \rfloor - \lfloor \eta n \rfloor -1}  \sum_{j' = \lfloor a' n \rfloor}^{ \lfloor an \rfloor -1 }    \sum_{j = \lfloor a n \rfloor}^{ i }w_{i,j'}^u w_{i,j}^v E[ D_{i + \lfloor \eta n \rfloor +1}^T D_{j'} D_{i + \lfloor \eta n \rfloor +1}^T D_{j}|\mathcal{F}_{i-1}] \\
			& =
			\frac{2}{n^2 \| \Gamma \|_F^2 }\sum\limits_{i = \lfloor an \rfloor }^{ \lfloor b n \rfloor - \lfloor \eta n \rfloor -1}  \sum_{j' = \lfloor a' n \rfloor}^{ \lfloor an \rfloor -1 }    \sum_{j = \lfloor a n \rfloor}^{ i }w_{i,j'}^u w_{i,j}^v tr(D_{j'}^T \Gamma D_{j}).
		\end{align*}
		We then show that $\widetilde{R}$ is negligible. 
		\begin{align*}
			E[\widetilde{R}^2] & \leq  \frac{4}{n^4 \| \Gamma \|_F^4 }\sum\limits_{i_1, i_2 = \lfloor an \rfloor }^{ \lfloor b n \rfloor - \lfloor \eta n \rfloor -1}  \sum_{j'_1, j'_2 = \lfloor a' n \rfloor}^{ \lfloor an \rfloor -1 }    \sum_{j_1 = \lfloor a n \rfloor}^{ i_1 } \sum_{j_2 = \lfloor a n \rfloor}^{ i_2 } \left|E[ D_{j'_1}^T \Gamma D_{j_1} D_{j'_2}^T \Gamma D_{j_2} ] \right|  \\
			& \lesssim  \frac{4}{n^4 \| \Gamma \|_F^4 } \sum\limits_{i_1 \leq i_2 }^{ }  \sum_{j' = \lfloor a' n \rfloor}^{ \lfloor an \rfloor -1 }    \sum_{j = \lfloor a n \rfloor}^{ i_1 }  \left|E[ D_{j'}^T \Gamma D_{j} D_{j'}^T \Gamma D_{j} ] \right| \\
			& \lesssim  \frac{4}{n^4 \| \Gamma \|_F^4 } n^4 tr(\Gamma^4) \\
			& = o(1).
		\end{align*}
		Next, we focus on the first term  
		\begin{align*}
			& \sum\limits_{i = \lfloor an \rfloor }^{ \lfloor b n \rfloor - \lfloor \eta n \rfloor -1 }  E [ \xi_{a,i}^{u} \xi_{a,i}^{v}    | \mathcal{F}_{i-1} ] \\ = &\frac{2}{n^2 \|\Gamma\|_F^2 }  \sum_{i = \lfloor an \rfloor}^{ \lfloor b n \rfloor - \lfloor \eta n \rfloor -1} \sum_{j = \lfloor an \rfloor}^{ i} \sum_{j' = \lfloor an \rfloor}^{ i} w_{i,j}^u w_{i,j'}^v  E \left[  D_{i + \lfloor \eta n \rfloor +1}^T D_{j}  D_{j'}^T  D_{i + \lfloor \eta n \rfloor +1}  | \mathcal{F}_{i-1} \right] \\
			= & \frac{2}{n^2 \|\Gamma\|_F^2 }  \sum_{i = \lfloor an \rfloor}^{ \lfloor b n \rfloor - \lfloor \eta n \rfloor -1} \sum_{j = \lfloor an \rfloor}^{ i} \sum_{j' = \lfloor an \rfloor}^{ i} w_{i,j}^u w_{i,j'}^v D_{j}^T \Gamma  D_{j'},
		\end{align*}
		whose mean can be calculated as
		\begin{align*}
			\sum\limits_{i = \lfloor an \rfloor }^{ \lfloor b n \rfloor - \lfloor \eta n \rfloor -1 }  E [ \xi_{a,i}^{u} \xi_{a,i}^{v}   ] 
			= &\frac{2}{n^2 \|\Gamma\|_F^2 }  \sum_{i = \lfloor a n \rfloor}^{ \lfloor b n \rfloor - \lfloor \eta n \rfloor -1} \sum_{j = \lfloor an \rfloor}^{ i }   w_{i,j}^u w_{i,j}^v tr\left( \Gamma^2 \right)\\
			=&  \frac{2}{n^2 }  \sum_{i = \lfloor an \rfloor}^{ \lfloor b n \rfloor - \lfloor \eta n \rfloor -1} \sum_{j = \lfloor an \rfloor}^{ i }   w_{i,j}^u w_{i,j}^v \\
			\rightarrow & C_{u,v}(a,b) \text{ as }n \rightarrow \infty.
		\end{align*}
		Next, we show that the variance of the first term is asymptotically 0. Observe that 
		\begin{align*}
			\cov( D_{j_1}^T \Gamma  D_{j'_1},  D_{j_2}^T \Gamma  D_{j'_2}  ) =  \left\lbrace 
			\begin{array}{ll}
				E[ D_{1}^T \Gamma  D_{1}D_{1}^T \Gamma  D_{1} ] - tr(\Gamma^2)^2, & j_1=j_1'=j_2=j_2'; \\
				tr(\Gamma^4), & j_1 = j_2, j_1'=j_2', j_1 \neq j_1'; \\
				tr(\Gamma^4), & j_1 = j_2', j_1'=j_2, j_1 \neq j_1';\\
				0,& \text{ otherwise}.
			\end{array} \right. 
		\end{align*}
		Thus, if $j_1 = j_2, j_1'=j_2', j_1 \neq j_1'$,
		\begin{align*}
			& \var \left( \sum\limits_{i = \lfloor an \rfloor }^{ \lfloor b n \rfloor - \lfloor \eta n \rfloor -1 }  E [ \xi_{a,i}^{u} \xi_{a,i}^{v}    | \mathcal{F}_{i-1} ] \right) \\ 
			\lesssim  & \frac{4}{n^4 \| \Gamma \|_F^4 } \sum_{i_1, i_2 = \lfloor an \rfloor}^{ \lfloor b n \rfloor - \lfloor \eta n \rfloor -1} \sum_{j_1, j_1' = \lfloor an \rfloor}^{i_1} \sum_{j_2, j_2' = \lfloor an \rfloor}^{i_2}  \left| \cov( D_{j_1}^T \Gamma  D_{j'_1},  D_{j_2}^T \Gamma  D_{j'_2}  ) \right|  \\
			\lesssim & \frac{1}{n^4 \| \Gamma \|_F^4 } O(n^4) tr(\Gamma^4) \\ = & o(1),
		\end{align*}
		where the above inequality holds true since there are at most $O(n^4)$ non-zero terms. The case that $ j_1 = j_2', j_1'=j_2, j_1 \neq j_1' $ can be shown similarly. When $ j_1=j_1'=j_2=j_2' $, it has been shown in the proof of \eqref{eq:V1n} and \eqref{eq:V2n} that $ E[ D_{1}^T \Gamma  D_{1}D_{1}^T \Gamma  D_{1} ] \lesssim   \| \Gamma \|_F^4 $, thus
		\begin{align*}
			& \var \left( \sum\limits_{i = \lfloor an \rfloor }^{ \lfloor b n \rfloor - \lfloor \eta n \rfloor -1 }  E [ \xi_{a,i}^{u} \xi_{a,i}^{v}    | \mathcal{F}_{i-1} ] \right) \\ 
			\lesssim  & \frac{4}{n^4 \| \Gamma \|_F^4 } \sum_{i_1, i_2 = \lfloor an \rfloor}^{ \lfloor b n \rfloor - \lfloor \eta n \rfloor -1} \sum_{j_1 = \lfloor an \rfloor}^{\min \{i_1, i_2\} }  \left| E( D_{j_1}^T \Gamma  D_{j_1}  D_{j_1}^T \Gamma  D_{j_1}  ) \right|   \\
			\lesssim & \frac{4}{n^4 \| \Gamma \|_F^4 } n^3 \| \Gamma \|_F^4 \\
			= & o(1).
		\end{align*}
	\end{proof}
	
	\subsection{Proof of Theorem \ref{thm:alt}}
	We first state a lemma.
	\begin{lemma} \label{lem:alt}
		Under Assumption \ref{ass:main}, for any deterministic sequence of vectors $\delta_n \in \mathbb{R}^p$, 
		\begin{align*}
			& \sup_{1 \leq l<k \leq n }    \left| \frac{1}{\|\Gamma \|_F} \sum_{i=l}^{k}  X_{i+\lfloor \eta n \rfloor +1}^T  \delta_n  \right| = o_p\left( \frac{\sqrt{n}\| \delta_n \|_2 }{ \|\Gamma \|_F^{1/2} } \right)
			\\
			& \sup_{1 \leq l<k \leq n }    \left|  \frac{1}{\|\Gamma \|_F} \sum_{i=l}^{k} \frac{ i+\lfloor \eta n \rfloor +1 }{n}  X_{i+\lfloor \eta n \rfloor +1}^T  \delta_n  \right| = o_p\left( \frac{\sqrt{n}\| \delta_n \|_2 }{ \|\Gamma \|_F^{1/2} } \right)
			\\
			& \sup_{1 \leq l<k \leq n}    \left|  \frac{1}{\|\Gamma \|_F} \sum_{j=l}^{k}   \delta_n^T X_{j}  \right| =  
			o_p\left( \frac{\sqrt{n}\| \delta_n \|_2 }{ \|\Gamma \|_F^{1/2} } \right)
		\end{align*}
		and 
		\begin{align*}
			& \sup_{1 \leq l<k \leq n}    \left|  \frac{1}{\|\Gamma \|_F} \sum_{j=l}^{k} \frac{j}{n}   \delta_n^T X_{j}  \right| =  o_p\left( \frac{\sqrt{n}\| \delta_n \|_2 }{ \|\Gamma \|_F^{1/2} } \right)
		\end{align*}
	\end{lemma}
	
	Given the bounds above we have
	\begin{multline*}
		\frac{\sqrt{2}}{n \|\Gamma \|_F} S_{Y_n}^{u}(a,b| \eta)  = \frac{\sqrt{2}}{n \|\Gamma \|_F} S_{X_n}^{u}(a,b| \eta)
		+ o_p\left( \frac{\sqrt{n}\| \delta_n \|_2 }{ \|\Gamma \|_F^{1/2} } \right)
		\\
		+\left\lbrace  \begin{array}{ll}
			\left( \frac{\sqrt{2}}{n^2} \sum_{i = \lfloor (\phi \vee a) n \rfloor}^{ \lfloor b n \rfloor - \lfloor \eta n \rfloor -1} \sum_{j = \lfloor (\phi \vee a) n \rfloor}^{ i }   w_{i,j}^u  \right)   \frac{n\|\delta_n \|_2^2}{\|\Gamma \|_F}, & \text{ if }  \phi < b - \eta; \\
			0, & \text{ otherwise },
		\end{array} \right. 
	\end{multline*}
	Recall that in the proof of Theorem \ref{thm:main}, we decompose $ \widetilde{H}_{Y_n}(\lfloor \phi n \rfloor ;1, n|\lfloor \eta n \rfloor) $ as a continuous functional of $ \frac{\sqrt{2}}{n \|\Gamma \|_F} S_{Y_n}^{u}(a,b| \eta) $, then by replacing each $ \frac{\sqrt{2}}{n \|\Gamma \|_F} S_{Y_n}^{u}(a,b| \eta)$ with the above decomposition, Theorem \ref{thm:alt} follows straight-forwardly under the case that $ n \| \delta_n \|_2^2 / \| \Gamma \|_F \rightarrow c^2 \in (0, \infty)$ or $ n \| \delta_n \|_2^2 / \| \Gamma \|_F \rightarrow 0$. For the case that $ n \| \delta_n \|_2^2 / \| \Gamma \|_F \rightarrow \infty$, it follows from  Lemma \ref{lem:alt} that 
	\begin{align*}
		\widetilde{H}_{Y_n}(\lfloor \phi n \rfloor ;1, n|\lfloor \eta n \rfloor) =&  \widetilde{H}_{X_n}(\lfloor \phi n \rfloor ;1, n| \lfloor \eta n \rfloor) + o_p\left( \frac{\sqrt{n}\| \delta_n \|_2 }{ \|\Gamma \|_F^{1/2} } \right)  
		\\ 
		&+ \sum_{\stackrel{1 \leq j_1 , j_3 \leq  \lfloor \phi n \rfloor}{|j_1-j_3|>\lfloor \eta n \rfloor} } \sum_{\stackrel{\lfloor \phi n \rfloor + \lfloor \eta n \rfloor + 1 \leq j_2, j_4 \leq n}{|j_2-j_4|>\lfloor \eta n \rfloor} } \frac{ \| \delta_n \|_2^2}{ n^3\| \Gamma \|_F } ,
	\end{align*}
	which implies that $\widetilde{H}_{Y_n}(\lfloor \phi n \rfloor ;1, n|\tau)$ goes to infinity in probability. Then, the result follows similarly as in the proof of Theorem \ref{cor:T_alt}. \hfill $\Box$

	\subsection{Proof of Auxiliary Lemmas}
	\subsubsection{Proof of Lemma \ref{lem:cov}}
	\begin{proof}
		Firstly, for the case $j=0$, let $i_{min} = \min \{ i_1, i_2, \change{...}, i_{h} \}$, $ \widetilde{c}_i = \sum_{u = i+1}^{\infty} c_{u} $, $\widetilde{c}_{i,(l,\cdot)}$ be the $l$-th row of $\widetilde{c}_i$ and $\widetilde{c}_{i,(l,k)}$ be the $(l,k)$-th entry of $ \widetilde{c}_i $, the absolute value of cumulant can be bounded as 
		\begin{align*}
			&\sum_{l_1, \change{...},l_h=1}^p  |cum( \widetilde{D}_{i_1, l_1},  \change{...},  \widetilde{D}_{i_{h}, l_{h}})| \\
			\leq & \sum_{l_1, \change{...},l_h=1}^p  \sum_{j =0}^{\infty} \sum_{k_1, \change{...}, k_h = 1}^p \left|  \prod_{g=1}^h \widetilde{c}_{i_g-i_{min} + j, (l_g, k_g)}\right| |cum(\epsilon_{i_{min} - j, k_1}, \change{...}, \epsilon_{i_{min} - j, k_h})|  \\
			= &   \sum_{k_1, \change{...}, k_h = 1}^p \left(\sum_{l_1, \change{...},l_h=1}^p  \sum_{j =0}^{\infty}  \left| \prod_{g=1}^h \widetilde{c}_{i_g-i_{min} + j, (l_g, k_g)} \right| \right) |cum(\epsilon_{0, k_1}, \change{...}, \epsilon_{0, k_h}) |  \\
			\leq &  \left( \sum_{j =0}^{\infty} \prod_{g=1}^h \|\widetilde{c}_{i_g-i_{min} + j}\|_1 \right)\sum_{k_1, \change{...}, k_h = 1}^p |cum(\epsilon_{0, k_1}, \change{...}, \epsilon_{0, k_h})| \\
			\lesssim & \left( \sum_{j =0}^{\infty} \left(  \sum_{u =j+1}^{\infty} \| c_{u} \|_1\right)^h \right) \| \Gamma \|_{F}^h,
		\end{align*}
		which proves the case $ j=0 $. For other cases, the results can be shown similarly. To bound the square cumulant, notice that under assumption \ref{ass:main} (\ref{A41A1}), it can be easily shown that there exists a constant $C$ such that   
		\begin{align*}
			\max_{1 \leq k_1, \change{...}, k_h \leq p } |cum(\epsilon_{0, k_1}, \change{...}, \epsilon_{0, k_h})| \leq C.
		\end{align*}
		Next, for any $(l_1, \change{...}, l_h)$, we can bound $ |cum( \widetilde{D}_{i_1, l_1},  \change{...},  \widetilde{D}_{i_{h}, l_{h}})| $ as follows
		\begin{align*}
			&  |cum( \widetilde{D}_{i_1, l_1},  \change{...},  \widetilde{D}_{i_{h}, l_{h}})| \\
			\leq &   \sum_{j =0}^{\infty} \sum_{k_1, \change{...}, k_h = 1}^p \prod_{g=1}^h |\widetilde{c}_{i_g-i_{min} + j, (l_g, k_g)} cum(\epsilon_{i_{min} - j, k_1}, \change{...}, \epsilon_{i_{min} - j, k_h})|  \\
			\leq & C \sum_{j =0}^{\infty}  \sum_{k_1, \change{...}, k_h = 1}^p \prod_{g=1}^h |\widetilde{c}_{i_g-i_{min} + j, (l_g, k_g)}|   \\
			\leq & C \sum_{j =0}^{\infty} \left( \sum_{u =j+1}^{\infty} \| c_{u} \|_{\infty} \right)^h < \infty \text{ by assumption \ref{ass:main} (\ref{A41A5})}.
		\end{align*}
		The proof is similar when $j \neq 0$, thus there exists a constant $C'$ such that for any $(l_1, \change{...}, l_h)$, we have $|cum( D_{i_1, l_1}, \change{...}, D_{i_j, l_j}, \widetilde{D}_{i_{j+1}, l_{j+1}}, \change{...},  \widetilde{D}_{i_{h}, l_{h}})|\leq C'$.
		As a consequence, it holds that
		\begin{multline*}
			\sum_{l_1,l_2,\change{...},l_h=1}^p cum^2( D_{i_1, l_1}, \change{...}, D_{i_j, l_j}, \widetilde{D}_{i_{j+1}, l_{j+1}}, \change{...},  \widetilde{D}_{i_{h}, l_{h}})/C'^2 \leq \\
			\sum_{l_1,l_2,\change{...},l_h=1}^p |cum( D_{i_1, l_1}, \change{...}, D_{i_j, l_j}, \widetilde{D}_{i_{j+1}, l_{j+1}}, \change{...},  \widetilde{D}_{i_{h}, l_{h}})|/C' \lesssim \| \Gamma \|_F^h.
		\end{multline*}
	\end{proof}
	
	\subsubsection{Proof of Lemma \ref{lem:exp}}
	\begin{proof} We show that $\{ X_i \}$, $\{ D_i \}$, $\{ \widetilde{D}_i \}$,$\{ \varepsilon_i \}$ are all UGMC(8), then the result follows from Remark 9.6 \cite{wang2019hypothesis}. Firstly, we have
		\begin{align*}
			& E \left[ ( X_{i, l} - X_{i,l}' )^8 \right] \\ 
			= & E \left[ \left(  \sum_{j =i}^{\infty} c_{j, (l,\cdot)}^T (\epsilon_{i-j} - \epsilon_{i-j}' ) \right) ^8 \right] \\
			= &  \sum_{j_1, \change{...},  j_{8}=i}^{\infty} E \left[   \left(  c_{j_1, (l,\cdot)}^T (\epsilon_{i-j_1} - \epsilon_{i-j_1}' ) \right)  \change{...}   \left( c_{j_8, (l,\cdot)}^T (\epsilon_{i-j_8} - \epsilon_{i-j_8}' ) \right) \right] \\
			\leq & \sum_{j_1, \change{...},  j_{8}=i}^{\infty} \left( E \left[  \left( c_{j_1, (l,\cdot)}^T (\epsilon_{i-j_1} - \epsilon_{i-j_1}' ) \right)^{8} \right]\right)^{1/8} \change{...} \left(E \left[  \left( c_{j_8, (l,\cdot)}^T (\epsilon_{i-j_8} - \epsilon_{i-j_8}' ) \right)^{8} \right]\right)^{1/8}.
		\end{align*}
		Next, we bound the term  $ E [ ( c_{j_1, (l,\cdot)}^T (\epsilon_{i-j_1} - \epsilon_{i-j_1}' ) )^{8} ] $ as an example.
		\begin{align*}
			& E \left[  \left( c_{j_1, (l,\cdot)}^T (\epsilon_{i-j_1} - \epsilon_{i-j_1}' ) \right)^{8} \right] \\
			= &E \left[\left(  \sum_{k = 1}^p c_{j_1, (l,k)} (\epsilon_{i-j_1, k} - \epsilon_{i-j_1, k}' ) \right)^8 \right] \\
			= &  \sum_{k_1, \change{...}, k_8} \left( \prod_{g=1}^8 c_{j_1, (l,k_g)} \right) E \left[ \prod_{g=1}^8 (\epsilon_{i-j_1, k_g} - \epsilon_{i-j_1, k_g}' ) \right] \\
			\lesssim & \left(  \sum_{k = 1}^p c_{j_1, (l,k)} \right)^8 \leq \| c_{j_1} \|_{\infty}^8.
		\end{align*}
		Thus, we have
		$
		E [ ( X_{i, l} - X_{i,l}' )^8 ] \leq ( \sum_{j =i}^{\infty} \| c_{j} \|_{\infty} )^8.
		$
		It can be shown similarly that $\sup_l E [ |X_{0,l}|^8 ] \leq C^8$ for some constant $C$, which concludes that $\{X_i\}$ is UGMC(8). From Lemma 9.4 [\cite{wang2019hypothesis}], $\{ D_i \}$, $\{ \widetilde{D}_i \}$,$\{ \varepsilon_i \}$ are also UGMC(8).
	\end{proof}
	
	\subsubsection{Proof of Lemma \ref{lem:main}}
	\begin{proof}
		Let $\pi$ be any disjoint partition over the set $\{ l_1,l_2,l_3,l_4, l_5, l_6 \}$ such that for any $B \in \pi$, $|B| \neq 1$. Under Assumption \ref{ass:main},
		\begin{align*}
			& E \left[ (D_{i}^T D_{j})^6 \right] \\
			= &  \sum_{l_1,l_2,l_3,l_4, l_5, l_6} E[ D_{i,l_1} D_{i,l_2} D_{i,l_3} D_{i,l_4} D_{i,l_5} D_{i, l_6} ] E[D_{j,l_1} D_{j,l_2} D_{j,l_3} D_{j,l_4} D_{j,l_5} D_{j, l_6} ]  \\
			= &  \sum_{l_1,l_2,l_3,l_4, l_5, l_6} (E[ D_{i,l_1} D_{i,l_2} D_{i,l_3} D_{i,l_4} D_{i,l_5} D_{i, l_6} ])^2   \\
			= &  \sum_{l_1,l_2,l_3,l_4, l_5, l_6} \left(\sum_{\pi} \prod_{B \in \pi} cum(D_{i,l_k} : l_k \in B ) \right)^2   \\
			\lesssim &  \sum_{l_1,l_2,l_3,l_4, l_5, l_6}  \sum_{\pi} \prod_{B \in \pi} cum^2(D_{i,l_k} : l_k \in B )  \text{ by Cauchy's inequality} \\
			\lesssim & \sum_{\pi} \prod_{B \in \pi} \left\lbrace \sum_{l_k \in B} \sum_{l_k=1}^p cum^2(D_{i,l_k} : l_k \in B ) \right\rbrace \\
			\lesssim &   \sum_{\pi} \prod_{B \in \pi} \| \Gamma \|^{|B|}_F \hspace{2cm} \text{ by Lemma \ref{lem:cov}} \\
			\lesssim & \| \Gamma \|^6_F.
		\end{align*}
		
	\end{proof}
	
	\subsubsection{Proof of Lemma \ref{lem:run}}
	\begin{proof}
		The triangle inequality implies that
		\begin{multline*}
			\sup_{0 < a < b- \eta < 1 -\eta} \left| \frac{\sqrt{2}}{n \|\Gamma \|_F} \sum_{i = \lfloor an \rfloor }^{ \lfloor bn \rfloor - \lfloor \eta n \rfloor -1 } v_i D_i^T \widetilde{D}_{i+\lfloor \eta n \rfloor} \right| \leq \\ \sup_{\eta <b < 1} \left| \frac{\sqrt{2}}{n \|\Gamma \|_F} \sum_{i = 1 }^{ \lfloor bn \rfloor - \lfloor \eta n \rfloor -1 }v_i D^T_{i} \widetilde{D}_{i+\lfloor \eta n \rfloor} \right| + \sup_{0 < a < 1 - \eta} \left| \frac{\sqrt{2}}{n \|\Gamma \|_F} \sum_{i = 1 }^{ \lfloor an \rfloor  -1 } v_iD^T_{i} \widetilde{D}_{i+\lfloor \eta n \rfloor} \right|.
		\end{multline*}	
		It is sufficient to show that
		\begin{align*}
			\left\| \sup_{a \in (0, 1-\eta)} \Bigg| \frac{\sqrt{2}}{n \|\Gamma \|_F} \sum_{i = 1 }^{ \lfloor an \rfloor }  v_i D^T_{ i} \widetilde{D}_{i+\lfloor \eta n \rfloor} \Bigg| \right\|_2 = o(1).
		\end{align*} 
		To prove this, the idea is to use Proposition 1 in \cite{wu2007strong}, for any $n = 2^d$
		\begin{align*}
			\left\| \sup_{a \in (0, 1-\eta)} \Bigg| \frac{\sqrt{2}}{n \|\Gamma \|_F} \sum_{i = 1 }^{ \lfloor an \rfloor }  v_i D^T_{ i} \widetilde{D}_{i+\lfloor \eta n \rfloor} \Bigg| \right\|_2  \leq \sum\limits_{h=0}^d \left[ \sum_{u =1}^{2^{d-h}} \left\| \frac{\sqrt{2}}{2^{d} \|\Gamma \|_F} \sum_{i = 2^h(u-1)+1 }^{ 2^h u } v_i D^T_{ i} \widetilde{D}_{i+\lfloor \eta n \rfloor} \right\|_2^2 \right]^{1/2}.
		\end{align*}
		Applying Lemma \ref{lem:exp} and \ref{lem:cov},
		\begin{align*}
			& \left\| \frac{\sqrt{2}}{2^{d} \|\Gamma \|_F} \sum_{i = 2^h(u-1)+1 }^{ 2^h u } v_i D^T_{ i} \widetilde{D}_{i+\lfloor \eta n \rfloor} \right\|_2^2 \\
			\leq & \frac{2}{2^{2d} \|\Gamma \|_F^2} \sum_{i_1, i_2 = 2^h(u-1)+1 }^{ 2^h u } \left| E \left[ D^T_{ i_1} \widetilde{D}_{i_1+\lfloor \eta n \rfloor}D^T_{ i_2} \widetilde{D}_{i_2 +\lfloor \eta n \rfloor}  \right] \right| \\
			\leq &  \frac{2}{2^{2d} \|\Gamma \|_F^2} \sum_{i_1, i_2 = 2^h(u-1)+1 }^{ 2^h u } \sum_{l_1, l_2=1}^p \Bigg\{ |cum(D_{i_1, l_1},\widetilde{D}_{i_1+\lfloor \eta n \rfloor, l_1}, D_{i_2, l_2},\widetilde{D}_{i_2+\lfloor \eta n \rfloor, l_2} )| \\
			& \hspace{2cm} + |cum(  D_{ i_1, l_1} \widetilde{D}_{i_1+\lfloor \eta n \rfloor, l_1} )  cum(  D_{ i_2, l_2} \widetilde{D}_{i_2+\lfloor \eta n \rfloor, l_2} )| \\
			& \hspace{2cm} + |cum(D_{i_1, l_1}, D_{i_2, l_2})  cum(\widetilde{D}_{i_1+\lfloor \eta n \rfloor, l_1}, \widetilde{D}_{i_2+\lfloor \eta n \rfloor, l_2} )| \\
			& \hspace{2cm} +  |cum(D_{i_1, l_1}, \widetilde{D}_{i_2+\lfloor \eta n \rfloor, l_2} )cum(D_{i_2, l_2}, \widetilde{D}_{i_1+\lfloor \eta n \rfloor, l_1} )|  \Bigg\} \\
			\lesssim & \frac{2}{2^{2d} \|\Gamma \|_F^2} \sum_{i_1, i_2 = 2^h(u-1)+1 }^{ 2^h u } \bigg\{ \rho^{|i_1-i_2|} p^2 \rho^{ \lfloor \eta n \rfloor } + p^2\rho^{2 \lfloor \eta n \rfloor} \\
			& \hspace{3cm} + \rho^{|i_1-i_2|} \| \Gamma \|_F^2 + \rho^{|i_1-i_2 -\lfloor \eta n \rfloor |} \|\Gamma \|_F^2 \bigg\} \\
			\lesssim & 2^{h-2d}.
		\end{align*}
		To continue the calculation, 
		\begin{align*}
			\left\| \sup_{a \in (0, 1-\eta)} \Bigg| \frac{\sqrt{2}}{n \|\Gamma \|_F} \sum_{i = 1 }^{ \lfloor an \rfloor }  v_i D^T_{ i} \widetilde{D}_{i+\lfloor \eta n \rfloor} \Bigg| \right\|_2  \lesssim \sum\limits_{h=0}^d \left[ \sum_{u =1}^{2^{d-h}} 2^{h-2d} \right]^{1/2} = O\left(  \frac{d}{2^{d/2}}\right),
		\end{align*}
		which concludes the case when $n=2^d$. For arbitrary integer $n$, the statement follows from the fact that there exists $d$ such that $2^{d-1} \leq n < 2^{d}$ and 
		\begin{multline*}
			\left\| \sup_{a \in (0, 1-\eta)} \Bigg| \frac{\sqrt{2}}{n \|\Gamma \|_F} \sum_{i = 1 }^{ \lfloor an \rfloor }  v_i D^T_{ i} \widetilde{D}_{i+\lfloor \eta n \rfloor} \Bigg| \right\|_2 \\ \leq 2 \left\| \sup_{a \in (0, 1-\eta)} \Bigg| \frac{\sqrt{2}}{2^{d} \|\Gamma \|_F} \sum_{i = 1 }^{ \lfloor a 2^d \rfloor }  v_i D^T_{ i} \widetilde{D}_{i+\lfloor \eta n \rfloor} \Bigg| \right\|_2 =  O\left(  \frac{d}{2^{d/2}}\right) = o(1).
		\end{multline*} 
	\end{proof}
	
	\subsection{Proof of Lemma \ref{eq:final2}}
	Using Proposition 1 in \cite{wu2007strong}, for any $n = 2^d$
	\begin{multline} \label{eq:final}
		\left\| \sup_{a} \frac{\sqrt{2}}{n \|\Gamma \|_F}  \Bigg|  \sum_{i = 1 }^{ \lfloor an \rfloor } \sum_{j = 1 }^{ i } v_j \widetilde{D}_{i+ \lfloor \eta n \rfloor +1}^T D_{j}\Bigg| \right\|_2 \leq \\
		\sum\limits_{h=0}^d \left[ \sum_{u =1}^{2^{d-h}} \left\| \frac{\sqrt{2}}{2^{d} \|\Gamma \|_F} \sum_{i = 2^h(u-1)+1 }^{ 2^h u } \sum_{j = 1 }^{ i }v_j \widetilde{D}_{i+ \lfloor \eta n \rfloor +1}^T D_{j} \right\|_2^2 \right]^{1/2}.
	\end{multline}
	For the summands inside the bracket,
	\begin{align*}
		& \left\| \frac{\sqrt{2}}{2^{d} \|\Gamma \|_F} \sum_{i = 2^h(u-1)+1 }^{ 2^h u } \sum_{j = 1 }^{ i }v_j \widetilde{D}_{i+ \lfloor \eta n \rfloor +1}^T D_{j} \right\|_2^2 \\
		\leq &   \frac{2}{2^{2d} \|\Gamma \|_F^2} \sum_{i_1, i_2 = 2^h(u-1)+1 }^{ 2^h u } \sum_{j_1 = 1 }^{ i_1 } \sum_{j_2 = 1 }^{ i_2 } \sum_{l_1, l_2=1}^p \left|  E\left[  \widetilde{D}_{i_1+ \lfloor \eta n \rfloor +1, l_1} D_{j_1, l_1}\widetilde{D}_{i_2+ \lfloor \eta n \rfloor +1, l_2} D_{j_2, l_2}  \right] \right|.
	\end{align*}
	Express the expectation using cumulants
	\begin{align*}
		& \left|  E\left[  \widetilde{D}_{i_1+ \lfloor \eta n \rfloor +1, l_1} D_{j_1, l_1}\widetilde{D}_{i_2+ \lfloor \eta n \rfloor +1, l_2} D_{j_2, l_2}  \right] \right| \\
		\leq & |cum(\widetilde{D}_{i_1+ \lfloor \eta n \rfloor +1, l_1}, D_{j_1, l_1},\widetilde{D}_{i_2+ \lfloor \eta n \rfloor +1, l_2}, D_{j_2, l_2})| \\
		+ & |cum(\widetilde{D}_{i_1+ \lfloor \eta n \rfloor +1, l_1}, D_{j_1, l_1}) cum(\widetilde{D}_{i_2+ \lfloor \eta n \rfloor +1, l_2}, D_{j_2, l_2})| \\
		+ & |cum(\widetilde{D}_{i_1+ \lfloor \eta n \rfloor +1, l_1}\widetilde{D}_{i_2+ \lfloor \eta n \rfloor +1, l_2})cum( D_{j_1, l_1}, D_{j_2, l_2})| \\
		+ & |cum(\widetilde{D}_{i_1+ \lfloor \eta n \rfloor +1, l_1}, D_{j_2, l_2}) cum(D_{j_1, l_1},\widetilde{D}_{i_2+ \lfloor \eta n \rfloor +1, l_2})|.
	\end{align*}
	By using Lemma \ref{lem:exp} and \ref{lem:cov}, we have 
	\begin{align*}
		& \sum_{l_1, l_2=1}^p |cum(\widetilde{D}_{i_1+ \lfloor \eta n \rfloor +1, l_1}, D_{j_1, l_1},\widetilde{D}_{i_2+ \lfloor \eta n \rfloor +1, l_2}, D_{j_2, l_2})| \lesssim \rho^{i_1 \vee i_2  - j_1 \wedge j_2} p^2 \rho^{\lfloor \eta n \rfloor}, \\
		&\sum_{l_1, l_2=1}^p |cum(\widetilde{D}_{i_1+ \lfloor \eta n \rfloor +1, l_1}, D_{j_1, l_1}) cum(\widetilde{D}_{i_2+ \lfloor \eta n \rfloor +1, l_2}, D_{j_2, l_2})| \lesssim \rho^{i_1  - j_1} \rho^{i_2 - j_2} p^2 \rho^{\lfloor \eta n \rfloor}, \\
		&\sum_{l_1, l_2=1}^p  |cum(\widetilde{D}_{i_1+ \lfloor \eta n \rfloor +1, l_1}\widetilde{D}_{i_2+ \lfloor \eta n \rfloor +1, l_2})cum( D_{j_1, l_1}, D_{j_2, l_2})| \lesssim \rho^{|j_1  - j_2|} \| \Gamma \|_F^2, \\
		&\sum_{l_1, l_2=1}^p  |cum(\widetilde{D}_{i_1+ \lfloor \eta n \rfloor +1, l_1}, D_{j_2, l_2}) cum(D_{j_1, l_1},\widetilde{D}_{i_2+ \lfloor \eta n \rfloor +1, l_2})| \lesssim \rho^{|i_1 +\lfloor \eta n \rfloor - j_2| } \| \Gamma \|_F^2.
	\end{align*}
	Since $p^2 \rho^{\lfloor \eta n \rfloor} = O(\| \Gamma \|_F^2)$, some straightforward calculation shows that
	\begin{align*}
		& \left\| \frac{\sqrt{2}}{2^{d} \|\Gamma \|_F} \sum_{i = 2^h(u-1)+1 }^{ 2^h u } \sum_{j = 1 }^{ i } \widetilde{D}_{i+ \lfloor \eta n \rfloor +1}^T D_{j} \right\|_2^2 \\
		\lesssim  & \frac{1}{2^{2d}} \sum_{i_1, i_2 = 2^h(u-1)+1 }^{ 2^h u } \sum_{j_1 = 1 }^{ i_1 } \sum_{j_2 = 1 }^{ i_2 } \left\{ \rho^{i_1 \vee i_2  - j_1 \wedge j_2} + \rho^{i_1  - j_1} \rho^{i_2 - j_2} + \rho^{|j_1  - j_2|} +  \rho^{|i_1 +\lfloor \eta n \rfloor - j_2| } \right\} \\
		\lesssim & \frac{1}{2^{2d}} 2^{3h} u.
	\end{align*}
	Plugging the above bound into Equation \eqref{eq:final} results
	\begin{align*}
		\frac{1}{n} \left\| \sup_{a} \frac{\sqrt{2}}{n \|\Gamma \|_F}  \Bigg|  \sum_{i = 1 }^{ \lfloor an \rfloor } \sum_{j = 1 }^{ i } v_j \widetilde{D}_{i+ \lfloor \eta n \rfloor +1}^T D_{j}\Bigg| \right\|_2 \lesssim \frac{1}{2^d} \sum\limits_{h=0}^d \left[ \frac{1}{2^{2d}} 2^{3h} \sum_{u =1}^{2^{d-h}}   u \right]^{1/2} \lesssim 2^{-d/2},
	\end{align*}
	which concludes Equation \eqref{eq:final2} when $n=2^d$. For arbitrary $n$, there exists $d$ such that $2^{d-1} \leq n < 2^d $ and 
	\begin{multline*}
		\frac{1}{n} \left\| \sup_{a} \frac{\sqrt{2}}{n \|\Gamma \|_F}  \Bigg|  \sum_{i = 1 }^{ \lfloor an \rfloor } \sum_{j = 1 }^{ i } v_j \widetilde{D}_{i+ \lfloor \eta n \rfloor +1}^T D_{j}\Bigg| \right\|_2 \leq \\
		4 \frac{1}{2^{d}} \left\| \sup_{a} \frac{\sqrt{2}}{2^{d} \|\Gamma \|_F}  \Bigg|  \sum_{i = 1 }^{ \lfloor a2^{d} \rfloor } \sum_{j = 1 }^{ i } v_j \widetilde{D}_{i+ \lfloor \eta n \rfloor +1}^T D_{j}\Bigg| \right\|_2 \lesssim 2^{-d/2} = o(1). 
	\end{multline*}
	
	\subsubsection{Proof of Lemma \ref{lem:neg2}}
	For term \eqref{eq:re3}, we prove that
	\begin{align*}
		\left\|  \frac{\sqrt{2}}{n^2 \|\Gamma \|_F}  \sum_{i = \lfloor an \rfloor }^{ \lfloor bn \rfloor - \lfloor \eta n \rfloor -1 } \sum_{j = 1}^{ \lfloor an \rfloor -1}v_j \widetilde{D}_{i+ \lfloor \eta n \rfloor +1}^T D_{j} \right\|_6 \lesssim \frac{1}{\sqrt{n}}.
	\end{align*}
	Firstly, notice that 
	\begin{multline*}
		E \left[ \left( \frac{\sqrt{2}}{n^2 \|\Gamma \|_F}  \sum_{i = \lfloor an \rfloor }^{ \lfloor bn \rfloor - \lfloor \eta n \rfloor -1 } \sum_{j = 1}^{ \lfloor an \rfloor -1}v_j \widetilde{D}_{i+ \lfloor \eta n \rfloor +1}^T D_{j} \right)^6 \right] \\ 
		\lesssim  \frac{1}{n^{12} \|\Gamma \|_F^6} \sum_{i_1, \change{...}, i_6 = \lfloor an \rfloor }^{ \lfloor bn \rfloor - \lfloor \eta n \rfloor -1 } \sum_{j_1 = 1}^{ \lfloor an \rfloor -1} \cdots  \sum_{j_6 = 1}^{ \lfloor an \rfloor -1} \left|  E\left[ \prod_{k=1}^6 \widetilde{D}_{i_k+ \lfloor \eta n \rfloor +1}^T D_{j_k}   \right] \right|.
	\end{multline*}
	Then,  let $\pi$ be any partition of the index set $$ 
	\{  (i_1+ \lfloor \eta n \rfloor +1 ,l_1), (j_1,l_1), \change{...}, (i_6+ \lfloor \eta n \rfloor +1 ,l_6), (j_6,l_6) \}
	$$ 
	such that $|B|>1$ for any $B \in \pi$. By the  moment-cumulant formula, 
	\begin{align*}
		\left| E\left[ \prod_{k=1}^6 \widetilde{D}_{i_k+ \lfloor \eta n \rfloor +1}^T D_{j_k}   \right] \right| & = \left|  \sum\limits_{l_1, \change{...}, l_6 =1}^p E\left[ \prod_{k=1}^6 \widetilde{D}_{i_k+ \lfloor \eta n \rfloor +1, l_k} D_{j_k, l_k}   \right] \right| \\
		& \leq \sum\limits_{l_1, \change{...}, l_6 =1}^p \sum_{\pi} \prod_{B \in \pi} |cum(Z_{i,l} : (i,l) \in B )|,
	\end{align*}
	where $Z_{i,l} = D_{i,l} \text{ if } (i,l) \in \{ (j_k,l_k) \}_{k=1}^6;$ otherwise $ Z_{i,l}=  \widetilde{D}_{i,l}$. Set $\mathbb{I}_1 = \{  (i_k+ \lfloor \eta n \rfloor +1 ,l_k) \}_{k=1}^6 $, $\mathbb{I}_2 = \{  (j_k, l_k) \}_{k=1}^6 $ and 
	\begin{align*}
		& \pi_1 := \{A| A \in \pi,  A \subseteq \mathbb{I}_1 \},  \\
		& \pi_2: = \{A| A \in \pi, A \nsubseteq \mathbb{I}_1, A \nsubseteq \mathbb{I}_2 \}, \\
		& \pi_3 := \{A | A \in \pi,  A \subseteq \mathbb{I}_2 \}.
	\end{align*}
	Here, we note that $\pi = \pi_1 \cup \pi_2 \cup \pi_3$. For notational convenience, write
	\begin{align*}
		\begin{array}{ll}
			B = \{ (\tilde{i}_1+ \lfloor \eta n \rfloor +1 ,\tilde{l}_1), \change{...}, (\tilde{i}_{|B|}+ \lfloor \eta n \rfloor +1 ,\tilde{l}_{|B|}) \}, & \text{if } B \in \pi_1; \\
			B = \{ (\tilde{j}_1, \tilde{l}_1), \change{...}, (\tilde{j}_{|B|}, \tilde{l}_{|B|}) \}, & \text{if } B \in \pi_3,
		\end{array}
	\end{align*}
	and if $B \in \pi_2$, we can represent the set $B$  as
	\begin{align*}
		B = \big\{ (\tilde{i}_1+ \lfloor \eta n \rfloor +1 ,\tilde{l}_1), \change{...}, (\tilde{i}_{\tilde{k}}+ \lfloor \eta n \rfloor +1 ,\tilde{l}_{\tilde{k}}),  (\tilde{j}_1, \tilde{l}_1'), \change{...}, (\tilde{j}_{|B|-\tilde{k}}, \tilde{l}_{|B|-\tilde{k}}') \big\},
	\end{align*}
	where $0 < \tilde{k} < |B|$. Then, we can decompose the product of cumulants as
	\begin{align*}
		&\prod_{B \in \pi} cum(Z_{i,l} : (i,l) \in B ) \\
		= & \left\lbrace  \prod_{B \in \pi_1}  cum(\widetilde{D}_{\tilde{i}_1+ \lfloor \eta n \rfloor +1, \tilde{l}_1}, \change{...},\widetilde{D}_{\tilde{i}_{|B|}+ \lfloor \eta n \rfloor +1, \tilde{l}_{|B|}} ) \right\rbrace \\
		&  \times  \left\lbrace  \prod_{B \in \pi_2}  cum(\widetilde{D}_{\tilde{i}_1+ \lfloor \eta n \rfloor +1, \tilde{l}_1}, \change{...},\widetilde{D}_{\tilde{i}_{\tilde{k}}+ \lfloor \eta n \rfloor +1, \tilde{l}_{\tilde{k}}}, D_{\tilde{j}_1, \tilde{l}_1'}, \change{...}, D_{\tilde{j}_{|B|- \tilde{k}}, \tilde{l}_{|B| - \tilde{k}}'} )\right\rbrace  \\
		& \times \left\lbrace  \prod_{B \in \pi_3}  cum(D_{\tilde{j}_1, \tilde{l}_1}, \change{...}, D_{\tilde{j}_{|B|}, \tilde{l}_{|B|}} ) \right\rbrace.
	\end{align*}
	Apply Lemma \ref{lem:exp} and \ref{lem:cov},
	\begin{align*}
		\left\lbrace 
		\begin{array}{ll}
			cum(\widetilde{D}_{\tilde{i}_1+ \lfloor \eta n \rfloor +1, \tilde{l}_1}, \change{...},\widetilde{D}_{\tilde{i}_{|B|}+ \lfloor \eta n \rfloor +1, \tilde{l}_{|B|}} ) \lesssim \rho^{ i_{max}^B - i_{min}^B }, &  \text{ if } B \in \pi_1; \\
			\sum_{\tilde{l}_1, \change{...}, \tilde{l}_{|B|}=1}^p | cum(D_{\tilde{j}_1, \tilde{l}_1}, \change{...}, D_{\tilde{j}_{|B|}, \tilde{l}_{|B|}} )| \lesssim  \| \Gamma \|_F^{|B|}, & \text{ if } B \in  \pi_3,
		\end{array} \right. 
	\end{align*} 
	where  $i_{max}^B=\max\{ \tilde{i}_1, \change{...}, \tilde{i}_{|B|}  \}$ and $ i_{min}^B=\max\{ \tilde{i}_1, \change{...}, \tilde{i}_{|B|}  \} $. If $B \in \pi_2$,
	\begin{align*} 
		cum(\widetilde{D}_{\tilde{i}_1+ \lfloor \eta n \rfloor +1, \tilde{l}_1}, \change{...},\widetilde{D}_{\tilde{i}_{\tilde{k}}+ \lfloor \eta n \rfloor +1, \tilde{l}_{\tilde{k}}}, D_{\tilde{j}_1, \tilde{l}_1'}, \change{...}, D_{\tilde{j}_{|B|- \tilde{k}}, \tilde{l}_{|B| - \tilde{k}}'} ) \lesssim \rho^{i_{max}^B + \lfloor \eta n \rfloor - j_{min}^B}.
	\end{align*}
	where  $i_{max}^B=\max\{ \tilde{i}_1, \change{...}, \tilde{i}_{\tilde{k}}   \}$ and $ j_{min}^B = \min \{ \tilde{j}_1, \change{...}, \tilde{j}_{|B| - \tilde{k}} \} $. Thus,
	\begin{multline*}
		\sum_{i_1, \change{...}, i_6 = \lfloor an \rfloor }^{ \lfloor bn \rfloor - \lfloor \eta n \rfloor -1 } \sum_{j_1 = 1}^{ \lfloor an \rfloor -1} \cdots  \sum_{j_6 = 1}^{ \lfloor an \rfloor -1}  \sum\limits_{l_1, \change{...}, l_6 =1}^p \sum_{\pi} \prod_{B \in \pi} |cum(Z_{i,l} : (i,l) \in B )| \lesssim \\   \sum_{\pi}   \left\{  \prod_{B \in \pi_1} \left(   \sum_{ \tilde{i}_1, \change{...}, \tilde{i}_{|B|}= \lfloor an \rfloor }^{ \lfloor bn \rfloor  - \lfloor \eta n \rfloor -1} \rho^{ i_{max}^B - i_{min}^B }\right)\right\}  \left\{  \prod_{B \in \pi_3} \left( \sum_{\tilde{j}_1, \change{...}, \tilde{j}_{|B|} =1 }^{ \lfloor an \rfloor-1} p^{6 -  |B|}   \| \Gamma \|_F^{|B|} \right) \right\}  \\ \times \left\{  \prod_{B \in \pi_2} \left(   \sum_{\tilde{i}_1, \change{...}, \tilde{i}_{\tilde{k}} = \lfloor an \rfloor}^{ \lfloor bn \rfloor  - \lfloor \eta n \rfloor -1} \sum_{\tilde{j}_1, \change{...}, \tilde{j}_{|B| - \tilde{k}} =1 }^{ \lfloor an \rfloor-1} \rho^{ i_{max}^B - j_{min}^B } \rho^{ \lfloor \eta n \rfloor}\right) \right\}.
	\end{multline*}
	To continue the proof, notice that there exists a constant $N_{\rho}$ such that $m^6 \rho^{m/2} < 1$ if $m > N_{\rho}$. Thus, for any $B \in \pi_1$, we have
	\begin{align}\label{eq:sum1}
		\begin{split}
			& \sum_{\tilde{i}_1, \change{...}, \tilde{i}_{|B|} = \lfloor an \rfloor}^{\lfloor bn \rfloor - \lfloor \eta n \rfloor -1} \rho^{ i_{max}^B - i_{min}^B } \\
			\leq & \sum_{ \lfloor an \rfloor \le i _1 < i_2 \leq \lfloor bn \rfloor - \lfloor \eta n \rfloor -1} (i_2-i_1)^{ |B|-2} \rho^{i_2 - i_1} \\
			\lesssim & N_{\rho}^{|B|-2}\sum^{| i_1 - i_2 | \leq N_{\rho}}_{\lfloor an \rfloor \le i _1 < i_2 \leq \lfloor bn \rfloor - \lfloor \eta n \rfloor -1} \rho^{i_2 - i_1} + \sum^{| i_1 - i_2 | > N_{\rho}}_{\lfloor an \rfloor \le i _1 < i_2 \leq \lfloor bn \rfloor - \lfloor \eta n \rfloor -1} (\sqrt{\rho})^{i_2 - i_1} \\
			= & O(n).
		\end{split}
	\end{align}
	Also, it can be easily seen that
	\begin{align*}
		\prod_{B \in \pi_3} \left( \sum_{\tilde{j}_1, \change{...}, \tilde{j}_{|B|}=1 }^{ \lfloor an \rfloor-1} p^{6 -  |B|}   \| \Gamma \|_F^{|B|} \right) \lesssim n^{\sum_{B \in \pi_3}|B|} p^{6 - \sum_{B \in \pi_3} |B|}   \| \Gamma \|_F^{\sum_{B \in \pi_3}|B|}.
	\end{align*}
	In addition, for any $B \in \pi_2$
	\begin{align*}
		& \sum_{\tilde{i}_1, \change{...}, \tilde{i}_{\tilde{k}} = \lfloor an \rfloor}^{  \lfloor bn \rfloor - \lfloor \eta n \rfloor -1}\rho^{ i_{max}^B - \lfloor an \rfloor } \\
		\lesssim & \sum\limits_{i_1 = \lfloor an \rfloor}^{\lfloor bn \rfloor - \lfloor \eta n \rfloor -1} (i_{1} - \lfloor an \rfloor)^{ |B \cap \mathbb{I}_1|-1 } \rho^{i_{1} - \lfloor an \rfloor }  \\
		\lesssim & \sum\limits_{i_{1} - \lfloor an \rfloor > N_{\rho}} (\sqrt{\rho})^{ i_{1} - \lfloor an \rfloor } + O(1) \\
		= & O(1).
	\end{align*}
	It can shown similarly that $ \sum_{\tilde{j}_1, \change{...}, \tilde{j}_{|B|-\tilde{k}} =1}^{ \lfloor an \rfloor-1} \rho^{\lfloor an \rfloor - j_{min}^B} = O(1) $. Thus,
	\begin{multline*}
		\sum_{\tilde{i}_1, \change{...}, \tilde{i}_{\tilde{k}} = \lfloor an \rfloor}^{  \lfloor bn \rfloor - \lfloor \eta n \rfloor -1 } \sum_{\tilde{j}_1, \change{...}, \tilde{j}_{|B|-\tilde{k}} =1}^{ \lfloor an \rfloor-1} \rho^{ i_{max}^B - j_{min}^B } = \\
		\left\{   \sum_{\tilde{i}_1, \change{...}, \tilde{i}_{\tilde{k}} = \lfloor an \rfloor}^{  \lfloor bn \rfloor - \lfloor \eta n \rfloor -1 } \rho^{ i_{max}^B - \lfloor an \rfloor }\right\} \left\{  \sum_{\tilde{j}_1, \change{...}, \tilde{j}_{|B|-\tilde{k}} =1}^{ \lfloor an \rfloor-1} \rho^{\lfloor an \rfloor - j_{min}^B} \right\} = O(1).
	\end{multline*}
	In conclusion, we have
	\begin{multline*}
		E \left[ \left( \frac{\sqrt{2}}{n^2 \|\Gamma \|_F}  \sum_{i = \lfloor an \rfloor }^{ \lfloor bn \rfloor - \lfloor \eta n \rfloor -1 } \sum_{j = 1}^{ \lfloor an \rfloor -1} v_j \widetilde{D}_{i+ \lfloor \eta n \rfloor +1}^T D_{j} \right)^6 \right] \\
		\lesssim \left\lbrace  \begin{array}{ll}
			\frac{1}{n^{12} \|\Gamma \|_F^6} n^9 \|\Gamma \|_F^6  = \frac{1}{n^3}, & \text{if } \pi_2 = \emptyset; \\
			\frac{1}{n^{12} \|\Gamma \|_F^6} n^3 n^{\sum_{B \in \pi_3}|B|} \rho^{\lfloor \eta n \rfloor} p^{6 - \sum_{B \in \pi_3} |B|} \| \Gamma \|_F^{\sum_{B \in \pi_3}|B|} \lesssim \frac{1}{n^3}, &\text{if } \pi_2 \neq \emptyset .
		\end{array} \right. 
	\end{multline*}
	
	\subsubsection{Proof of Lemma \ref{lem:neg}}
	\begin{proof} For notational convenience, denote
		\begin{align*}
			\hat{R}(a,b|\eta) = \sum_{i = \lfloor an \rfloor }^{ \lfloor bn \rfloor - \lfloor \eta n \rfloor -1 } v_i  D_{i}^T  \widetilde{D}_{ \lfloor b n \rfloor }.
		\end{align*}
		Let $\pi$ be any disjoint partition over the set $\mathbb{I}$ such that $|B|>1$ for any $B \in \pi$, where $\mathbb{I}$ is defined as
		\begin{multline*}
			\mathbb{I}:=\big\{ (i_1,l_1),(i_2,l_2),(i_3,l_3),(i_4,l_4), (i_5, l_5), (i_6, l_6),\\ ( \lfloor bn \rfloor ,l_1),(\lfloor bn \rfloor,l_2),(\lfloor bn \rfloor,l_3),(\lfloor bn \rfloor,l_4), (\lfloor bn \rfloor, l_5), (\lfloor bn \rfloor, l_6)  \big\}.
		\end{multline*}
		Any such $\pi$ is a disjoint union of 3 sets as $\pi = \pi_1 \cup \pi_2 \cup \pi_3$, where $\pi_1 := \{A| A \in \pi,  A \subseteq \mathbb{I}_1 \}, \pi_2: = \{A| A \in \pi, A \nsubseteq \mathbb{I}_1, A \nsubseteq \mathbb{I}_2 \}, \pi_3 := \{A | A \in \pi,  A \subseteq \mathbb{I}_2\}$ and $\mathbb{I}_1, \mathbb{I}_2$ are defined as
		\begin{align*}
			\mathbb{I}_1 & := \big\{ (i_1,l_1),(i_2,l_2),(i_3,l_3),(i_4,l_4), (i_5, l_5), (i_6, l_6) \big\}, \\
			\mathbb{I}_2 & := \big\{ ( \lfloor bn \rfloor ,l_1),(\lfloor bn \rfloor,l_2),(\lfloor bn \rfloor,l_3),(\lfloor bn \rfloor,l_4), (\lfloor bn \rfloor, l_5), (\lfloor bn \rfloor, l_6)  \big\}. 
		\end{align*}
		For notational convenience, denote
		\begin{align*}
			Z_{i,l} = 
			\left\lbrace 
			\begin{array}{ll}
				D_{i,l} & (i,l) \in \mathbb{I}_1; \\
				\widetilde{D}_{i,l} & (i,l) \in \mathbb{I}_2.
			\end{array}
			\right. 
		\end{align*}
		Then, we have
		\begin{align*}
			E\left[ \hat{R}(a, b|\eta)^6 \right] 
			\leq  &\sum_{i_1,\change{...}, i_6 = \lfloor an\rfloor}^{ \lfloor bn\rfloor - \lfloor \eta n \rfloor  -1}   \sum_{l_1,\dots, l_6=1}^p \left| E \left[  D_{i_1, l_1} \widetilde{D}_{\lfloor bn \rfloor , l_1} \cdots  D_{i_6, l_6} \widetilde{D}_{\lfloor bn \rfloor , l_6} \right] \right| \\
			\leq & \sum_{i_1,\change{...}, i_6}  \sum_{l_1,\change{...}, l_6} \sum_{\pi} \prod_{B \in \pi} |cum(Z_{i,l} : (i,l) \in B )|.
		\end{align*}
		Similarly, we write
		\begin{align*}
			\begin{array}{ll}
				B = \{ (\tilde{i}_1 ,\tilde{l}_1), \change{...}, (\tilde{i}_{|B|} ,\tilde{l}_{|B|}) \}, & \text{if } B \in \pi_1; \\
				B = \{ ( \lfloor bn \rfloor, \tilde{l}_1), \change{...}, ( \lfloor bn \rfloor, \tilde{l}_{|B|}) \}, & \text{if } B \in \pi_3,
			\end{array}
		\end{align*}
		and if $B \in \pi_2$, we can represent the set $B$  as
		\begin{align*}
			B = \big\{ (\tilde{i}_1 ,\tilde{l}_1), \change{...}, (\tilde{i}_{\tilde{k}} ,\tilde{l}_{\tilde{k}}),  (\lfloor bn \rfloor, \tilde{l}_1'), \change{...}, (\lfloor bn \rfloor, \tilde{l}_{|B|-\tilde{k}}') \big\},
		\end{align*}
		where $0 < \tilde{k} < |B|$. Then, we can decompose the product of cumulants as
		\begin{align*}
			&\prod_{B \in \pi} cum(Z_{i,l} : (i,l) \in B ) \\
			= & \left\lbrace  \prod_{B \in \pi_1}  cum(D_{\tilde{i}_1, \tilde{l}_1}, \change{...},D_{\tilde{i}_{|B|}, \tilde{l}_{|B|}} ) \right\rbrace \\
			&  \times  \left\lbrace  \prod_{B \in \pi_2}  cum(D_{\tilde{i}_1, \tilde{l}_1}, \change{...}, D_{\tilde{i}_{ \tilde{k}}, \tilde{l}_{\tilde{k}}}, \widetilde{D}_{\lfloor b n \rfloor , \tilde{l}_1'}, \change{...},\widetilde{D}_{ \lfloor b n \rfloor , \tilde{l}_{|B|-\tilde{k}}' } )\right\rbrace  \\
			& \times \left\lbrace  \prod_{B \in \pi_3}  cum(\widetilde{D}_{\lfloor b n \rfloor , \tilde{l}_1}, \change{...},\widetilde{D}_{ \lfloor b n \rfloor , \tilde{l}_{|B|} } ) \right\rbrace.
		\end{align*}
		The following bounds on the cumulants are from Lemma \ref{lem:exp} and \ref{lem:cov} 
		\begin{align*}
			\left\lbrace 
			\begin{array}{ll}
				cum(D_{\tilde{i}_1, \tilde{l}_1}, \change{...},D_{\tilde{i}_{|B|}, \tilde{l}_{|B|}} ) \lesssim \rho^{ i_{max}^B - i_{min}^B }, &  \text{ if } B \in \pi_1; \\
				\sum_{\tilde{l}_1, \change{...}, \tilde{l}_{|B|}=1}^p |cum(\widetilde{D}_{\lfloor b n \rfloor , \tilde{l}_1}, \change{...},\widetilde{D}_{ \lfloor b n \rfloor , \tilde{l}_{|B|} } )| \lesssim  \| \Gamma \|_F^{|B|}, & \text{ if } B \in  \pi_3,
			\end{array} \right. 
		\end{align*}
		where  $i_{max}^B=\max\{ \tilde{i}_1, \change{...}, \tilde{i}_{|B|}  \}$ and $ i_{min}^B=\max\{ \tilde{i}_1, \change{...}, \tilde{i}_{|B|}  \} $. If $B \in \pi_2$,
		\begin{align*}
			cum(D_{\tilde{i}_1, \tilde{l}_1}, \change{...}, D_{\tilde{i}_{ \tilde{k}}, \tilde{l}_{\tilde{k}}}, \widetilde{D}_{\lfloor b n \rfloor , \tilde{l}_1'}, \change{...},\widetilde{D}_{ \lfloor b n \rfloor , \tilde{l}_{|B|-\tilde{k}}' } ) \lesssim \rho^{ \lfloor bn \rfloor - \lfloor \eta n \rfloor - i_{min}^B } \rho^{ \lfloor \eta n \rfloor},
		\end{align*}
		where $i_{min}^B = \min \{ \tilde{i}_1, \change{...}, \tilde{i}_{\tilde{k}} \}$. Thus,  we have
		\begin{multline*}
			E\left[ \hat{R}(a, b|\eta)^6 \right] \lesssim  \sum_{\pi}     \prod_{B \in \pi_1} \left(   \sum_{ \tilde{i}_1, \change{...}, \tilde{i}_{|B|} = \lfloor an \rfloor}^{\lfloor bn \rfloor - \lfloor \eta n \rfloor -1} \rho^{ i_{max}^B - i_{min}^B }\right) \times \\  \prod_{B \in \pi_2} \left(   \sum_{\tilde{i}_1, \change{...}, \tilde{i}_{\tilde{k}} = \lfloor an \rfloor}^{\lfloor bn \rfloor - \lfloor \eta n \rfloor -1} \rho^{ \lfloor bn \rfloor - \lfloor \eta n \rfloor - i_{min}^B } \rho^{ \lfloor \eta n \rfloor}\right) 
			p^{6 - \sum_{B\in \pi_3} |B|}   \| \Gamma \|_F^{\sum_{B\in \pi_3}|B|}  .
		\end{multline*}
		We know from Equation \eqref{eq:sum1} that for $B \in \pi_1$,
		\begin{align*} 
			\begin{split}
				\sum_{\tilde{i}_1, \change{...}, \tilde{i}_{|B|} = \lfloor an \rfloor}^{\lfloor bn \rfloor - \lfloor \eta n \rfloor -1} \rho^{ i_{max}^B - i_{min}^B } =  O(n).
			\end{split}
		\end{align*}
		For any $B \in \pi_2$, 
		\begin{align*}
			\begin{split}
				& \sum_{\tilde{i}_1, \change{...}, \tilde{i}_{\tilde{k}} = \lfloor an \rfloor}^{\lfloor bn \rfloor - \lfloor \eta n \rfloor -1} \rho^{ \lfloor bn \rfloor - \lfloor \eta n \rfloor - i_{min}^B } \\
				\lesssim & \sum\limits_{i_1 = \lfloor an \rfloor}^{\lfloor bn \rfloor - \lfloor \eta n \rfloor -1} (\lfloor bn \rfloor - \lfloor \eta n \rfloor - i_{1})^{ \tilde{k}-1 } \rho^{ \lfloor bn \rfloor - \lfloor \eta n \rfloor - i_{1} }  \\
				\lesssim & \sum\limits_{\lfloor bn \rfloor - \lfloor \eta n \rfloor - i_{1} > N_{\rho}} (\sqrt{\rho})^{ \lfloor bn \rfloor - \lfloor \eta n \rfloor - i_{1} } + O(1) \\
				= & O(1).
			\end{split}
		\end{align*}
		As a consequence, for any $\pi = \pi_1 \cup \pi_2 \cup \pi_3$, we have
		\begin{align*}
			\prod_{B \in \pi_1} \left(   \sum_{\tilde{i}_1, \change{...}, \tilde{i}_{|B|} = \lfloor an \rfloor}^{\lfloor bn \rfloor - \lfloor \eta n \rfloor -1} \rho^{ i_{max}^B - i_{min}^B }\right) = O(n^3),
		\end{align*}
		and
		\begin{align*}
			\prod_{B \in \pi_2 } \left(   \sum_{\tilde{i}_1, \change{...}, \tilde{i}_{\tilde{k}} = \lfloor an \rfloor}^{\lfloor bn \rfloor - \lfloor \eta n \rfloor -1} \rho^{ \lfloor bn \rfloor - \lfloor \eta n \rfloor - i_{min}^B } \right) =O(1).
		\end{align*}
		In conclusion, we have
		\begin{align*}
			\frac{1}{n^6 \|\Gamma \|_F^6}E\left[ \hat{R}(a, b|\eta)^6 \right]  \lesssim \left\lbrace   \begin{array}{ll}
				\frac{1}{n^6 \|\Gamma \|_F^6} n^3 \| \Gamma \|_F^6 = \frac{1}{n^3}, & \text{if } \pi_2 = \emptyset; \\
				\frac{1}{n^6 \|\Gamma \|_F^6} n^3 \rho^{\lfloor \eta n \rfloor} p^{6 - \sum_{B\in \pi_3} |B|}   \| \Gamma \|_F^{\sum_{B\in \pi_3}|B|} \lesssim \frac{1}{n^3}, & \text{if } \pi_2 \neq \emptyset.
			\end{array}\right. 
		\end{align*}
	\end{proof}
	
	\subsubsection{Proof of Lemma \ref{lem:alt}}
	
	Observe that under \ref{A41A3} we have $\delta_n^T \Gamma \delta_{n} = o(\|\delta_n\|^2 \|\Gamma\|_F)$, see the proof of Remark~\ref{rem:condA} for a detailed derivation of this type of bound.
	
	We only give details for the second inequality as others can be shown similarly. Note that for any $\delta_n \in \mathbb{R}^p$
	\begin{align*}
		\sup_{1 \leq l<k \leq n} \left|  \sum_{i=l}^{k} \frac{ i+\lfloor \eta n \rfloor +1 }{n} X_{ i+\lfloor \eta n \rfloor +1 }^T \delta_n  \right| \leq 2 \sup_{1 \leq k \leq n} \left| \sum_{i=1}^{k} \frac{ i+\lfloor \eta n \rfloor +1 }{n} X_{ i+\lfloor \eta n \rfloor +1 }^T \delta_n  \right|.
	\end{align*}
	Then, apply the BN decomposition to obtain
	\begin{multline*}
		\sup_{1 \leq k \leq n} \left| \sum_{i=1}^{k} \frac{ i+\lfloor \eta n \rfloor +1 }{n} X_{i+\lfloor \eta n \rfloor +1}^T \delta_n  \right| 
		\\
		\leq
		\sup_{1 \leq k \leq n} \left| \sum_{i=1}^{k} \frac{ i+\lfloor \eta n \rfloor +1 }{n} D_{i+\lfloor \eta n \rfloor +1}^T \delta_n  \right| +  \sup_{1 \leq k \leq n} \left|  \sum_{i=1}^{k} \frac{ i+\lfloor \eta n \rfloor +1 }{n} \varepsilon_{i+\lfloor \eta n \rfloor +1}^T \delta_n \right|.
	\end{multline*}
	For the first term on the right hand side, Kolmogorov's inequality implies that 
	\begin{align*}
		&\sup_{1 \leq k \leq n} \left| \sum_{i=1}^{k} \frac{ i+\lfloor \eta n \rfloor +1 }{n} D_{i+\lfloor \eta n \rfloor +1}^T \delta_n  \right| 
		\\
		= &  O_p\left( \left\{  \text{var} \left(\sum_{i=1}^{n} \frac{ i+\lfloor \eta n \rfloor +1 }{n} D_{i+\lfloor \eta n \rfloor +1}^T \delta_n \right) \right\}^{1/2} \right) 
		\\
		= & O_p \left( n^{1/2} ( \delta_n^T \Gamma \delta_{n}  )^{1/2} \right) = o_P(n^{1/2}\|\delta_n\|\|\Gamma\|_F^{1/2}).
	\end{align*}
	Next, we bound $\sup_{1 \leq k \leq n} |  \sum_{i=1}^{k} (i+\lfloor \eta n \rfloor +1)/n \varepsilon_{i+\lfloor \eta n \rfloor +1}^T \delta_n |$. Some algebra show that
	\begin{multline*}
		\sum_{i=1}^{k} \frac{i+\lfloor \eta n \rfloor +1}{n} \varepsilon_{i+\lfloor \eta n \rfloor +1}^T \delta_n =  \frac{k + \lfloor \eta n \rfloor +1}{n} \widetilde{D}_{k+ \lfloor \eta n \rfloor +1}^T\delta_n 
		\\ 
		- \frac{1 }{n} \left( \widetilde{D}_{\lfloor \eta n \rfloor +2} +  \cdots + \widetilde{D}_{k+\lfloor \eta n \rfloor} \right)^T\delta_n - \frac{\lfloor \eta n \rfloor +2}{n}\widetilde{D}_{\lfloor \eta n \rfloor +1}^T\delta_n .
	\end{multline*}
	According to Proposition 1 in \cite{wu2007strong}, for any $n = 2^d$,
	\begin{align*}
		\left[ E \left( \sup_{1 \leq k \leq n} \left| \sum_{i=1}^{k} \widetilde{D}_{i+\lfloor \eta n \rfloor+1}^T \delta_n  \right| \right)^2\right]^{1/2} \leq \sum_{h=1}^d \left[ \sum_{u=1}^{2^{d-h}} E \left( \sum_{i=2^h(u-1)+1}^{2^h u} \widetilde{D}_{i+\lfloor \eta n \rfloor+1}^T \delta_n \right)^2 \right]^{1/2}.
	\end{align*}
	For each term in the square bracket
	\begin{align*}
		& E \left( \sum_{i=2^h(u-1)+1}^{2^h u}  \widetilde{D}_{i+\lfloor \eta n \rfloor+1}^T \delta_n \right)^2 
		\\  
		= & \left|  \sum_{2^h(u-1)+1 \leq i_1,  i_2 \leq 2^h u} \sum_{j_1, j_2=1}^p  \delta_{n,j_1}\delta_{n,j_2} E[ \widetilde{D}_{i_1+\lfloor \eta n \rfloor+1, j_1} \widetilde{D}_{i_2+\lfloor \eta n \rfloor+1, j_2} ] \right| \\
		\lesssim & \sum_{2^h(u-1)+1 \leq i_1,  i_2 \leq 2^h u} \left\{ \sum_{j_1, j_2=1}^p 
		\delta_{n,j_1}^2 \delta_{n, j_2}^2  \right\}^{1/2} \left\{ \sum_{j_1, j_2=1}^p E^2[ \widetilde{D}_{i_1+\lfloor \eta n \rfloor+1, j_1} \widetilde{D}_{i_2+\lfloor \eta n \rfloor+1, j_2} ] \right\}^{1/2} 
		\\
		\lesssim & 2^{2h} \| \delta_n \|_2^2 \|\Gamma \|_F.
	\end{align*}
	So, we have
	\begin{align*}
		\left[ E \left( \sup_{1 \leq k \leq n} \left| \sum_{i=1}^{k} \widetilde{D}_{i+\lfloor \eta n \rfloor+1}^T \delta_n  \right| \right)^2\right]^{1/2} \lesssim \| \delta_n \|_2 \|\Gamma \|_F^{1/2} \sum_{h=1}^d 2^{(d+h)/2} \lesssim \| \delta_n \|_2 \|\Gamma \|_F^{1/2} 2^d.
	\end{align*}
	For general $n$, there exists $d$ such that $2^{d-1} \leq n < 2^d$ and 
	\begin{multline*}
		\left[ E \left( \sup_{1 \leq k \leq n} \left| \sum_{i=1}^{k} \widetilde{D}_{i+\lfloor \eta n \rfloor+1}^T \delta_n  \right| \right)^2\right]^{1/2} \leq \left[ E \left( \sup_{1 \leq k \leq 2^d} \left| \sum_{i=1}^{k} \widetilde{D}_{i+\lfloor \eta n \rfloor+1}^T \delta_n  \right| \right)^2\right]^{1/2} 
		\\ 
		\lesssim \| \delta_n \|_2 \|\Gamma \|_F^{1/2} 2^d  \lesssim \| \delta_n \|_2 \|\Gamma \|_F^{1/2} n.
	\end{multline*}
	The above inequality implies
	\begin{align*}
		\frac{1}{\|\Gamma \|_F} \sup_{1 \leq k \leq n} \left|\frac{1}{n} \sum_{i=1}^{k} \widetilde{D}_{i+\lfloor \eta n \rfloor+1}^T \delta_n  \right| = O_p\left( \frac{ \| \delta_n \|_2 }{\|\Gamma \|_F^{1/2}} \right) = o_p\left( \frac{\sqrt{n}\| \delta_n \|_2 }{ \|\Gamma \|_F^{1/2} } \right).
	\end{align*}
	Next, the maximal inequality implies that 
	\begin{align*}
		E\left[  \max_k \left| \frac{k + \lfloor \eta n \rfloor +1}{n} \widetilde{D}_{k+ \lfloor \eta n \rfloor +1}^T\delta_n \right|\right] \lesssim n^{1/4} \max_k \left[E\left(\widetilde{D}_{k+ \lfloor \eta n \rfloor +1}^T\delta_n. \right)^4\right]^{1/4}.
	\end{align*}
	Then, by Cauchy's inequality
	\begin{align*}
		& E \left[ ( \widetilde{D}_{k+ \lfloor \eta n \rfloor +1}^T\delta_n )^4 \right] 
		\\
		\leq & \sum_{j_1, j_2, j_3, j_4=1}^p  |\delta_{n,j_1}\delta_{n,j_2}\delta_{n,j_3}\delta_{n,j_4}   E[ \widetilde{D}_{k+ \lfloor \eta n \rfloor +1, j_1}\widetilde{D}_{k+ \lfloor \eta n \rfloor +1, j_2} \widetilde{D}_{k+ \lfloor \eta n \rfloor +1, j_3}  \widetilde{D}_{k+ \lfloor \eta n \rfloor +1, j_4}  ] | 
		\\
		\lesssim & \left\{ \sum_{j_1, j_2, j_3, j_4=1}^p \delta_{n,j_1}^2 \delta_{n,j_2}^2 \delta_{n,j_3}^2 \delta_{n,j_4}^2 \right\}^{1/2} \times 
		\\
		& \hspace{1cm} \left\{ \sum_{j_1, j_2, j_3, j_4=1}^p  E^2[ \widetilde{D}_{k+ \lfloor \eta n \rfloor +1, j_1}\widetilde{D}_{k+ \lfloor \eta n \rfloor +1, j_2} \widetilde{D}_{k+ \lfloor \eta n \rfloor +1, j_3}  \widetilde{D}_{k+ \lfloor \eta n \rfloor +1, j_4}  ] \right\}^{1/2}.
	\end{align*}
	Applying Lemma \ref{lem:cov},
	\begin{align*}
		& \sum_{j_1, j_2, j_3, j_4=1}^p  E^2[ \widetilde{D}_{k+ \lfloor \eta n \rfloor +1, j_1}\widetilde{D}_{k+ \lfloor \eta n \rfloor +1, j_2} \widetilde{D}_{k+ \lfloor \eta n \rfloor +1, j_3}  \widetilde{D}_{k+ \lfloor \eta n \rfloor +1, j_4}  ] 
		\\
		\lesssim & \sum_{j_1, j_2, j_3, j_4=1}^p cum^2( \widetilde{D}_{k+ \lfloor \eta n \rfloor +1, j_1},\widetilde{D}_{k+ \lfloor \eta n \rfloor +1, j_2}, \widetilde{D}_{k+ \lfloor \eta n \rfloor +1, j_3},  \widetilde{D}_{k+ \lfloor \eta n \rfloor +1, j_4} ) 
		\\
		+ & \sum_{j_1, j_2, j_3, j_4=1}^p cum^2( \widetilde{D}_{k+ \lfloor \eta n \rfloor +1, j_1},\widetilde{D}_{k+ \lfloor \eta n \rfloor +1, j_2})cum^2(\widetilde{D}_{k+ \lfloor \eta n \rfloor +1, j_3},  \widetilde{D}_{k+ \lfloor \eta n \rfloor +1, j_4} )
		\\
		+ & \sum_{j_1, j_2, j_3, j_4=1}^p cum^2( \widetilde{D}_{k+ \lfloor \eta n \rfloor +1, j_1},\widetilde{D}_{k+ \lfloor \eta n \rfloor +1, j_3})cum^2( \widetilde{D}_{k+ \lfloor \eta n \rfloor +1, j_2},  \widetilde{D}_{k+ \lfloor \eta n \rfloor +1, j_4} ) 
		\\
		+ &  \sum_{j_1, j_2, j_3, j_4=1}^p cum^2( \widetilde{D}_{k+ \lfloor \eta n \rfloor +1, j_1}, \widetilde{D}_{k+ \lfloor \eta n \rfloor +1, j_4})cum^2(\widetilde{D}_{k+ \lfloor \eta n \rfloor +1, j_3},  \widetilde{D}_{k+ \lfloor \eta n \rfloor +1, j_2} ) 
		\\
		\lesssim & \| \Gamma \|_F^4,
	\end{align*}
	which entails that $ E[( \widetilde{D}_{k+ \lfloor \eta n \rfloor +1}^T\delta_n )^4] \lesssim \| \delta_n \|_2^4 \| \Gamma \|_F^2 $ and 
	\begin{align*}
		\max_k \left| \frac{1}{\| \Gamma \|_F} \frac{k + \lfloor \eta n \rfloor +1}{n} \widetilde{D}_{k+ \lfloor \eta n \rfloor +1}^T\delta_n \right| = O_p \left( \frac{ n^{1/4} \| \delta_n \|_2}{ \| \Gamma \|_F^{1/2} }   \right) = o_p\left( \frac{\sqrt{n}\| \delta_n \|_2 }{ \|\Gamma \|_F^{1/2} } \right).
	\end{align*} 
	\hfill $\Box$

	\section{Testing for Covariance Matrix Change}
	\label{sec:covsim}

	In this section, we examine the finite sample performance of our test applied to test for a change in the covariance matrix, in comparison with a recent method developed by \cite{avanesov2016change}. In the latter paper, they proposed a high dimensional covariance change point detection scheme that involves the choices of several tuning parameters. For the purpose of completeness, we present their method below in detail.
	
	They first consider a set of window sizes $\mathcal{N} \in \mathbb{N}$. For each window size $n \in \mathcal{N}$, define a set of central points $\mathbb{T}_n := \{n+1,\change{...},N-n+1\}$, where $N$ is the sample size. For $n \in \mathcal{N}$ define a set of indices belong to the window on the left side from the central point $t \in \mathbb{T}_n$ as $\mathcal{I}_n^l(t) := \{t-n,\change{...},t-1\}$ and indices in the right side $\mathcal{I}_n^r(t) := \{t+1,\change{...},t+n\}$. Denote the sum of number of central points for all window sizes $n \in \mathcal{N}$ as $T := \sum_{n \in \mathcal{N}}|\mathbb{T}_n|$.
	
	For each window size $n$, each center point $t$ and either left side or right side $\mathfrak{G} \in \{l,r\}$, they define a de-sparsified estimator of precision matrix as
	$$\hat{T}_n^{\mathfrak{G}}(t) := \hat{\Theta}_n^{\mathfrak{G}}(t) + \hat{\Theta}_n^{\mathfrak{G}}(t)^T - \hat{\Theta}_n^{\mathfrak{G}}(t)^T\hat{\Sigma}_n^{\mathfrak{G}}(t)\hat{\Theta}_n^{\mathfrak{G}}(t),$$
	where
	$$\hat{\Sigma}_n^{\mathfrak{G}}(t) = \frac{1}{n}\sum_{i \in \mathcal{I}_n^{\mathfrak{G}(t)}}X_iX_i^T,$$
	and $\hat{\Theta}_n^{\mathfrak{G}}$ is the precision matrix estimated by Graphical Lasso. Define a $p \times p$ matrix with elements
	$$Z_{i,uv}  := \Theta_u^*X_i\Theta_v^*X_i - \Theta_{uv},$$
	where $\Theta^* := \mathbb{E}[X_iX_i^T]^{-1}$ for all data before the change point location, $\Theta_u^*$ is the $u-th$ row and denote the variance as $\sigma_{u,v} := Var(Z_{1,uv})$ and the diagonal matrix ${S} = diag(\sigma_{1,1},\sigma_{1,2},\change{...},\sigma_{p,p-1},\sigma_{p,p})$. Finally their test statistic is
	$$A = \max_{n\in \mathcal{N},t\in \mathbb{T}_n} \left\|\sqrt{\frac{n}{2}}S^{-1}\overline{(\hat{T}_n^l(t) - \hat{T}_n^r(t))}\right\|_{\infty},$$
	where $\bar{M}$ means the vector composed of stacked columns of the matrix $M$. Their test rejects the null hypothesis when the above statistic is greater than some critical value, which is determined via a bootstrap procedure.  More details can be found in \cite{avanesov2016change}.

	Here we let  $X_t's$ be $p$-dimensional multivariate normal random vectors with mean ${0}$ and variance $\boldsymbol{\Sigma}_t$. We fix  $p = 10$ and the sample size as $n = 100$ or $200$. Under the null, we set the common covariance matrix as (1) $0.8I_p$ or (2) $AR(0.4)$. Under the alternative, we let $\boldsymbol{\Sigma}_1 = \cdots = \boldsymbol{\Sigma}_{n/2} \neq \boldsymbol{\Sigma}_{n/2+1} = \cdots = \boldsymbol{\Sigma}_n$, where $\Sigma_{n/2}$ is (1) $0.8I_p$ (2) $AR(0.8)$ and $\Sigma_{n/2+1}$ is (1) $0.4I_p$ (2) $AR(0.4)$.

	The results are summarized in Table \ref{simulation:cov}, where our method is denoted as ``SN" and the other method as ``AB". There is no tuning parameter in our method, however a few tuning parameters need to be specified for the method ``AB". In particular,  the window size was chosen as 30, and the stable set which was used to estimate the precision matrix was chosen as $\{1,2,\change{...},40\}$. As we can see from Table \ref{simulation:cov}, there is a huge size distortion with the ``AB" test, which could be due  to the way  the tuning parameter is selected.   By contrast, our SN method has fairly accurate size.   In terms of the power, SN method is powerful under the alternative. ``AB" test has a perfect rejection rate under the alternative, but this shall not be taken too seriously given the huge over-rejection under the null. Overall the SN method seems quite favorable given its accurate size and reasonable power as well as the  tuning-free implementation. It is worth mentioning that there is really no guidance or data-driven formula provided as to the choice of tuning parameters in \cite{avanesov2016change}. We tried several choices but all of them delivered large size distortion, which indicates the choice of tuning parameters is indeed a difficult issue for their test.
	

	\begin{table}[h!]
		\centering
		
		\begin{tabular}{|c|c|cc|cc|}
			
			\hline
			&&\multicolumn{2}{c|}{Diagonal}&\multicolumn{2}{c|}{AR}\\
			\cline{2-6}
			&&$\mathcal{H}_0$&$\mathcal{H}_1$&$\mathcal{H}_0$&$\mathcal{H}_1$\\
			\hline
			\multirow{2}{*}{$N = 100, p = 10$}&$SN$&5.0&91.4&6.5&90.0\\
			\cline{2-6}
			&$AB$&66.6&100&73.8&100\\
			\hline
			\multirow{2}{*}{$N = 200, p = 10$}&$SN$&4.3&100&4&100\\
			\cline{2-6}
			&$AB$&46&100&83&100\\
			\hline
			
		\end{tabular}
		\caption{Empirical Rejection Rates for Tests of Covariance Matrix Change}	\label{simulation:cov}
	\end{table}

	\newpage

	\section{Simulation results for change-point estimation}
	\label{sec:WBSsim}
	
	In this section, we present the WBS-based estimation results and compare with a few other alternative methods via simulations, for independent data in Section~\ref{sub:estimate} and time series in Section~\ref{sub:estimateTS}.
	
	\subsection{Change-point estimation: independent data}
	\label{sub:estimate}
	
	As described in Section~\ref{sec:estimation}, we can combine the WBS idea with the self-normalized statistics to estimate  the number and locations for change points in the mean of high-dimensional independent data. In this subsection, we compare our WBS method (denoted as WBS-SN, with $L_0=10$) with  binary segmentation(BS-SN) and INSPECT, the latter of which was developed by \cite{wang2018high} targeting sparse and strong changes.
	
	Following \cite{wang2018high}, we consider a three change-points model and the change points are located at $[n/4]$, $2[n/4]$, and $3[n/4]$. The mean vectors for those four different zones are $\mu_1,\change{...},\mu_4$. Thus we draw $\floor{n/4}$ i.i.d sample from $N(\mu_i,\sigma^2I_p)$ for each zone.  We define $\theta_1,\theta_2,\theta_3$ as three signals at change points, i.e. $\theta_i = \mu_{i+1} - \mu_i$ and $\nu_i = \|\theta_i\|_2$ as the signal strength for $i = 1, 2, 3$. Denote $s = \|\theta_i\|_0$ for all $i$ as the sparsity level.
	Specifically we let $n = 120, p = 50$ and set $\sigma = 1$. The total number of random segments used in WBS-SN is fixed as $M = 1000$. As we described before, we choose the threshold for WBS based on the reference sample. For INSPECT, we use all default parameters  in the  "InspectChangepoint" package in R. We consider two cases for the alternative, one is sparse where $s = 5$ and the other one is dense where $s = p = 50$. We denote the true number of change points as $N = 3$, and the estimated number is $\hat{N}$. The true location of change points are 30, 60 and 90.
	
	For sparse case, we set $\theta_1 = 2(k,k,k,k,k,0,\change{...},0)^T$, $\theta_2 = -2(k,k,k,k,k,0,\change{...},0)^T$, and	$\theta_3 = 2(k,k,k,k,k,0,\change{...},0)^T$ where $k \in \{ \sqrt{2.5/5},\sqrt{4/5}\}$. For the dense alternative, we set $\theta_1 = 2k\times\boldsymbol{1}_p$, $\theta_2 = -2k\times\boldsymbol{1}_p$ and $\theta_3 = 2k\times\boldsymbol{1}_p$, and let $k \in \{ \sqrt{2.5/p},\sqrt{4/p}\}$.
	To measure the estimation accuracy for the number of change points, we simply use Mean Squared Error between the estimated number and the truth; for the location estimate, since the change point location estimation can be viewed as a special case for classification, we utilize a metric called ``Averaged Rand Index", denoted as ARI to quantify the accuracy. See \cite{rand1971objective}, \cite{hubert1985comparing} and \cite{wang2018high}. The ARI is a positive value between 0 and 1. When estimation is perfect, the ARI is 1. If there is no change points estimated, the corresponding ARI is 0. The higher the ARI, the more accurate the estimation. Here we get ARI for each replicate and finally take the average to get the averaged ARI.


	As suggested by a referee, we further implemented another method based on consistent estimation of $\|\Sigma\|_F$ by dividing the sample into three equal parts and using the median of the Jackknife-based estimator for each part, mimicking an idea first proposed in \cite{LGS2021}. Then we couple the studentized test statistic with this consistent estimator and WBS, but did not find substantial gain in the (unreported) simulation studies.  Note that an extension of this idea to the time series setting seems nontrivial, as it will involve bandwidths when forming a consistent estimator of $\|\Gamma\|_F$ and theoretical justification in the testing context is expected to be challenging.

	\centerline{Please insert Table~\ref{simulation:WBS} here!}
	
	As seen from Table \ref{simulation:WBS}, binary segmentation does not work at all in all cases due to the non-monotonic change in the mean, whereas both WBS-SN and INSPECT provide more sensible estimates. To estimate the number of change points, WBS-SN outperforms INSPECT in the two dense cases and the Sparse$(\sqrt{4/5})$ case, whereas the performance of INSPECT in the Sparse$(\sqrt{2.5/5})$ case is superior; for the change point location estimation, WBS-SN is inferior to INSPECT in the Sparse$(\sqrt{2.5/5})$ case, which is probably not superising. For the other three cases, their performance is comparable. These findings are in general consistent with our intuition that  WBS-SN targets dense alternative and INSPECT targets sparse alterative. They suggest WBS-SN can be a useful complement to INSPECT as in practice we may not know a priori whether the change is sparse or dense.

	\subsection{Change-point estimation: time series}
	\label{sub:estimateTS}
	
	For multiple change point estimation in the mean of high-dimensional time series, we compare our WBS-SN [see Algorithm \ref{algorithm:wbs}] with the double CUSUM binary segmentation algorithm (denoted as DCBS) [\cite{cho2016change}] and the segmentation algorithm based on a bias-corrected statistic in \cite{li2019change} (denoted as Li). The latter two methods have been implemented in the R packages ``hdbinseg" and ``HdcpDetect", respectively.
	\begin{example}
		\label{exp:2}
		Consider the  model $Y_{t} = \mu_t  + X_{t}$,
		where $X_{t} = (X_{t,1}, X_{t,2}, \change{...}, X_{t,p})^T$ is generated from the following three models.
		\begin{itemize}
			\item [(i)] Gaussian errors with AR(1) type convariance structure: set $\Sigma_{\epsilon} = (0.5^{|i-j|})_{i,j=1}^p$. For $t=1,2, \change{...}, n$, let $\epsilon_{t} \overset{i.i.d}{\sim} N(0, \Sigma_{\epsilon})$ and
			$X_{t}  = \rho X_{t-1} + \epsilon_{t}.$
			\item[(ii)] Non-Gaussian errors: for $t=1,2, \change{...}, n$, $j = 1,2, \change{...}, p$
			$ \epsilon_{t,j} \overset{i.i.d}{\sim} Uniform(-2,2)$, and 
			$ X_{t}  = \rho X_{t-1} + \epsilon_{t}.$
			\item[(iii)] Motivated by the simulation models used in \cite{cho2016change}, we let $\varrho_k = 0.6 (k+1)^{-1}$ and define 
			$ \epsilon_{t,j} = \sum_{k=0}^{99} \varrho_k v_{t,j-k}, \text{ where }  v_{t,j} \overset{i.i.d}{\sim} N(0, 1)$, and 
			$ X_{t,j} = \rho X_{t-1,j} + \epsilon_{t,j} + 0.2 \epsilon_{t-1,j}$ for  $t=1,2, \change{...}, n$, $j = 1,2, \change{...}, p$. 
		\end{itemize}
		Further, we let 
		\begin{align*}
			\mu_t = \boldsymbol{\bf\delta}_1 \1_{t > k_1} + \boldsymbol{\bf\delta}_2 \1_{t > k_{2}} + \boldsymbol{\bf\delta}_3 \1_{t > k_3},
		\end{align*}
		where for $ r=1,2,3 $, $\boldsymbol{\bf\delta}_r = (\delta_{r,1}, \delta_{r,1}, \change{...}, \delta_{r,p})^T \in \mathbb{R}^p$. Denote $\Pi_r = \{ j| |\delta_{r, j}|>0, j=1,2, \change{...}, p \}$ and $ |\delta_{r, j}| \sim^{i.i.d} Uniform(0.75 \theta_r, 1.25 \theta_r)$ for $j \in \Pi_r$. The signs of $\{ \delta_{r, j} \}$ are randomly sampled with equal probability. Here, we set 
		$ (k_1, |\Pi_1|,\theta_1) = (\lfloor 0.3 n \rfloor, \lfloor 0.75 p \rfloor , 0.4)$, 
		$ (k_2, |\Pi_2|,\theta_2) = (\lfloor 0.6 n \rfloor, \lfloor 0.25 p \rfloor, 0.696)$, and  
		$ (k_3, |\Pi_3|,\theta_3) = (\lfloor 0.8 n \rfloor, \lfloor 0.1 p \rfloor, 1.12)$.
		
	\end{example}
	
	\centerline{Please insert Table~\ref{wbs:old} here!}

	Table~\ref{wbs:old} reports the estimation results in terms of the frequency for the estimated number of change points among 200 simulations and ARI for WBS-SN (based on two trimming levels $\eta=0.01$ and $0.02$ and three choices of $L_0=6\lfloor n\eta\rfloor+7+\lfloor\theta n\rfloor$ with $\theta=0.1,0.15,0.2$), DCBS and Li. We fix the sample size $n=500$ and let $p=250$ and $500$.  It appears that when the temporal dependence is weak (i.e., $\rho=0.3$), DCBS performs the best in terms of estimation accuracy for the number and location of change points for all settings, WBS-SN with either trimming level and any choice of $L_0$ is comparable to Li in terms of ARI but outperforms Li in terms of estimation of the number of change points. Li's method tends to overestimate the number of change points in all settings. When the temporal dependence is moderately strong (i.e., $\rho=0.6$), WBS-SN outperforms both DCBS and Li according to both criteria. In this case, DCBS tends to underestimate the number of change-points, resulting in a small ARI.   Overall our WBS-SN method seems fairly competitive and performs quite stably for two levels of temporal dependence. The trimming level $\eta$ could have an impact on the estimation accuracy, and the magnitude of impact could depend on the data generating process (in particular, the magnitude of temporal dependence), the dimensionality and sample size etc. The choice of $L_0$ seems to have little impact on the performance in most cases, suggesting that  its optimal choice may not be necessary as long as it is in certain range.

	\clearpage
	\newpage
	\begin{table}[h!]
		\centering
		\begin{tabular}{cc|ccccccc|c|c}
			\hline
			&&\multicolumn{7}{c|}{$\hat{N} - N$}&MSE&ARI\\
			\hline
			&&-3&-2&-1&0&1&2&3&&\\
			\hline
			\multirow{3}{*}{Sparse($\sqrt{2.5/5}$)}&WBS-SN&2&12&38&48&0&0&0&1.04&0.75\\
			\hline
			&BS-SN&100&0&0&0&0&0&0&9&0\\
			\hline
			&INSPECT&0&16&1&76&7&0&0&0.72&0.85\\
			\hline
			\hline
			\multirow{3}{*}{Sparse ($\sqrt{4/5}$)}&WBS-SN&0&0&1&96&3&0&0&0.04&0.95\\
			\hline
			&BS-SN&100&0&0&0&0&0&0&9&0\\
			\hline
			&INSPECT&0&0&0&83&17&0&0&0.17&0.96\\
			\hline
			\hline
			\multirow{3}{*}{Dense($\sqrt{2.5/p}$)}&WBS-SN&2&13&36&49&0&0&0&1.06&0.70\\
			\hline
			&BS-SN&100&0&0&0&0&0&0&9&0\\
			\hline
			&INSPECT&0&30&2&45&19&4&0&1.57&0.69\\
			\hline
			\hline
			\multirow{3}{*}{Dense($\sqrt{4/p}$)}&WBS-SN&0&0&1&92&7&0&0&0.08&0.95\\
			\hline
			&BS-SN&100&0&0&0&0&0&0&9&0\\
			\hline
			&INSPECT&0&6&0&72&17&5&0&0.61&0.92\\
			\hline
		\end{tabular}
		\caption{Estimation Result for Multiple Change Points in Mean of High-dimensional Independent Data}\label{simulation:WBS}
	\end{table}
	
	\clearpage
	\newpage

	\begin{table}[h!] \centering 
		\begin{footnotesize}
			\caption{Estimation Result for Multiple Change Points in Mean of High-dimensional Time Series: Example 6.2 with $n=500$ and $p=250, 500$ based on 200 simulations. $\text{WBS-SN}^1$ and $\text{WBS-SN}^2$ correspond to trimming $\eta = 0.01$ and $\eta = 0.02$ respectively. The minimal length $L_0=6\tau+7+\lfloor \theta n\rfloor$, where $\tau=\lfloor n\eta\rfloor$ and $\theta=0.1,0.15,0.2$. The rows labeled with DCBS-Li correspond to methods of double CUSUM binary segmentation algorithm (left) in Cho (2016) and segmentation algorithm
				based on a bias-corrected statistic (right) in Li et al. (2019).  } 
			\label{wbs:old} 
			\begin{tabular}{@{\extracolsep{5pt}} ccccccccc|ccccc} 
				\\[-1.8ex]\hline 
				\hline \\[-1.8ex] 
				&  & &  &   \multicolumn{5}{c|}{$\text{WBS-SN}^1$}  &   \multicolumn{5}{c}{$\text{WBS-SN}^2$}   \\ 
				\multirow{2}{*}{Case} & \multirow{2}{*}{$p$} & \multirow{2}{*}{$\rho$} & \multirow{2}{*}{$\theta$} &   \multicolumn{4}{c}{$\#$ of change points ($\%$)} & \multirow{2}{*}{ARI} &   \multicolumn{4}{c}{$\#$ of change points ($\%$)} & \multirow{2}{*}{ARI}  \\ \cline{5-8}  \cline{10-13}
				& & &  &  $\leq 2$ & 3 & 4 & $\geq 5$  & & $\leq 2$ &3 &4 &$\geq 5$ &    \\ \hline
				\multirow{12}{*}{(i)} & \multirow{8}{*}{$250$} & \multirow{4}{*}{$0.3$} & $0.1$ & $0$ & $193$ & $6$ & $1$ & $0.932$ & $0$ & $200$ & $0$ & $0$ & $0.919$ \\ 
				&  &  & $0.15$ & $0$ & $200$ & $0$ & $0$ & $0.933$ & $0$ & $199$ & $1$ & $0$ & $0.914$ \\ 
				&  &  & $0.2$ & $0$ & $199$ & $1$ & $0$ & $0.934$ & $0$ & $199$ & $1$ & $0$ & $0.916$ \\ 
				&  &  & DCBS-Li & $0$ & $200$ & $0$ & $0$ & $0.999$ &  $0$ & $168$ & $9$ & $23$ & $0.923$ \\
				\cline{3-14}
				&  & \multirow{4}{*}{$0.6$} & $0.1$ & $14$ & $164$ & $21$ & $1$ & $0.876$ & $14$ & $178$ & $8$ & $0$ & $0.865$ \\ 
				&  & & $0.15$ & $12$ & $175$ & $13$ & $0$ & $0.887$ & $4$ & $195$ & $1$ & $0$ & $0.878$ \\ 
				&  &  & $0.2$ & $13$ & $176$ & $10$ & $1$ & $0.882$ & $5$ & $194$ & $1$ & $0$ & $0.887$ \\ 
				&  &  & DCBS-Li &  $182$ & $18$ & $0$ & $0$ & $0.633$ &  $0$ & $3$ & $0$ & $197$ & $0.477$ \\
				\cline{2-14}
				& \multirow{8}{*}{$500$} & \multirow{4}{*}{$0.3$} & $0.1$ & $0$ & $194$ & $6$ & $0$ & $0.937$ & $0$ & $198$ & $2$ & $0$ & $0.917$ \\ 
				&  &  & $0.15$ & $0$ & $198$ & $2$ & $0$ & $0.937$ & $0$ & $200$ & $0$ & $0$ & $0.913$ \\ 
				&  &  & $0.2$ & $0$ & $198$ & $2$ & $0$ & $0.936$ & $0$ & $200$ & $0$ & $0$ & $0.913$ \\ 
				&  &  & DCBS-Li &  $0$ & $200$ & $0$ & $0$ & $1.000$ &  $0$ & $176$ & $2$ & $22$ & $0.927$ \\
				\cline{3-14}
				&  & \multirow{4}{*}{$0.6$} & $0.1$ & $3$ & $170$ & $25$ & $2$ & $0.907$ & $1$ & $192$ & $7$ & $0$ & $0.894$ \\ 
				&  &  & $0.15$ & $1$ & $186$ & $13$ & $0$ & $0.915$ & $0$ & $197$ & $3$ & $0$ & $0.891$ \\ 
				&  &  & $0.2$ & $3$ & $187$ & $10$ & $0$ & $0.914$ & $0$ & $199$ & $1$ & $0$ & $0.893$   \\ 
				&  &  & DCBS-Li & $178$ & $22$ & $0$ & $0$ & $0.655$ & $0$ & $5$ & $2$ & $193$ & $0.438$ \\
				\hline 
				\hline 
				\multirow{12}{*}{(ii)} & \multirow{8}{*}{$250$} & \multirow{4}{*}{$0.3$} & $0.1$ & $0$ & $194$ & $6$ & $0$ & $0.936$ & $0$ & $197$ & $3$ & $0$ & $0.915$ \\ 
				&  &  & $0.15$ & $0$ & $196$ & $4$ & $0$ & $0.934$ & $0$ & $198$ & $2$ & $0$ & $0.917$ \\ 
				&  &  & $0.2$ & $0$ & $198$ & $2$ & $0$ & $0.934$ & $0$ & $199$ & $1$ & $0$ & $0.915$ \\ 
				&  &  & DCBS-Li &  $0$ & $200$ & $0$ & $0$ & $0.998$ & $0$ & $172$ & $10$ & $18$ & $0.938$ \\
				\cline{3-14}
				&  &\multirow{4}{*}{$0.6$} & $0.1$ & $27$ & $154$ & $18$ & $1$ & $0.869$ & $23$ & $168$ & $8$ & $1$ & $0.854$ \\ 
				&  &  & $0.15$ & $26$ & $165$ & $9$ & $0$ & $0.869$ & $10$ & $188$ & $2$ & $0$ & $0.879$ \\ 
				&  &  & $0.2$ & $26$ & $172$ & $2$ & $0$ & $0.874$ & $14$ & $185$ & $1$ & $0$ & $0.880$ \\ 
				&  &  & DCBS-Li & $147$ & $53$ & $0$ & $0$ & $0.742$ & $0$ & $3$ & $1$ & $196$ & $0.435$ \\ 
				\cline{2-14}
				& \multirow{8}{*}{$500$} & \multirow{4}{*}{$0.3$} & $0.1$ & $0$ & $195$ & $5$ & $0$ & $0.934$ & $0$ & $198$ & $2$ & $0$ & $0.921$ \\ 
				&  &  & $0.15$ & $0$ & $197$ & $3$ & $0$ & $0.934$ & $0$ & $199$ & $1$ & $0$ & $0.917$ \\ 
				&  &  & $0.2$ & $0$ & $197$ & $3$ & $0$ & $0.933$ & $0$ & $200$ & $0$ & $0$ & $0.917$ \\ 
				&  &  & DCBS-Li & $0$ & $200$ & $0$ & $0$ & $1.000$ &  $0$ & $175$ & $0$ & $25$ & $0.917$ \\
				\cline{3-14}
				&  & \multirow{4}{*}{$0.6$} & $0.1$ & $7$ & $171$ & $21$ & $1$ & $0.907$ & $0$ & $190$ & $10$ & $0$ & $0.891$ \\ 
				&  &  & $0.15$ & $8$ & $182$ & $10$ & $0$ & $0.908$ & $0$ & $199$ & $1$ & $0$ & $0.893$ \\ 
				&  &  & $0.2$ & $7$ & $188$ & $5$ & $0$ & $0.913$ & $0$ & $198$ & $2$ & $0$ & $0.895$ \\  
				&  &  & DCBS-Li & $187$ & $13$ & $0$ & $0$ & $0.554$ & $0$ & $1$ & $3$ & $196$ & $0.414$ \\
				\hline 
				\hline 
				\multirow{12}{*}{(iii)} & \multirow{8}{*}{$250$} & \multirow{4}{*}{$0.3$} & $0.1$ & $0$ & $183$ & $16$ & $1$ & $0.925$ & $0$ & $196$ & $4$ & $0$ & $0.908$ \\ 
				&  &  & $0.15$ & $0$ & $193$ & $7$ & $0$ & $0.928$ & $0$ & $200$ & $0$ & $0$ & $0.909$ \\ 
				&  &  & $0.2$ & $0$ & $199$ & $1$ & $0$ & $0.930$ & $0$ & $199$ & $1$ & $0$ & $0.905$ \\ 
				&  &  & DCBS-Li &  $0$ & $198$ & $2$ & $0$ & $0.998$ & $0$ & $162$ & $21$ & $17$ & $0.983$ \\ 
				\cline{3-14}
				&  & \multirow{4}{*}{$0.6$} & $0.1$ & $29$ & $140$ & $28$ & $3$ & $0.848$ & $41$ & $155$ & $4$ & $0$ & $0.812$ \\ 
				&  &  & $0.15$ & $26$ & $158$ & $15$ & $1$ & $0.864$ & $24$ & $172$ & $4$ & $0$ & $0.846$ \\ 
				&  &  & $0.2$ & $26$ & $161$ & $12$ & $1$ & $0.868$ & $22$ & $175$ & $3$ & $0$ & $0.848$ \\ 
				&  &  & DCBS-Li & $171$ & $28$ & $0$ & $1$ & $0.656$ &  $0$ & $14$ & $9$ & $177$ & $0.675$  \\ 
				\cline{2-14}
				&  \multirow{8}{*}{$500$} & \multirow{4}{*}{$0.3$} & $0.1$ & $0$ & $195$ & $5$ & $0$ & $0.937$ & $0$ & $196$ & $4$ & $0$ & $0.913$ \\ 
				&  &  & $0.15$ & $0$ & $195$ & $5$ & $0$ & $0.935$ & $0$ & $199$ & $1$ & $0$ & $0.914$ \\ 
				&  &  & $0.2$ & $0$ & $198$ & $2$ & $0$ & $0.937$ & $0$ & $200$ & $0$ & $0$ & $0.913$ \\ 
				&  &  & DCBS-Li & $0$ & $200$ & $0$ & $0$ & $1.000$ &  $0$ & $175$ & $19$ & $6$ & $0.977$ \\
				\cline{3-14}
				&  & \multirow{4}{*}{$0.6$} & $0.1$ & $9$ & $153$ & $35$ & $3$ & $0.895$ & $3$ & $192$ & $5$ & $0$ & $0.883$ \\ 
				&  &  & $0.15$ & $4$ & $174$ & $22$ & $0$ & $0.903$ & $0$ & $199$ & $1$ & $0$ & $0.890$ \\ 
				&  &  & $0.2$ & $4$ & $184$ & $12$ & $0$ & $0.909$ & $0$ & $200$ & $0$ & $0$ & $0.890$ \\
				&  &  & DCBS-Li &  $190$ & $10$ & $0$ & $0$ & $0.632$ &  $1$ & $9$ & $0$ & $190$ & $0.515$ \\ \hline   
			\end{tabular} 
		\end{footnotesize}
	\end{table}

	\clearpage
	\newpage

	
\end{document}